\documentclass[11pt]{amsart}

\usepackage[margin=2.7cm]{geometry}
\usepackage[parfill]{parskip}
\usepackage{graphicx}
\usepackage{float}
\usepackage{amsmath}
\usepackage{amssymb}
\usepackage{amsthm}
\usepackage{array}
\usepackage{mathrsfs}
\usepackage[all]{xy}
\usepackage{fancyhdr}
\usepackage{enumitem}
\usepackage{subfig}
\usepackage{tabularx}
\usepackage{psfrag}
\usepackage{epsf}
\usepackage{amscd, latexsym, hyperref, rlepsf}

\newcommand{\C}{\mathbb{C}}
\newcommand{\Q}{\mathbb{Q}}
\newcommand{\Z}{\mathbb{Z}}

\newcommand{\bdry}{\partial}
\newcommand{\isom}{\cong}

\newcommand{\inv}{^{-1}}

\newtheorem{theorem}{Theorem}[section]
\newtheorem{lemma}[theorem]{Lemma}
\newtheorem{prop}[theorem]{Proposition}
\newtheorem{cor}[theorem]{Corollary}

\newtheorem{question}[theorem]{Question}

\title{Symplectic fillings of Seifert fibered spaces}
\author{Laura Starkston}
\address{Department of Mathematics \\
The University of Texas \\
Austin, TX 78712, USA}
\email{lstarkston@math.utexas.edu}
\urladdr{http://www.ma.utexas.edu/$\sim$lstarkston}

\begin{document}

\begin{abstract}
We give finiteness results and some classifications up to diffeomorphism of minimal strong symplectic fillings of Seifert fibered spaces over $S^2$ satisfying certain conditions, with a fixed natural contact structure. In some cases we can prove that all symplectic fillings are obtained by rational blow-downs of a plumbing of spheres. In other cases, we produce new manifolds with convex symplectic boundary, thus yielding new cut-and-paste operations on symplectic manifolds containing certain configurations of symplectic spheres.
\end{abstract} 

\maketitle

\section{Introduction}
{

There are now a number of constructions of symplectic 4-manifolds coming from cutting out a neighborhood of a nice configuration of symplectic curves and replacing that neighborhood with a new piece. In the case that this piece is a rational homology ball, this technique is called a rational blow-down. The fact that some rational blow-downs can be done symplectically was first shown by Symington \cite{Symington1}, \cite{Symington2}. This is interesting both from the perspective of understanding the geography of symplectic 4-manifolds, and because of the information this yields about the Seiberg Witten invariants of the underlying smooth 4-manifold. A number of recent papers have explored when it is possible to replace a neighborhood of more general configurations of symplectic curves with Milnor fibers and symplectic rational homology balls (see \cite{BhupalStipsicz}, \cite{GayStipsicz1}, \cite{GayStipsicz2}, \cite{StipsiczSzaboWahl}).

The goal here is to understand not only when a neighborhood of a configuration of curves can be replaced by a specific symplectic or complex manifold, but what are all \emph{possible} symplectic fillings that can replace such a neighborhood of curves. Following a method of argument used by Lisca \cite{Lisca} and Bhupal and Stipsicz \cite{BhupalStipsicz}, we can produce a finite list of all possible diffeomorphism types that may replace a convex neighborhood of certain configurations of symplectic spheres, while preserving the existence of a symplectic structure. Namely, this is a classification of convex symplectic fillings of the contact boundary of a convex neighborhood of these spheres. 

Similar classifications have been made for lens spaces with their standard tight contact structures (Lisca \cite{Lisca}), other tight contact structures on $L(p,1)$ (Plamenevskaya and Van Horn-Morris \cite{PlamenevskayaVanHornMorris}), and some links of simple singularities (Ohta and Ono \cite{OhtaOno}). By contrast, there are now many results which show that there are contact manifolds with infinitely many non-diffeomorphic Stein or strong symplectic fillings satisfying various properties (see \cite{Smith}, \cite{OzbagciStipsicz}, \cite{OhtaOnoInfinite}, \cite{AEMS}, \cite{Akbulut}, \cite{AkbulutYasui}, \cite{AkhmedovOzbagci}). We obtain some finiteness results in this paper, but it remains an open question how to characterize contact three manifolds which have finitely many versus infinitely many non-diffeomorphic Stein or strong symplectic fillings.

The proof technique limits the configurations of curves for which we are able to solve this problem. The most obvious conditions necessary for these arguments are the following.
\begin{enumerate}
\item The curves are symplectic spheres that intersect each other $\omega$-orthogonally. (Note that it suffices to assume the spheres intersect positively and transversely since then they can be isotoped through symplectic spheres to be $\omega$-orthogonal \cite{GompfSymplGluing}).
\item  The spheres are configured in a star-shaped graph (there is a unique vertex of valence $>2$) with $k$ arms.
\item The decoration representing the self-intersection number on the central vertex is at most $-k-1$, and the decorations on the vertices in the arms are at most $-2$.
\end{enumerate}

Because these conditions ensure that the dual graph (in the sense of Stipsicz, Szab\'{o}, and Wahl \cite{StipsiczSzaboWahl}) is a star-shaped graph whose central vertex has positive decoration (see section \ref{dualgraph}), we will we will say configurations satisfying these conditions are \emph{dually positive}. 

A neighborhoood of such spheres has a handlebody diagram given in figure \ref{PlumbingSFS} where $e_0\leq -k-1$ and $b_i^j\leq -2$.
\begin{figure}
	\centering
	\subfloat[Handlebody diagram]{
	\label{PlumbingSFS}
	\psfragscanon
	\psfrag{e}{$e_0$}
	\psfrag{1}{$b_1^1$}
	\psfrag{2}{$b_{n_1}^1$}
	\psfrag{3}{$b_1^2$}
	\psfrag{4}{$b_{n_2}^2$}
	\psfrag{5}{$b_1^k$}
	\psfrag{6}{$b_{n_k}^k$}
	\includegraphics[scale=.5]{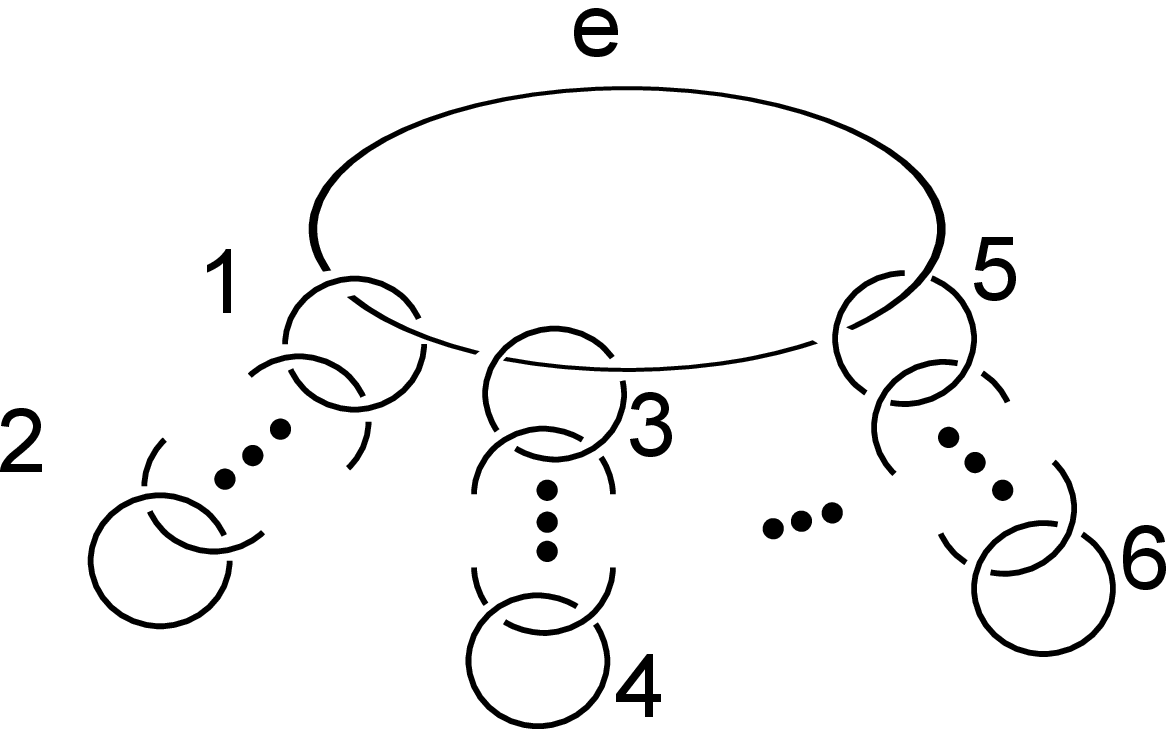}
	\psfragscanoff
	}
	\;\;\;\; \subfloat[Surgery diagram]{
	\label{SFS2}
	\psfragscanon
	\psfrag{1}{$e_0$}
	\psfrag{2}{$r_1$}
	\psfrag{3}{$r_2$}
	\psfrag{4}{$r_k$}
	\includegraphics[scale=.28]{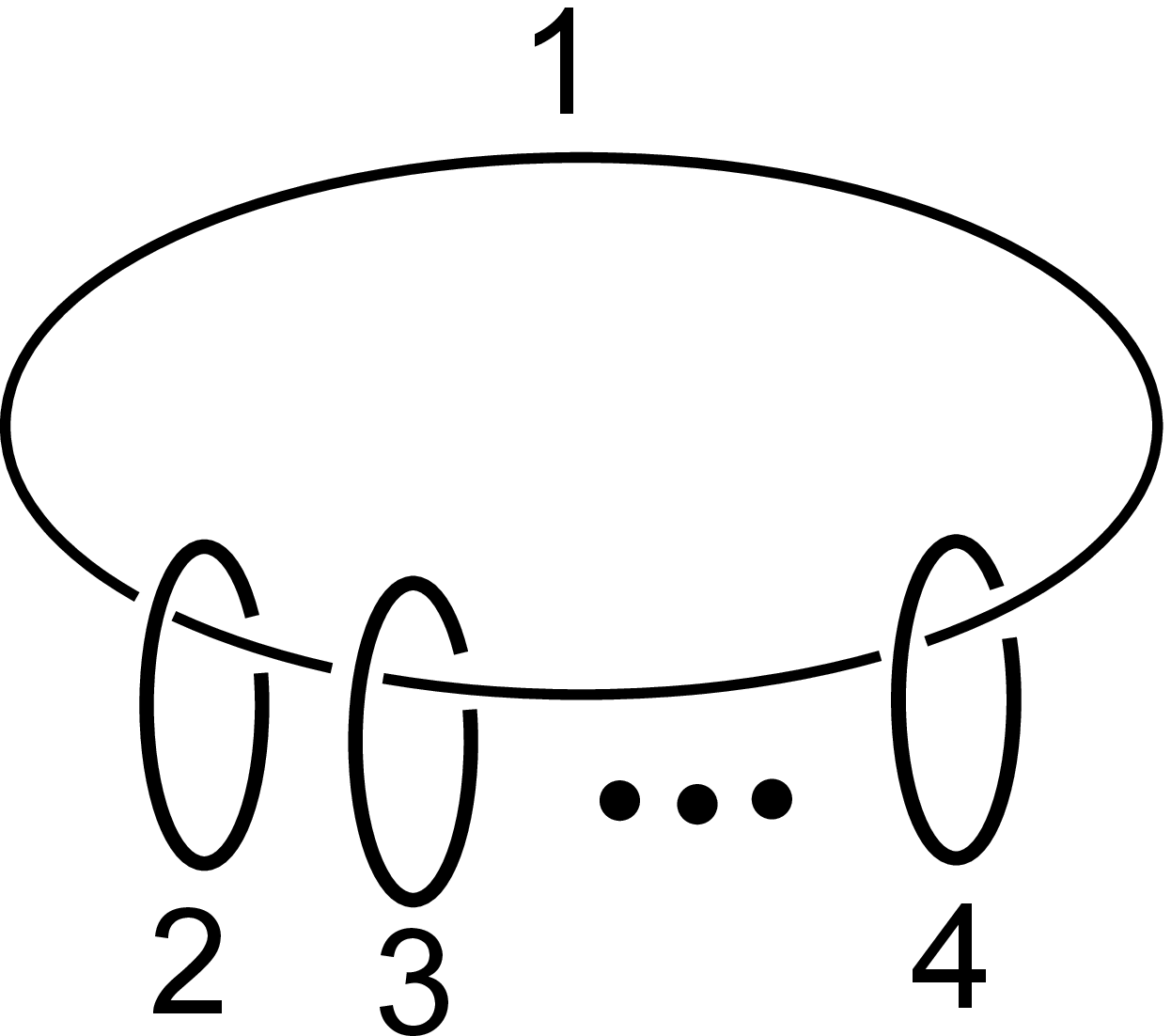}
	\psfragscanoff
	}
\caption{Handlebody diagram for a neighborhood of spheres, which can also be seen as a surgery diagram for the boundary of such a neighborhood which is a Seifert fibered space over $S^2$.}
\label{SeifertFiberedDiagrams}
\end{figure}
Note that the boundary of such a neighborhood of spheres is a Seifert fibered 3-manifold over $S^2$, with surgery diagram given by figure \ref{PlumbingSFS}. This surgery diagram is equivalent to that in figure \ref{SFS2} when
$$r_i=b_1^i-\frac{1}{b_2^i-\frac{1}{\ddots-\frac{1}{b_{n_i}^i}}}.$$
In general, a choice of $e_0\in \Z$ and $r_1,\cdots, r_k\in \Q$ such that $r_i<-1$ for all $i$, uniquely determines a Seifert fibered space by surgery as in figure \ref{SFS2}. We denote the manifold described by the surgery diagram by $Y(e_0;r_1,\cdots, r_k)$.

 We will say a Seifert fibered space is \emph{dually positive} if  it arises as the boundary of a dually positive plumbing of spheres. Any neighborhood of a dually positive configuration of symplectic spheres contains a symplectically convex neighborhood of the spheres (by \cite{GayStipsicz2} or \cite{GayMark}- see section \ref{cutpaste}), whose boundary inherits a fixed contact structure. As this contact structure arises as the boundary of a symplectic plumbing of spheres, we will refer to this contact structure as $\xi_{pl}$ (we suppress the Seifert invariants for simplicity). Therefore the results here provide a method for the classification of convex symplectic fillings of dually positive Seifert fibered spaces, with the contact structures $\xi_{pl}$.
 
 In the case of dually positive Seifert fibered spaces, the result of Mark and Gay will show that these contact structures which are induced on the convex boundary are supported by planar open books. These open book decompositions make it clear that the contact structures $\xi_{pl}$ coincide with those studied much earlier by Sch\"{o}nenberger \cite{Schonenberger} and Etg\"{u} and Ozbagci \cite{EtguOzbagci} and shown to be horizontal contact structures in \cite{Ozbagci}. By a result of Park and Stipsicz in \cite{ParkStipsicz}, the contact structure $\xi_{pl}$ is contactomorphic to the contact structure induced by the complex tangencies on the link of a complex singularity. This contact structure is sometimes referred to as the canonical contact structure in the literature.

In specific cases, we are able to state explicit classifications of the diffeomorphism types of possible symplectic fillings. Under less specific assumptions, we can find some loose upper bounds to obtain the following finiteness result.

\begin{theorem} 
If $Y$ is a dually positive Seifert fibered space with $k$ singular fibers and either $k\in \{3,4,5\}$ or $e_0\leq-k-3$, then $(Y,\xi_{pl})$ has finitely many minimal strong symplectic fillings up to diffeomorphism. If $Y$ is any dually positive Seifert fibered space, then the minimal strong symplectic fillings of $(Y,\xi_{pl})$ realize only finitely many values of Euler characteristic and signature. Loose upper bounds in both cases can be determined from the Seifert invariants.
\label{thm:finite}
\end{theorem}

Note that the fact that the minimal symplectic fillings realize finitely many values of Euler characteristic and signature follows from the fact that these contact structures are supported by planar open book decompositions by a result of Kaloti \cite{Kaloti}. It is frequently possible to obtain tight upper bounds on the number of diffeomorphism types or the values of Euler characteristic and signature in specific examples, by analyzing the dual graph associated to the Seifert fibered space as will be described in section \ref{mainargument}.

In some simple families of examples, we explicitly work out the possible diffeomorphism types of the minimal symplectic fillings. In some cases, all fillings are related to the original plumbing of spheres by a rational blow-down of a subgraph of the original configuration. 

\begin{theorem}
If $(Y,\xi_{pl})$ is the boundary of a dually positive plumbing of spheres, where the star-shaped graph has exactly three arms, the central vertex has self-intersection coefficient $-4$, and
\begin{enumerate}
\item each arm has arbitrary length, but each sphere in any arm has self-intersection coefficient $-2$\\ OR
\item each arm has length one, and each sphere in each arm has self-intersection coefficient strictly less than $-4$
\end{enumerate}
then $(Y,\xi_{pl})$ has exactly two minimal strong symplectic fillings up to diffeomorphism, given by the original symplectic plumbing and the manifold obtained by rationally blowing down the central $-4$ sphere.
\label{thm:easyexamples}
\end{theorem}

In the previous cases the central vertex had the largest possible coefficient to satisfy the dually positive condition. When the coefficient on the central vertex $e_0\leq -k-3$ where $k$ is the number of arms in the graph, we can prove a similar theorem for any value of $k\geq 3$.

\begin{theorem}
If $(Y,\xi_{pl})$ is the boundary of a dually positive plumbing of spheres with $k$ arms, where the coefficient on the central vertex is $e_0\leq -k-3$, and the coefficients on the spheres in the arms are all $-2$, then all diffeomorphism types minimal strong symplectic fillings are obtained from the original plumbing of spheres by a rational blow-down of the central vertex sphere together with $-e_0-4$ spheres of square $-2$ in one of the arms.
\label{thm:e0neg}
\end{theorem}
In special cases, as determined in \cite{StipsiczSzaboWahl}, the entire star-shaped graph can be rationally blown down. The simplest example which is dually positive is shown in figure \ref{fig4333}. We can show that the expected rational blow-downs provide a complete classification for this graph.
\begin{theorem}
If $(Y,\xi_{pl})$ is the boundary of a dually positive plumbing of spheres in the configuration of figure \ref{fig4333}, there are exactly three diffeomorphism types of minimal strong symplectic fillings of $(Y,\xi_{pl})$: the original plumbing of spheres, the rational blowdown of the central $-4$ sphere, and the rational blowdown of the entire configuration.
\label{thm:4333}
\end{theorem}

\begin{figure}
\begin{center}
\psfragscanon
\psfrag{1}{$-3$}
\psfrag{2}{$-4$}
\psfrag{3}{$-3$}
\psfrag{4}{$-3$}
\includegraphics[scale=.3]{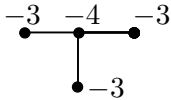}
\psfragscanoff
\end{center}
\caption{Graph referred to in theorem \ref{thm:4333}.}
\label{fig4333}
\end{figure}

The configuration in the previous theorem generalizes to the family $\mathcal{W}_{p,q,r}$ studied in \cite{StipsiczSzaboWahl}. We classify the symplectic fillings corresponding to this family of graphs in section \ref{rationalclassifications}. In this case, all convex fillings are obtained by a rational blow-down of a subgraph or the entire graph.

However, there are new manifolds arising as possible fillings that are not obtained by a rational blowdown of a subgraph. Note that these replacements still have smaller Euler characteristic than the original plumbing.

\begin{theorem}
If $(Y,\xi_{pl})$ is the boundary of a dually positive plumbing of spheres, with $k=4$ or $k=5$ arms, whose central vertex has coefficient $-k-1$, and all the vertices in the arms have coefficient $-2$ and each arm has length at least $k-3$, then all possible diffeomorphism types of symplectic fillings determined by this process (listed as handlebody diagrams in section \ref{newexamples}) are realized as strong symplectic fillings of $(Y,\xi_{pl})$.
\label{thm:newfillings}
\end{theorem}
Some of these diffeomorphism types cannot be obtained by a rational blowdown, or by replacing a neighborhood of a linear plumbing of spheres with a different symplectic filling of its lens space boundary. Namely, these are genuinely new diffeomorphism types that may be convex symplectic fillings. Note, if the arms have length less than $k-3$, some of the diffeomorphism types cannot be realized, but the classification can still be made.

The main argument that provides the set-up for all these results will be given in section \ref{mainargument}, with the proof of theorem \ref{thm:finite} at the end of that section. For concreteness, we follow this in section \ref{simplestexamples} by a detailed proof of the classification of fillings of a simple family of examples, proving part (1) of theorem \ref{thm:easyexamples}. We follow this by classification results that can be stated in terms of rational blow-downs in section \ref{rationalclassifications}, proving theorem \ref{thm:e0neg}, part (2) of theorem \ref{thm:easyexamples},  theorem \ref{thm:4333}, and its generalization to the family $\mathcal{W}_{p,q,r}$. In section \ref{newexamples}, we give explicit descriptions of the new 4-manifolds which strongly symplectically fill the dually positive Seifert fibered spaces in theorem \ref{thm:newfillings}. Section 6 contains some further observations and questions.

\textbf{Acknowledgements:} I am incredibly grateful to my advisor Robert Gompf for help, advice, and support throughout this project. Many thanks for comments from Heesang Park and Burak Ozbagci and for helpful conversations and feedback from John Etnyre, Matt Hedden, Amey Kaloti, Jeremy Van Horn-Morris, and especially to Tom Mark and Cagri Karakurt for patiently explaining some important results needed for this work. This material is based upon work supported by the National Science Foundation under Grant No. DGE-1110007.

}

\section{Upper bounds on diffeomorphism types of strong symplectic fillings}
{
\label{mainargument}
The main argument here gives upper bounds (in terms of explicit diffeomorphism types) of strong symplectic fillings of dually positive Seifert fibered spaces with the contact structure $\xi_{pl}$. First, using a construction of Stipsicz, Szab\'{o} and Wahl \cite{StipsiczSzaboWahl}, we will build the symplectic plumbing of spheres inside a closed symplectic manifold, such that the complement is also a symplectic plumbing of spheres, now with concave boundary. The dually positive condition will allow us to ensure that the concave piece contains a sphere of self-intersection number $+1$. Then we cut out the original plumbing of spheres and glue in an arbitrary convex symplectic filling to form a closed symplectic manifold which still contains a sphere of self-intersection number $+1$ (analogous to Lisca's method to classify symplectic fillings of $(L(p,q), \xi_{std})$). A theorem of McDuff implies that this closed manifold is a symplectic blow-up of $\C P^2$ with its standard Kahler structure. To identify the topology of the unknown convex symplectic filling, it suffices to understand how the concave cap can symplectically embed into the blow-up of $\C P^2$, since the convex filling must be its complement. Using arguments of Lisca \cite{Lisca}, we obtain homological restrictions coming from the adjunction formula and intersection information of the spheres. Under certain conditions, we can show that the homology classes these spheres represent, uniquely determines a symplectic embedding of their neighborhood into a blow up of $\C P^2$. By deleting the possible embeddings of the concave cap from blow-ups of $\C P^2$, we obtain the diffeomorphism types of all possible convex symplectic fillings of the given dually positive Seifert fibered space with contact structure $\xi_{pl}$.

\subsection{The dual graph construction}
{
\label{dualgraph}
First we describe the dual graph construction of Stipsicz, Szab\'{o}, and Wahl \cite{StipsiczSzaboWahl}, which provides a symplectic embedding of the neighborhood of dually positive spheres into a blow-up of $\C P^2$ whose complement is the concave cap we need. 

The construction of the dual graph begins by looking at a sphere bundle over a sphere, which has a standard symplectic structure. Each fiber will have self intersection number $0$, and will intersect each section of the bundle in a single point. The zero and infinity sections have self-intersection numbers $n$ and $-n$ respectively. It is convenient to have a handlebody diagram for the sphere bundle $B_n$, in which the $0$-section, $d$ distinct fibers, and the $\infty$-section are all visible. Such a diagram is given by figure \ref{figRuledSurface}.
\begin{figure}
\begin{center}
\psfragscanon
\psfrag{1}{$\cup d$ 3-handles, 1 4-handle}
\psfrag{2}{0}
\psfrag{3}{0}
\psfrag{4}{0}
\psfrag{5}{$-n$}
\psfrag{6}{$n$}
\psfrag{7}{$d$}
\includegraphics[scale=.4]{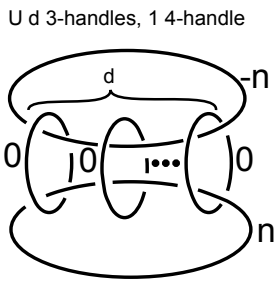}
\psfragscanoff
\end{center}
\caption{A handlebody diagram for the sphere bundle $B_n$. The relevant spheres (the zero section, infinity section, and $d$ fibers) are represented by the cores of the 2-handles together with the pushed in Seifert surfaces for the attaching circles.}
\label{figRuledSurface}
\end{figure}

To obtain the dual plumbing, we build the original plumbing inside a symplectic blow-up of $B_n$. We will allow blowups to be performed along the intersections of the various spheres in the picture. The proper transforms of these spheres will remain symplectic, and new exceptional spheres are symplectic submanifolds as well. At the beginning the spheres we keep track of are just the $0$-section, the $\infty$-section, and the $d$ fibers. As we blow-up, we include the new exceptional spheres in the picture. Figure \ref{figDualGraph} shows an example, keeping track of both the standard short-hand notation to keep track of these blowups, as well as the corresponding handlebody diagrams.
\begin{figure}
\begin{center}
\includegraphics[scale=.5]{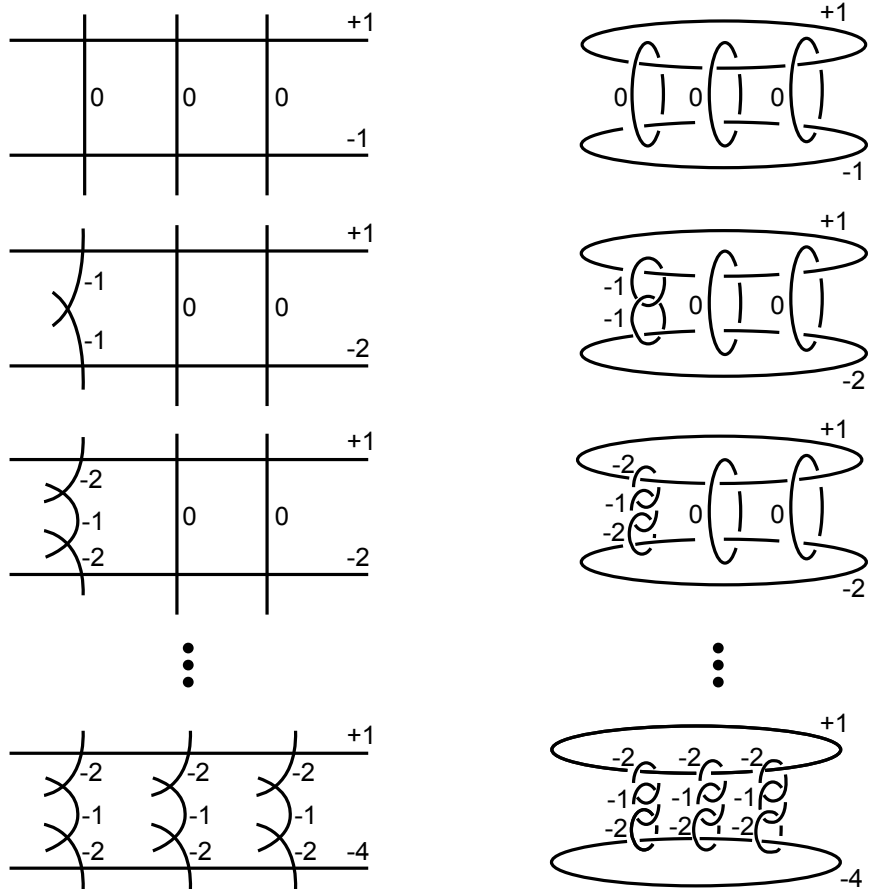}
\end{center}
\caption{A sequence of blow-ups, keeping track of the spheres whose neighborhoods are plumbings. In the last stage we can see the star-shaped plumbing with 4 vertices, central vertex labeled with $-4$, and the three arms are labelled with $-3$, and its dual plumbing which is also star-shaped with 4 vertices, but the central vertex is labeled with $+1$ and three arms labeled with $-2$. The two plumbings are glued together along equators of the regular fibers and equators of the $-1$ spheres on the exceptional fibers.}
\label{figDualGraph}
\end{figure}
Note that each exceptional fiber will contain at least one sphere of self-intersection $-1$, which is the exceptional sphere from the most recent blow-up of that fiber. We perform the blow-ups so that if we remove the vertices corresponding to these spheres in the resulting plumbing graph, we will get a two component graph, one of which is $\Gamma$ and the other is the dual graph $\Gamma'$.

To see that the plumbings coming from $\Gamma$ and $\Gamma'$ glue together to give the blow-up of our original sphere bundle, imagine cutting the blown-up sphere bundle as shown above, along an equator of each regular fiber, and along the equator of the last exceptional sphere in each exceptional fiber, so that these equators all match up smoothly to form a 3-manifold. Note that each sphere in either graph is a symplectic submanifold since it is either one of the distinguished sections, a fiber, an exceptional sphere, or the proper transform of one of these objects.

\subsubsection{Dually Positive Graphs}
{
\label{duallypositive}
Note that the conditions for a dually positive configuration of spheres, ensures that one can build a dual configuration which is a star-shaped graph whose central vertex has self-intersection number $+1$ as follows. Suppose the central vertex of the dually positive configuration has square $e_0\leq -k-1$ where $k$ is the number of arms in the original graph. Start with a sphere bundle with zero section of self-intersection number $-1$, and infinity section of self-intersection number $+1$, and keep track of $d=-e_0-1$ fibers (of self-intersection number 0). Blow up once at the intersections of each of the $-1-e_0$ fibers with the zero section. The proper transform of the zero section now has self-intersection number $e_0$, the exceptional spheres and proper transforms of the fibers have self-intersection number $-1$, and the infinity section is unchanged so it still has self-intersection number $+1$. Next, in $k$ of these singular fibers, blow-up at the intersection of the new exceptional sphere with the proper transforms of the original fiber. By continuing to blow up at points where an exceptional sphere of square $-1$ intersects an adjacent sphere, it is possible to build the dually positive graph emanating from the zero section and ending just before the most recent exceptional $-1$ spheres in each fiber, without ever blowing up at a point on the infinity section. Therefore the dual graph has central vertex with coefficient $+1$, and $-1-e_0$ arms, whose vertices have all negative coefficients.

Note that when $e_0<-k-1$, the dual graph contains $d=-e_0-1>k$ arms, but only $k$ of them are used to construct the singular fibers. Each of the remaining $-e_0-1-k$ arms shows up in the dual graphs construction as two $-1$ spheres, one of which is adjacent to the proper transform of the zero section and will be cut in half to split the graph from the dual graph. The remaining $-1$ sphere (the proper transform of a regular fiber after one blow-up at its intersection with the $0$-section) persists in the dual graph as a short arm.
}

}

\subsection{Cut and paste for symplectic manifolds}
{
\label{cutpaste}
The easiest way to ensure one can glue together two symplectic manifolds with orientation reversing diffeomorphic boundary, is by showing that one piece has convex boundary and the other has concave boundary. A concave or convex boundary inherits a contact structure, and a concave boundary will glue to a contactomorphic convex boundary to give a closed symplectic manifold (after possibly rescaling the symplectic form on one piece).  Therefore we would like to show a neighborhood of our dually positive configuration of spheres has convex boundary. In some cases, we will need to recognize the contact structure induced on the boundary in terms of an open book decomposition. For these results we use a theorem of Gay and Mark.

Their set up starts with a configuration of symplectic surfaces $\mathcal{C}=C_1\cup \cdots \cup C_n$ intersecting $\omega$-orthogonally according to a negative definite graph $\Gamma$ with no edges from a vertex to itself. For each vertex $v_j$, let $s_j$ be the sum of the valence of that vertex with the self-intersection number of the corresponding symplectic surface. Assume $s_j\leq 0$ for all vertices $v_j$ (a.k.a. no bad vertices). Let $\Sigma$ be the surface obtained from connect summing $|s_j|$ copies of $D^2$ to each $C_j$ and then connect summing these surfaces together according to the graph. Let $\{c_1,\cdots , c_k\}$ be simple closed curves, with one around each connected sum neck, and $\tau$ the product of right handed Dehn twists around $c_1,\cdots , c_k$.
 
 \begin{theorem}[Gay and Mark \cite{GayMark} Theorem 1.1] Any neighborhood of $\mathcal{C}$ contains a neighborhood $(Z,\eta)$ of $\mathcal{C}$ with strongly convex boundary, that admits a Lefschetz fibration $\pi: Z\to D^2$ having regular fiber $\Sigma$ and exactly one singular fiber $\Sigma_0=\pi\inv(0)$. The vanishing cycles are $c_1,\cdots , c_k$ and $\mathcal{C}$ is the union of the closed components of $\Sigma_0$. The induced contact structure $\xi$ on $\bdry Z$ is supported by the induced open book $(\Sigma,\tau)$.
 \label{thmplumbingconvex}
 \end{theorem}
 
For dually positive configurations of spheres, the hypothesis that $s_j\leq 0$ is satisfied and $\Sigma$ is a disk with holes.

Note that the fact that such neighborhoods of symplectic surfaces have convex neighborhoods was proven (under the weaker assumption that the plumbing graph be negative definite instead of the condition on the $s_j$) by Gay and Stipsicz in \cite{GayStipsicz2}. While this result is sufficient to allow us to make a gluing argument, we need the information about the open book decomposition to identify the contact structure in some of the more difficult cases (see section \ref{newexamples}).

Now we may choose an arbitrarily small convex neighborhood of the dually positive configuration of spheres to cut out, and then (after possibly rescaling the symplectic form) glue in any strong symplectic filling of the contact boundary, which has the canonical contact structure $\xi_{pl}$, to obtain a closed symplectic manifold containing a symplectic embedding of the concave cap.
}

\subsection{McDuff's classification for closed symplectic manifolds}
{
\label{McDuff}
As explained in section \ref{duallypositive}, our condition that the configuration of spheres be dually positive ensures that the concave cap coming from the dual graph construction contains a sphere of self-intersection number $+1$. This condition is useful due to the following classification theorem.

 \begin{theorem}[McDuff \cite{McDuff}]
 If $(V^4,C^2,\omega)$ is a minimal symplectic pair (namely $V\setminus C$ contains no exceptional curves), where $C$ is a rational curve with self-intersection $C\cdot C=p\geq 0$, then $(V,\omega)$ is symplectomorphic either to $\C P^2$ with its usual Kahler form or to a symplectic $S^2$ bundle over a compact surface $M$. Further, this symplectomorphism may be chosen so that it takes $C$ either to a complex line or quadric in $\C P^2$, or to a fiber of the $S^2$ bundle, or (if $M=S^2$) to a section of this bundle.
 \end{theorem}

In our case, this classification simplifies to a single manifold up to blowing up symplectically.

\begin{cor}
If $(V^4,C^2,\omega)$ is a minimal symplectic pair where $C$ is a 2-sphere of self-intersection number $+1$ then $(V,C,\omega)$ is symplectomorphic to $(\C P^2,\C P^1,\omega_{std})$.
\end{cor}

By gluing any strong symplectic filling of a dually positive Seifert fibered space with contact structure $\xi_{pl}$ to a neighborhood of its dual configuration, we obtain a symplectic pair $(V,C,\omega)$ satisfying all hypotheses of the theorem except minimality. After blowing down exceptional spheres in $V\setminus C$, it follows that such a convex filling embeds symplectically (up to rescaling the symplectic form) in a blow-up of $(\C P^2, \omega_{std})$ where the blow-ups are disjoint from the standard $\C P^1\subseteq \C P^2$. The complement of the embedded convex filling is symplectomorphic to the corresponding plumbing of spheres described by the dual graph, and $+1$ sphere corresponding to the central vertex of the dual graph is identified with $\C P^1$.
}

\subsection{Homological restrictions on embeddings of the cap}
{
\label{sec:homology}
Denote by $(X_M,\omega_M)$ the closed symplectic manifold $\C P^2\# M \overline{\C P^2}$ with symplectic form $\omega_M$ given by some symplectic blow up of the standard Kahler form on $\C P^2$. We would like to determine all possible symplectic embeddings of the positive dual graph plumbings corresponding to the concave cap. To understand possible embeddings, we use some homological restrictions.

First fix a standard orthogonal basis $(\ell, e_1,\cdots, e_M)$ for $H_2(X_M;\Z)$ where $\ell$ is represented by the complex projective line so $\ell^2=+1$, and the $e_i$ are represented by the exceptional spheres created in the blow-ups, so $e_i^2=-1$, and $\ell\cdot e_i=e_i\cdot e_j=0$ for $i\neq j$. Because these are represented by symplectic surfaces, we can determine how the first Chern class of $X_M$ evaluates on each of these homology classes via the adjunction formula:
$$\langle c_1(X_M),\ell\rangle = \ell^2+2=3$$
$$\langle c_1(X_M),e_i\rangle = e_i^2+2 = 1$$

Now we would like to analyze what the spheres in the plumbing for the cap could represent in $H_2(X_M;\Z)$ in terms of this basis. We will refer to the embedded sphere representing the central vertex as $C_0$, the symplectic spheres representing the vertices adjacent to the central vertex as $C_1,\cdots, C_d$, and the symplectic spheres representing the other vertices in the dual graph by $C_{d+1}, \cdots, C_m$. We know that these spheres are also symplectic, so they must also satisfy the adjuction formula:
$$\langle c_1(X_M),[C_j]\rangle = [C_j]^2+2$$
Furthermore, we know that the sphere $C_0$ which has self-intersection number $+1$, is sent to the complex projective line so $[C_0]=\ell$. The intersection data implies spheres whose vertices are joined by an edge have homological intersection number $+1$, other distinct spheres have homological intersection number $0$, and the square of the homology class represented by each sphere is given by the decoration on the graph (which is negative for all but $C_0$).

Now suppose that for $j\in \{1,\cdots, m\}$ 
$$[C_j]=a_0^j\ell +\sum_{i=1}^M a_i^je_i$$
For $j\in \{1,\cdots, d\}$ we have $1=[C_j]\cdot [C_0]=[C_j]\cdot \ell$, so $a_0^j=1$. Using the adjunction formula, and our knowledge of how $c_1(X_M)$ evaluates on the standard basis we get the following formula for the coefficients $a_i^j$:
$$3+\sum_{i=1}^M a_i^j = 1-\sum_{i=1}^M (a_i^j)^2 +2$$
so
$$\sum_{i=1}^M(a_i^j)^2+a_i^j = 0$$
Note that since $a_i^j$ are integers, we have $(a_i^j)^2+a_i^j\geq 0$ with equality if and only if $a_i^j\in \{0,-1\}$. Therefore $a_i^j\in\{0,-1\}$ for all $i\in \{1,\cdots, M\}$, $j\in \{1,\cdots, d\}$. Furthermore since $-n_j=[C_j]^2=1-\sum_{i=1}^M (a_i^j)^2$, there are precisely $n_j+1$ values of $i$ for which $a_i^j$ is $-1$. Thus we can write for $j\in\{1,\cdots, d\}$:
$$[C_j] = \ell - e_{i_1^j}-\cdots -e_{i_{n_j+1}^j}.$$
We have further data given by the fact that $[C_j]\cdot [C_{j'}]=0$ when $j\neq j'\in \{1,\cdots , d\}$:
$$0=(\ell - e_{i_1^j}-\cdots -e_{i_{n_j+1}^j})\cdot (\ell - e_{i_1^{j'}}-\cdots -e_{i_{n_{j'}+1}^{j'}})=1-|\{i_1^j, \cdots , i_{n_j+1}^j\} \cap \{ i_1^{j'}, \cdots, i_{n_{j'}+1}^{j'}\}|$$
We conclude
\begin{lemma} If $C_j$ and $C_j'$ are distinct symplectic spheres in the dual graph configuration which are adjacent to the central vertex sphere, then $[C_j]$ and $[C_{j'}]$ share exactly one common basis element $e_i$ with coefficient $-1$. \label{lem:homologyint} \end{lemma}

If the graph we are considering has $k$ arms, corresponding to the $k$ singular fibers in the Seifert fibered space, then the coefficient on the central vertex of the graph, $e_0$ determines the relationship between $k$ and $d$. As discussed in section \ref{duallypositive}, the number of arms in the dual graph is $d=-e_0-1$, and the dually positive assumption implies $k\leq -e_0-1$. When $d=-e_0-1$ is strictly larger than $k$, there are $d-k$ additional short arms each made up of a single symplectic sphere of self-intersection number $-1$. Therefore $[C_j]=\ell-e_{i_1^j}-e_{i_2^j}$ for $j\in \{k+1,\cdots , d\}$.

\begin{lemma}
If $d=-e_0-1>k$ then the symplectic spheres $C_0,C_1,\cdots , C_d$ represent one of the following sets of homology classes in terms of the standard basis for $H_2(\C P^2\# M\overline{\C P^2})$ (up to relabelling).
$$\begin{array}{rcl|rcl}
\left[C_0\right] &=& \ell & \left[C_0\right] &=& \ell \\
\left[C_1\right] &=& \ell-e_1-e_\cdot-\cdots -e_\cdot & \left[C_1\right] &=& \ell-e_2-\cdots - e_k-e_{k+1}-\cdots -e_d -e_{\cdot}-\cdots -e_{\cdot}\\
\left[C_2\right] &=& \ell-e_1-e_\cdot-\cdots -e_\cdot & \left[C_2\right] &=& \ell-e_1-e_2-e_\cdot-\cdots -e_\cdot \\
&\vdots & & \vdots & \\
\left[C_k\right] &=& \ell-e_1-e_\cdot-\cdots -e_\cdot & \left[C_k\right] &=& \ell-e_1-e_k-e_\cdot-\cdots -e_\cdot \\
\left[C_{k+1}\right] &=& \ell-e_1-e_{k+1} & \left[C_{k+1}\right] &=& \ell-e_1-e_{k+1} \\
& \vdots & & \vdots & \\
\left[C_d\right] &=& \ell-e_1-e_d & \left[C_d\right] &=& \ell-e_1-e_d \\
\end{array}$$
Here $e_\cdot$ indicates that there can be additional distinct $e_i$'s with coefficient $-1$ in these homology classes if the corresponding square is sufficiently negative. They should all be distinct from each other and distinct from all labelled $e_i$'s. When $d=k+1$ there are additional possibilities given as follows where $1<j<k$ (by a symmetry we may actually assume $j\leq k/2$).
$$\begin{array}{rcl}
\left[C_0\right] &=& \ell\\
\left[C_1\right] &=& \ell-e_1-e_{i(1,j+1)}-e_{i(1,j+2)}\cdots -e_{i(1,k)}-e_\cdot -\cdots -e_\cdot \\
&\vdots & \\
\left[C_j\right] &=& \ell-e_1-e_{i(j,j+1)}-e_{i(j,j+2)}\cdots - e_{i(j,k)}-e_\cdot -\cdots -e_\cdot\\
\left[C_{j+1}\right] &=& \ell-e_2-e_{i(1,j+1)}-e_{i(2,j+1)}-\cdots -e_{i(j,j+1)}-e_\cdot-\cdots -e_\cdot \\
&\vdots & \\
\left[C_k\right] &=& \ell-e_2-e_{i(1,k)}-e_{i(2,k)}-\cdots -e_{i(j,k)}-e_\cdot -\cdots -e_\cdot \\
\left[C_{k+1}\right] &=& \ell-e_1-e_2\\
\end{array}$$
Here $i(a,b)$ are distinct for distinct pairs $(a,b)$, and are distinct from $1,2$.
\label{lem:e0neg}
\end{lemma}

\begin{proof}
First notice that there must be a common element $e_i$ with coefficient $-1$ in all the classes $[C_{k+1}],\cdots, [C_d]$. This is trivial in the case that $k+1=d$, and follows immediately from lemma \ref{lem:homologyint} when $k+2=d$. When $d>k+2$, lemma \ref{lem:homologyint} implies that $[C_{k+1}]$ and $[C_{k+2}]$ share a unique element with coefficient $-1$ so without loss of generality $[C_{k+1}]=\ell-e_1-e_2$ and $[C_{k+2}]=\ell-e_1-e_3$. If $[C_{k+3}]$ does not have $-1$ coefficient for $e_1$ then lemma \ref{lem:homologyint} implies $[C_{k+3}]=\ell-e_2-e_3$, but then there is no way that $[C_1]$ can have $e_i$'s with $-1$ coefficient for exactly one of $\{1,2\}$, exactly one of $\{1,3\}$, and exactly one of $\{2,3\}$ which is a contradiction. Since $[C_j]$ and $[C_{j'}]$ share only one element $e_i$ with coefficient $-1$, we find that $[C_{k+1}]=\ell-e_1-e_{k+1}, \cdots , [C_d]=\ell-e_1-e_d$ for $e_{k+1},\cdots , e_d$ all distinct.

Now consider the homology classes $[C_1],\cdots, [C_k]$. Each such class must either have $-1$ coefficient for $e_1$ or $-1$ coefficient for all of the classes $e_{k+1},\cdots , e_d$ (not both). If $d>k+1$ then there can be at most one of the spheres $C_1,\cdots , C_k$, whose homology class has coefficient $-1$ for all the classes $e_{k+1},\cdots, e_d$ since no two $[C_j]$ can share more than one $e_i$ with coefficient $-1$. This proves the first part of the lemma.

When $d=k+1$ and $[C_{k+1}]=\ell -e_1-e_2$, some of the classes $[C_1],\cdots , [C_k]$ must have coefficient $-1$ for $e_1$ and the rest must have coefficient $-1$ for $e_2$. Without loss of generality we assume the first $j$ have $-1$ coefficient for $e_1$, and the rest have $-1$ coefficient for $e_2$. Then for each pair $(a,b)\in \{1,\cdots , j\}\times \{j+1,\cdots , k\}$ we must add another $e_{i(a,b)}$ which occurs with $-1$ coefficient in $[C_a]$ and $[C_b]$. If $i(a,b)=i(a',b')$ for $(a,b)\neq (a',b')$ then either $a\neq a'$ so $[C_a]$ and $[C_{a'}]$ both have coefficient $-1$ for both $e_1$ and $e_{i(a,b)}$ or $b\neq b'$ so $[C_b]$ and $[C_{b'}]$ both have coefficient $-1$ for $e_2$ and $e_{i(a,b)}$, but homology classes of distinct pairs of spheres can only share one common element with coefficient $-1$.
\end{proof}

For $j\in \{d+1,\cdots, m\}$, we know that $0=[C_j]\cdot [C_0] = [C_j]\cdot \ell$. Therefore $a_0^j=0$ for all $j\in \{d+1,\cdots, m\}$. In this case the adjuction formula yields the following formula:
$$\sum_{i=1}^M a_i^j = -\sum_{i=1}^M (a_i^j)^2 +2$$
so
$$\sum_{i=1}^M(a_i^j)^2+a_i^j = 2$$
Thus, all but one of the $a_i^j$'s is either $0$ or $1$, and exactly one $a_i^j$ is either $1$ or $-2$ for for each $j\in \{d+1,\cdots, m\}$. An inductive argument of Lisca \cite[Proposition 4.4]{Lisca} implies that $a_i^j$ can never be equal to $-2$, so there is always a unique $a_i^j$ equal to $1$. Note that Lisca's statement refers to linear graphs of symplectic spheres embedded in a blow-up of $\C P^2$, but each arm of the star-shaped graph (starting at the central vertex) is a linear graph satisfying the necessary hypotheses.

In conclusion the symplectic spheres in the dual graph represent homology classes of the following form.
\begin{eqnarray*}
\left[C_0\right]&=& \ell\\
\left[C_1\right]&=& \ell -e_{i^1_1}-\cdots -e_{i^1_{n_1+1}}\\
&\vdots &\\
\left[C_k\right]&=& \ell -e_{i^k_1}-\cdots -e_{i^k_{n_k+1}}\\
\left[C_{k+1}\right]&=& \ell-e_{i^{k+1}_1}-e_{i^{k+1}_2}\\
&\vdots & \\
\left[C_d\right]&=& \ell-e_{i^d_1}-e_{i^d_2}\\
\left[C_{d+1}\right]&=& e_{i^{k+1}_0} -e_{i^{k+1}_1}-\cdots -e_{i^{k+1}_{n_{k+1}-1}}\\
&\vdots &\\
\left[C_{m}\right]&=& e_{i^{m}_0} -e_{i^{m}_1}-\cdots -e_{i^{m}_{n_{m}-1}}\\
\end{eqnarray*}

Here $i_h^j\neq i_{h'}^j$ when $h\neq h'$, and the values $n_j$ are determined by the square of $[C_j]$. Up to relabeling the $e_i$, the only remaining question is when $i^j_h$ can coincide with $i^{j'}_{h'}$ for $j\neq j'$. Additional restrictions on $[C_1],\cdots, [C_d]$ are given by lemmas \ref{lem:homologyint} and \ref{lem:e0neg}. Other intersection restrictions imply other relations between the sets $\{i_0^j, \cdots , i_{n_j}^j\}$ for different values of $j\in \{1, \cdots , m\}$, which can be exploited in specific examples.
}

\subsection{Translating homological restrictions into embeddings}
{
Given a finite list of homology classes a symplectic embedding of the spheres of the dual graph into $\C P^2\#M\overline{\C P^2}$ can represent, we would like to say that there are finitely many symplectic embeddings of the dual graph up to isotopy. There are two main steps to this process. The first is to follow the arguments of Lisca \cite{Lisca} to carefully blow-down to $\C P^2$, while keeping track of how this affects the embedded spheres of the dual graph. The second step is to analyze this blown-down embedding, and to try to understand the isotopy classes of a regular neighborhood of the blown-down dual graph in $\C P^2$.

\subsubsection{Blowing down $(\C P^2\#M\overline{\C P^2},\text{dual graph})$}
{
\label{sec:blowingdown}
The following theorem was used by Lisca to solve this part of the problem when classifying symplectic fillings of lens spaces. In that case, the plumbing of spheres providing the concave cap is linear.

\begin{theorem}[Lisca \cite{Lisca} Theorem 4.2]
Let $\omega_M$ be a symplectic form on $\C P^2 \# M\overline{\C P^2}$ obtained from the standard Kahler form by symplectic blow-ups. Let $\Gamma=C_0\cup \cdots \cup C_j$ be a union of $\omega_M$-orthogonal symplectic spheres, in the configuration of a linear graph, with self-intersection numbers $(1,1-b_1,-b_2,\cdots, -b_j)$, such that $C_0$ is a complex line. Then there is a sequence of symplectic blow-downs
$$(\C P^2\# M \overline{\C P^2}, \omega_M)\to (\C P^2\# (M-1)\overline{\C P^2},\omega_{M-1})\to \cdots \to (\C P^2,\omega_0)$$
with $\omega_0$ diffeomorphic to the standard Kahler form and such that $\Gamma$ descends to two $\omega_M$ orthogonal symplectic spheres, each of self-intersection number $1$.
\end{theorem}

In our situation, the dual graph is a star-shaped instead of linear. However, a star-shaped graph is simply the union of its $k$ arms, each of which is a linear graph emanating from the central vertex. Therefore Lisca's theorem applies to each of the arms of the star-shaped dual configuration. We would like to keep track of all of these arms at once as we blow-down the manifold. Though we do not need a new argument here, a summary of Lisca's proof is included here to make it clear how it applies in the star-shaped case. 

In our case, the dual configuration of spheres $\Gamma=C_0\cup \cdots \cup C_m$ (in the shape of a star-shaped graph), is assumed to be symplectically embedded in $(\C P^2\# M \overline{\C P^2},\omega_M)$. Here $C_0$ is identified with the complex projective line (by McDuff's theorem). First, choose an almost complex structure tamed by $\omega_M$, for which the spheres $C_0, \cdots , C_m$ are all $J$-holomorphic. This allows us to have more control over intersections of spheres with any $J$-holomorphic sphere we blow down. 

Relying on analysis of $J$-holomorphic curves by McDuff, Lisca proves a lemma (Lemma 4.5 in \cite{Lisca}), which says that as long as $M>0$, there exists a $J$-holomorphic sphere $\Sigma$ such that $[\Sigma]\cdot [C_0]=0$ and $[\Sigma]^2=-1$. Furthermore, we can find $\Sigma$ disjoint from $\Gamma$ if and only if there is a symplectic sphere $S$ of square $-1$ such that $[S]\cdot [C_j]=0$ for $j=0,1,\cdots, m$. Note that Lisca stated this in the case that $\Gamma$ is a linear plumbing, but the proof is unchanged for any configuration of symplectic spheres intersecting $\omega$-orthogonally. Therefore, it is possible to blow down $J$-holomorphic spheres $\Sigma$ until $X_M$ is blown down to $\C P^2$. Using the lemma, we can first blow down any $\Sigma$'s disjoint from $\Gamma$, until there exists $\Sigma$ for which $[\Sigma]\cdot [C_j]\neq 0$ for some $j$. Because $C_0,\cdots, C_m$ are $J$-holomorphic, $\Sigma$ must intersect them non-negatively. In our standard basis for $H_2(X_M;\Z)$ as in the previous section, the fact that $[\Sigma]\cdot [C_0]=0$ and $[\Sigma]^2=-1$ implies $[\Sigma]=\pm e_i$. Given the analysis in the previous section of how to write $[C_j]$ in terms of this standard basis, we find that $[\Sigma]\cdot [C_j]\in \{-1,0,1\}$ for every $j=0,\cdots, m$. Since $\Sigma$ and $C_j$ must intersect positively, either $\Sigma=C_j$ for some $j$, or $[\Sigma]\cdot [C_j]\in \{0,1\}$ for all $j=0,\cdots , m$. Lisca analyzes the relations between the $e_i$'s showing up with non-zero coefficients in different $[C_j]$'s within a linear plumbing, and proves that blowing down $\Sigma$ either reduces the length of the linear plumbing or reduces the absolute value of one of the self-intersection numbers of a $C_j$, ($j>0$), but the linear plumbing remains linear. This implies the conclusion of Lisca's theorem by induction.

In our case, we blow down a $J$-holomorphic sphere $\Sigma$, which may intersect any number of arms in the star-shaped graph. After blowing down, at least one arm is reduced in complexity, and each linear chain of symplectic spheres (originally these are the $k$ arms) remains linear. It is possible that two of these linear chains intersect after a blow-down, so the graph would no longer be star-shaped. Because the existence of a $J$-holomorphic sphere to blow down, and the homological properties of the $C_j$ do not depend on any assumptions about the non-intersection of the various arms, this does not prevent us from applying induction as in the linear case. We simply need to keep track of each linear chain separately, even as the chains intersect each other.

The conclusion is that each arm eventually descends to two symplectic spheres each of self-intersection number $1$, one of which is the original $C_0$ (since all blow-downs were done disjointly from $C_0$). Therefore in total we have $d+1$ symplectic spheres of self-intersection number $1$ inside $(\C P^2,\omega_0)$. Depending on choices of symplectic blow-ups determining $\omega_M$, the resulting symplectic form $\omega_0$ may no longer be the same as the standard Kahler form on $\C P^2$. However, if we only want to classify these symplectic fillings up to diffeomorphism, we may assume that $\omega_0=\omega_{std}$.
}
\subsubsection{Uniqueness of the pair after blowing down}
{
\label{uniqueembedding}
We need to analyze possible isotopy classes of a regular neighborhood of $d+1$ symplectic spheres in $(\C P^2,\omega_{std})$, each homologous to $\C P^1\subset \C P^2$. We also have that our original $J$ descends to an almost-complex structure $J_0$ on $\C P^2$, which is tamed by $\omega_0(=\omega_{std})$, and the remaining $d+1$ spheres in the reduction of the dual graph are $J_0$ holomorphic. We cannot assume that $J_0$ is the standard almost complex structure on $\C P^2$ since we had to choose $J$ originally so that each of the curves $C_0\cdots C_m$ were $J$-holomorphic. However, because $J_0$ is tamed by $\omega_{std}$, it is homotopic through almost complex structures tamed by $\omega_{std}$ to the standard almost complex structure $J_{std}$. This homotopy will allow us to isotope our symplectic spheres to complex projective lines in the following lemma. 

Working with curves which are $J$-holomorphic for some $J$ ensures that we need not worry about algebraically cancelling intersection points between curves since all intersections are positive. We want to control the way these $d+1$ spheres intersect, because an isotopy  of each of the $d+1$ spheres will extend to an isotopy of a regular neighborhood of their union only when the way these spheres intersect is preserved. Because all of these spheres must represent the homology class $\ell$ and they are all $J_0$-holomorphic, each pair of spheres must intersect at a single point. Because the blow-downs were disjoint from $C_0$ and each of the other $d$ spheres intersected $C_0$ in a different point initially, this remains true after blowing down. The other $d$ spheres may intersect each other at multi-points if before blowing down, a group of them intersected a common exceptional sphere. We will use the notation $\mathcal{I}^d_j$ to refer to a particular intersection configuration of $d+1$ spheres for which each pair intersects uniquely and positively. The intersection configuration $\mathcal{I}^d_j$ is specified by a diagram depicting $d+1$ (real) lines in the plane (each representing a sphere in $\C P^2$) as in figure \ref{fig:configurations}, where the multi-points of intersection are the relevant information.

First we would like to isotope our $J_0$ holomorphic spheres to complex projective lines, while keeping the intersection configuration fixed. A theorem of Gromov allows us to isotope each of the spheres to complex projective lines, but unfortunately we cannot always preserve the intersection configuration in sufficiently complicated configurations. However we are able to control this for the configurations given by figure \ref{fig:configurations}. These configurations will cover all of those we will obtain by blowing down dual graphs for dually positive Seifert fibered spaces with $3,4$ or $5$ singular fibers or with $k$ singular fibers and $e_0\leq -k-3$ (the more complicated configurations can arise when $e_0=-k-1$ or $-k-2$ and $k$ is large). 

The reason we have control over the configurations when $e_0\leq -k-3$ comes from the formula for the homology classes in lemma \ref{lem:e0neg}. After blowing down the exceptional spheres, the proper transforms of a subset of the spheres $C_0,C_1,\cdots , C_d$ will have a common intersection if they previously intersected the same exceptional sphere. Such intersections are determined by the homology classes (due to $J$-holomorphicity). If we also take into account the fact that $C_0$ intersects each other sphere in a distinct double point, then the only possible configurations to consider are $\mathcal{I}^d_1$ and $\mathcal{I}^d_2$ (figures \ref{fig:configuration3n1} and \ref{fig:configuration3n2}) containing $d+1\geq 4$ spheres. However, throughout our proof that we can isotope $J_0$ holomorphic spheres in such a configuration to complex projective lines in this configuration, we will also need to consider the intersection configuration $\mathcal{I}^d_3$ in figure \ref{fig:configuration3n3}.

We can also manage all the possible configurations of a small number of lines, which allows us to classify fillings of dually positive Seifert fibered spaces with $k=3,4$ or $5$ singular fibers. The possible intersection relations which can arise (keeping in mind that $C_0$ intersects each other sphere in a different point) are given by $\mathcal{I}_1^d$, $\mathcal{I}_2^d$, $\mathcal{I}_4^4$, $\mathcal{I}_5^5$, $\mathcal{I}_6^5$, and $\mathcal{I}_7^d$. We will additionally use $\mathcal{I}_3^d$,  $\mathcal{I}_8^d$, and $\mathcal{I}_{9}^d$ throughout the proof though they will not arise as initial intersection configurations. See figure \ref{fig:configurations}.

\begin{figure}
	\centering
	\subfloat[$\mathcal{I}_1^{d}$]{
	\label{fig:configuration3n1}	
	\includegraphics[scale=.5]{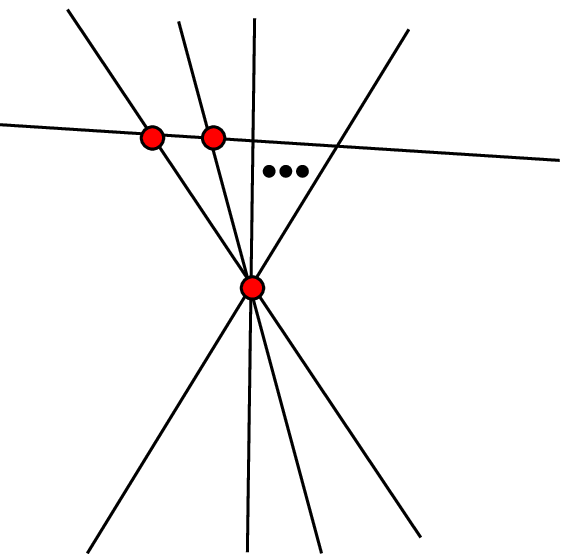}
	}
	\;\;\;\;\;\;\subfloat[$\mathcal{I}_2^{d}$]{
	\label{fig:configuration3n2}
	\includegraphics[scale=.5]{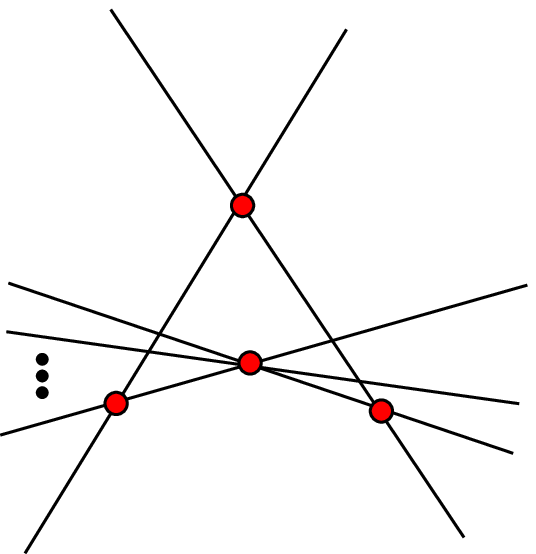}
	}
	\;\;\;\;\;\;\subfloat[$\mathcal{I}_3^{d}$]{
	\label{fig:configuration3n3}
	\includegraphics[scale=.5]{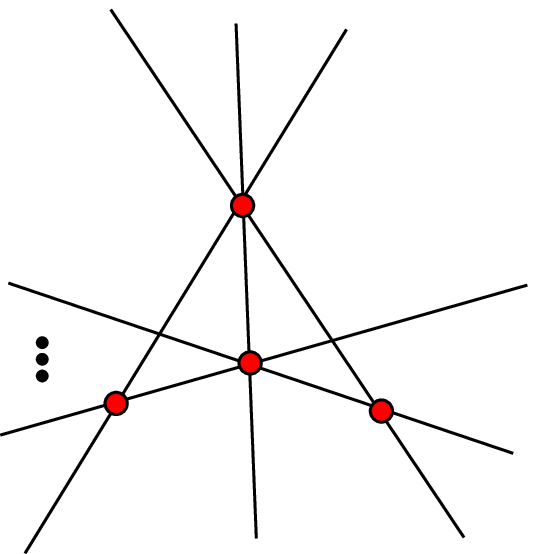}
	}\\
	\;\;\;\;\;\;\subfloat[$\mathcal{I}_4^4$]{
	\label{fig:configuration43}
	\includegraphics[scale=.5]{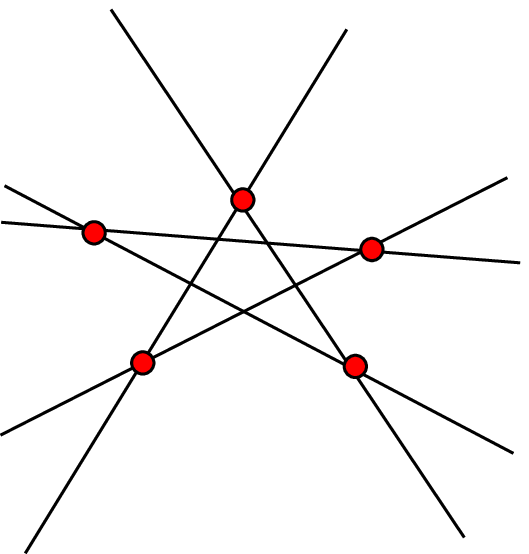}
	}
	\;\;\;\;\;\;\subfloat[$\mathcal{I}_5^5$]{
	\label{fig:configuration54}
	\includegraphics[scale=.5]{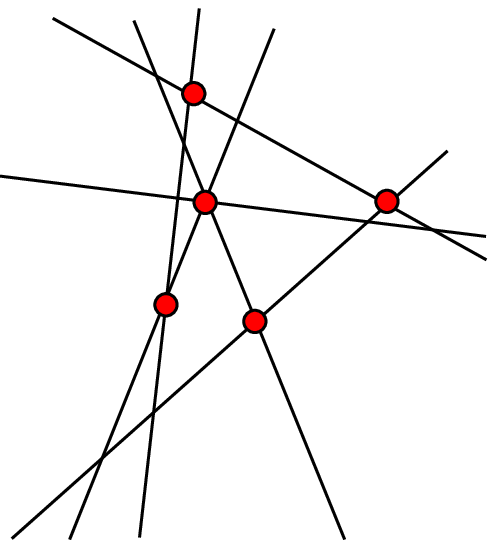}
	}
	\;\;\;\;\;\;\subfloat[$\mathcal{I}_6^5$]{
	\label{fig:configuration55}
	\includegraphics[scale=.5]{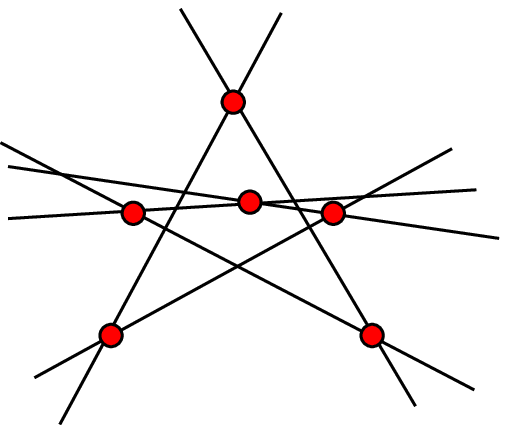}
	}\\
	\;\;\;\subfloat[$\mathcal{I}_7^d$]{
	\label{fig:configuration5n9}
	\includegraphics[scale=.5]{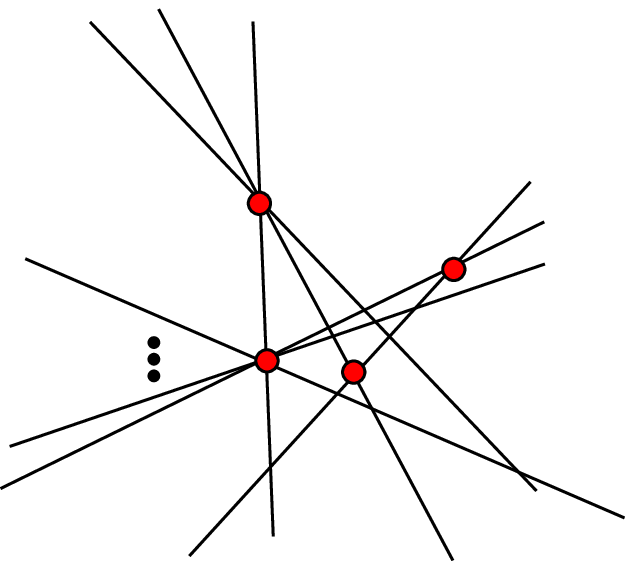}
	}	
	\;\;\;\;\;\;\subfloat[$\mathcal{I}_{8}^d$]{
	\label{fig:configuration5n10}
	\includegraphics[scale=.5]{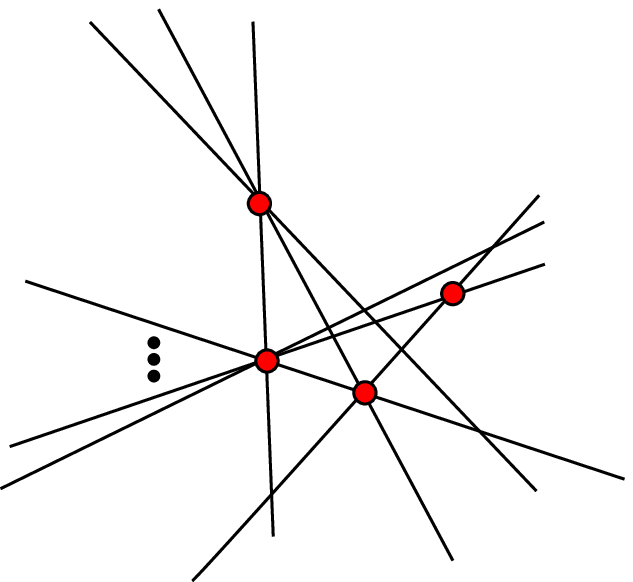}
	}
	\;\;\;\;\;\;\subfloat[$\mathcal{I}_{9}^d$]{
	\label{fig:configuration5n11}
	\includegraphics[scale=.5]{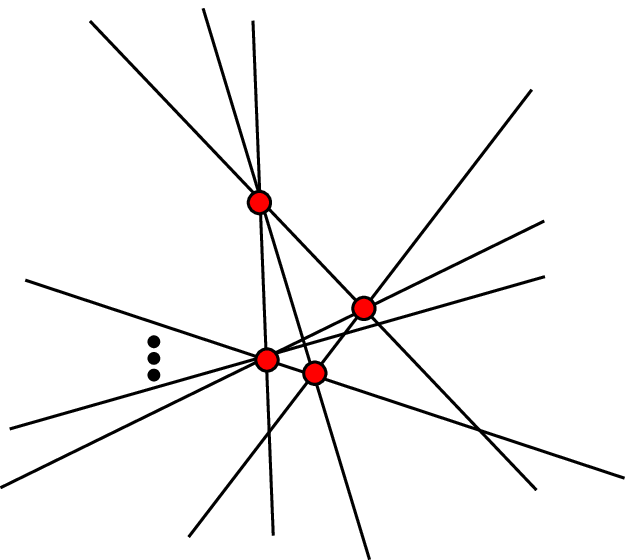}
	}
\caption{Each configuration $\mathcal{I}_j^d$ represents a union of $d+1$ spheres $C_0,C_1,\cdots, C_d$ for which each pair intersects at a single point. In our cases, $C_0$ will always intersect each of the other lines in a distinct double point. The diagrams represent how some of the other intersections can coincide at multi-points in the configurations we need to consider. We assume $d\geq 3$ for $\mathcal{I}_1^d$, $\mathcal{I}_2^d$, and $\mathcal{I}_3^d$ and that $d\geq 5$ for $\mathcal{I}_7^d$, $\mathcal{I}_{8}^d$, and $\mathcal{I}_{9}^d$.}
\label{fig:configurations}
\end{figure}

\begin{lemma}
Let $J_0$ be an almost complex structure on $\C P^2$ tamed by $\omega_{std}$. Suppose $C_0,C_1,\cdots C_d$ are $J_0$-holomorphic spheres homologous to $\C P^1\subset \C P^2$ such that $C_0,C_1,\cdots , C_d$ intersect according to one of the configurations given by one of $\mathcal{I}_1^{d}, \mathcal{I}_2^{d}, \mathcal{I}_4^4, \mathcal{I}_5^5$, $\mathcal{I}_6^5$, or $\mathcal{I}_7^d$ as given by figure \ref{fig:configurations}.  Then a regular neighborhood of $\cup_{i=0}^d C_i$ is isotopic to a regular neighborhood of complex projective lines in the same intersection configuration.
\label{lemma:4spherehtpy}
\end{lemma}

\begin{proof}
Because $J_0$ and $J_{std}$ are both tamed by $\omega_{std}$, there exists a family of almost complex structures $\{J_t\}$ on $\C P^2$ starting at $J_0$ and ending at $J_1=J_{std}$. A theorem of Gromov \cite{Gromov} states that for any $J$ tamed by the standard symplectic structure on $\C P^2$, any two points $v_1\neq v_2\in \C P^2$ lie on a unique non-singular rational (i.e. diffeomorphic to $S^2$) $J$-holomorphic curve homologous to $\C P^1\subset \C P^2$. Therefore for each $J_t$ in our homotopy, we can find a unique $J_t$-holomorphic sphere through two given points.

This provides an isotopy of each sphere $C_0,C_1,\cdots , C_d$ to a complex projective line (since these are the unique $J_{std}$-holomorphic spheres homologous to $\C P^1$). During this isotopy we can fix exactly two points on each sphere. This allows us to control the changes in the intersection configurations. At all times during this isotopy, each pair $C_i$, $C_j$ intersect positively at a unique point (since they will both be $J_t$-holomorphic). We will first analyze precisely what changes can can occur in the intersection configuration throughout this isotopy.

For the configuration $\mathcal{I}_1^{d}$, there are precisely two intersection points on each of $C_1,\cdots , C_d$, so for $i=1,\cdots, d$, we will choose the corresponding sphere $C_i^t$ to be the unique $J_t$-holomorphic sphere homologous to $\C P^1$ through those fixed points. Let $C_0^t$ be the unique $J_t$ holomorphic sphere homologous to $\C P^1$ through the points where $C_0^0=C_0$ intersects $C_1^0=C_1$ and where it intersects $C_2^0=C_2$. Then for all $0\leq t\leq 1$, $C_1^t,\cdots , C_d^t$ must continue to pass through a common point, and cannot intersect anywhere else. The intersections of $C_0^t$ with $C_1^t$ and $C_2^t$ are similarly fixed. While the intersections of $C_0^t$ with $C_i^t$ for $i\geq 3$ are not fixed points, these pairs must continue to intersect at double points since no two $C_i^t$ for $i=1,\cdots , d$ can intersect at any point away from their common $d$-fold intersection point, and $C_0$ cannot pass through this $d$-fold intersection because it intersects $C_1$ uniquely at a different point.

In general, if at a given intersection point, we choose to fix that point throughout the isotopy on every sphere $C_i^t$ for which $C_i^0$ passes through that point, then that intersection is preserved (though potentially other spheres may pass through that point during the isotopy if they are not otherwise constrained). In figure \ref{fig:configurations} we indicate such points by a red dot. We are allowed to fix these points as long as there are no more than two red dots on each sphere. Each of the configurations we are considering have the key property that no sphere passes through more than two multi-points, so we can place red dots on all the multi-points. This means that if the intersection configuration changes during the isotopy $\{\cup_{i=0}^d C_i^t\}_{t\in [0,1]}$, at worst it becomes a less generic intersection configuration.

We may analyze how double point intersections which are not fixed can move through the isotopy onto a third sphere, as in the argument above showing that the intersection configuration $\mathcal{I}_1^{d}$ is fixed under the isotopy through $J_t$-holomorphic spheres. The constraints we have are that each pair of spheres intersects at a unique point, and each sphere must pass through its two points we choose to keep fixed. These points include those indicated by a red dot in figure \ref{fig:configurations}, and the other fixed points do not provide particularly useful information except to provide the isotopy via Gromov's theorem. Analyzing this information, we obtain the information given by table \ref{table:degenerations}.

\begin{table}
$$\begin{array}{|l|c|c|c|c|}
\hline
\text{Original configuration} & \mathcal{I}_1^d & \mathcal{I}_2^d & \mathcal{I}_3^d & \mathcal{I}_4^4  \\ \hline
\text{Possible degenerations}  & \text{none} & \mathcal{I}_3^d & \text{none} & \mathcal{I}_2^4, \mathcal{I}_3^4  \\ \hline 
\end{array}
$$
$$
\begin{array}{|l|c|c|c|c|c|}
\hline
\text{Original configuration}  & \mathcal{I}_5^5 & \mathcal{I}_6^5 &\mathcal{I}_7^d & \mathcal{I}_{8}^d& \mathcal{I}_{9}^d \\\hline
\text{Possible degenerations} &  \mathcal{I}_2^5, \mathcal{I}_3^5, \mathcal{I}_7^5, \mathcal{I}_{8}^5, \mathcal{I}_{9}^5
&  \mathcal{I}_5^5, \mathcal{I}_7^5, \mathcal{I}_{8}^5, \mathcal{I}_{9}^5 & \mathcal{I}_{8}^d, \mathcal{I}_{9}^d & \mathcal{I}_{9}^d & \text{none} \\\hline
\end{array}
$$

\caption{Possible degenerations of each configuration that can occur during an isotopy through $J_t$ holomorphic spheres fixing the intersection points corresponding red dots in figure \ref{fig:configurations}.}
\label{table:degenerations}
\end{table}

Now suppose we have any of these starting intersection configurations and we have chosen two points on each sphere such that any sphere passing through a point corresponding to a red dot in figure \ref{fig:configurations} has that point chosen. These points together with the homotopy $\{J_t\}$ define an isotopy of the $d+1$ spheres which may change the intersection configuration as allowed by table \ref{table:degenerations}. If at some point in the isotopy the intersection configuration changes, redefine the isotopy from the time that the first degeneration occurs onwards by changing the choice of two points on each sphere so that they include the points corresponding to red dots in the new (more degenerate) configuration. Then repeat this again at a later time value if the intersection configuration changes again. Notice that after finitely many steps (at most three in these cases), the intersection configuration must remain fixed under the isotopy. Therefore we obtain an isotopy from our original configuration of $d+1$ $J_0$-holomorphic spheres to $d+1$ complex projective lines where the intersection configuration changes at most by finitely many degenerations.

Now we will show that we can modify an isotopy that degenerates once to an isotopy that preserves the intersection configuration and passes through $J_t$-holomorphic spheres for all but a small interval near the time where the degeneration originally occured.

Each of these possible degenerations occurs when a sphere $C_{i_0}^t$ which originally intersected other spheres $C_{i_1}^t, \cdots , C_{i_{\ell}}^t$ generically in double points for $t<t_0$ is isotoped so that it intersects $C_{i_1}^t, \cdots , C_{i_{\ell}}^t$ at their common intersection point ($\ell \geq 2$) for $t\geq t_0$, or when this occurs in finitely many different places at the same time $t_0$. 

Because the intersection points corresponding to red dots are fixed and distinct, $C_{i_0}\cap C_{i_j}$ for $j=1,\cdots , \ell$ cannot be points marked by a red dot. However, the intersection $C_{i_1}\cap \cdots \cap C_{i_{\ell}}$ can be (and usually will be) marked by a red dot, as the sphere $C_{i_0}^t$ approaches that fixed intersection point. Notice that in the intersection configurations we are considering, no sphere has more than three intersection points that are not marked by a red dot, except in $\mathcal{I}_1^d, \mathcal{I}_2^d, \mathcal{I}_3^d, \mathcal{I}_7^d, \mathcal{I}_{8}^d$ and $\mathcal{I}_{9}^d$. The configurations $\mathcal{I}_1^d, \mathcal{I}_3^d$, and $\mathcal{I}_{9}^d$ cannot degenerate at all, the configurations $\mathcal{I}_2^d$ and $\mathcal{I}_{8}^d$ can only degenerate at one place, and none of the spheres in $\mathcal{I}_7^d$ containing four intersections points not marked by a red dot can degenerate in a way that brings two pairs of these four points together. Therefore, if degenerations occur in multiple places at once in any of the configurations, the corresponding $C_{i_0}$'s will be distinct. Therefore we can perform all our modifications to the isotopy near degenerations in distinct places independently. 

After a degeneration at $t=t_0$, the isotopy is defined by fixing pairs of points on each sphere including all the red dots on the more degenerate intersection configuration. Therefore one of the two fixed points on $C_{i_0}^t$ is at its intersection with $C_{i_1}^t\cap \cdots \cap C_{i_{\ell}}^t$ for all $t\geq t_0$.

Choose an injective parametrization of the spheres $\cup_{i=0}^{d}C_i^t$ 
$$F:\sqcup_{i=0}^{d}S^2_i \times I \to \C P^2$$
so that $F(S_i^2, t) = C_i^t$. Let $p\in S^2_{i_1}$ be the unique preimage in $S_{i_1}^2$ of $v_1=C_{i_0}^{t}\cap C_{i_1}^{t} \cap \cdots \cap C_{i_{\ell}}^{t}\in \C P^2$ for $t\geq t_0$ (note one can choose the parameterization $F$ so that $p$ is the same for all $t\geq t_0$ since its image is fixed as $t\geq t_0$ varies). Let $v_0 \in \C P^2$ be the other point on $C_{i_0}^{t_0}$ which is fixed during the isotopy as $t\geq t_0$ varies. Choose a sufficiently small injective parametrized path $\gamma: I\to S^2_{i_1}$ with coordinate $s\in I$, such that $\gamma(0)=p$. We will extend our isotopy $F$ to a two parameter isotopy $G$ depending on $s$ and $t$ for all $t\geq t_0$ as follows. For $t\geq t_0$ and $s\in I$, let $G(S_{i_0}, t,s)=C_{i_0}^{(t,s)}$ be the unique $J_t$-holomorphic sphere homologous to $\C P^1$ that passes through $v_0$ and $F(\gamma(s),t)\in C_{i_1}$. Let $G(x, t,s)=F(x,t)$ for $x\in S_i^2$ where $i\neq i_0$. Then for $s=0$ and $t\geq t_0$, $C_{i_0}^{(t,0)}=C_{i_0}^t$. I claim that for $s>0$ and $t\geq t_0$, $C_{i_0}^{(t,s)}$ intersects each $C_{i_j}^{(t,s)}$ for $j=1,\cdots , \ell$ in a distinct double point so the degeneration is undone. This is because $C_{i_0}^{(t,s)}$ necessarily intersects $C_{i_1}^{(t,s)}$ at $F(\gamma(s),t)\neq F(p,t)=C_{i_1}^{(t,s)}\cap \cdots \cap C_{i_{\ell}}^{(t,s)}$ so $C_{i_0}^{(t,s)}$ cannot intersect $C_{i_j}^{(t,s)}$ at its intersection with $C_{i_{j'}}^{(t,s)}$ for $j,j'\in \{1,\cdots, \ell \}$. By choosing $\gamma$ to be a sufficiently small path the intersections of $C_{i_0}$ with other $C_i$ will be preserved (as generic double points except possibly at $v_0$ which is fixed). Additionally note that $C_{i_0}^{(t,s)}$ is distinct from $C_{i_0}^{(t,s')}$ for $s\neq s'$ since their intersections with $C_{i_1}^{(t,s)}$ are different by injectivity of $\gamma$. 

\begin{figure}
\centering
\psfragscanon
\psfrag{1}{$t$}
\psfrag{2}{$t_0$}
\psfrag{3}{$s$}
\includegraphics[scale=.8]{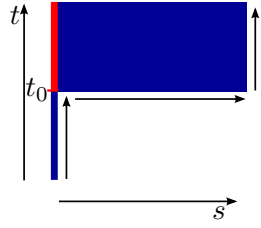}
\psfragscanoff
\caption{A schematic for the isotopies $F$ and $G$. The blue indicates where $F$ and $G$ map into the original less degenerate configuration, and the red indicates where $F$ and $G$ map into the more degenerate configuration.}
\label{fig:isotopy}
\end{figure}

Now we can define an isotopy of the $d+1$ spheres which preserves the original intersection configuration except at a single point during the isotopy where it degenerates by following the original isotopy until time $t_0$ then following the isotopy through $\cup C_i^{(t,s)}$ for $t=t_0$, $s\in [0,1]$, and then through $s=1$, $t\in [t_0,1]$. (See figure \ref{fig:isotopy} for a schematic.) We can modify this isotopy in a small neighborhood of the single degeneration point, by briefly moving out of the space of $J_t$-holomorphic curves. Consider the space of all smooth immersions of $d+1$ spheres, each homologous to $\C P^1$ with given intersection data. This space includes all the $J$-holomorphic configurations we have considered as a subspace. The main problem with working with smooth but not $J$-holomorphic spheres, is that a generic isotopy through $d+1$ smooth embeddings of such spheres passes through points where the number of geometric intersections between any two spheres can increase by adding an algebraically cancelling pair. However, if we stay sufficiently close to $J$-holomorphic configurations we can avoid this since algebraically cancelling pairs are introduced after a tangency between two spheres which is a closed condition. By construction we have two paths from the degenerate point in our isotopy out into the space of $d+1$ spheres intersecting in the configuration $\mathcal{I}$, we are interested in. Choosing a sufficiently small $\varepsilon$, we will find a path through embeddings of $d+1$ spheres homologous to $\C P^1$ intersecting according to $\mathcal{I}$, from $F(\sqcup_{i=0}^dS_i^2, t_0-\varepsilon)=\cup_{i=0}^d C_i^{(t_0-\varepsilon,0)}$ to $G(\sqcup_{i=0}^d S_i^2, t_0, \varepsilon)=\cup_{i=0}^d C_i^{(t_0,\varepsilon)}$. Let $N_1, N_2$ be small connected neighborhoods in $S_{i_0}^2$ so that $\overline{N_1}\subset N_2$ and the images $F(N_q,t)\subset C_{i_0}^t$ for $t\in [t_0-\varepsilon, t_0]$, $q=1,2$ contain the intersections of $C_{i_0}^{t}$ with $C_{i_j}^t$ for $j=1,\cdots, \ell$ but no other intersections, and similarly $G(N_q,t_0,s)\subset C_{i_0}^{(t,s)}$ contains the intersections of $C_{i_0}^{(t_0,s)}$ with $C_{i_j}^{(t_0,s)}$ for $j=1\cdots , \ell$, $s\in [0,\varepsilon]$ and no other intersections. We construct the isotopy $I: \sqcup_{i=0}^nS_i^2\times I  \to \C P^2$ as follows.
$$I(x,r)=\begin{cases} F(x,t_0-\varepsilon+2r\varepsilon) & \text{ for } x\notin N_1, r\in [0,\frac{1}{2}]\\
 G(x,t_0,2\varepsilon(r-\frac{1}{2})) & \text{ for } x\notin N_1, r\in [\frac{1}{2},1]\\
 H(x,r) & \text{ for } x\in N_2 \end{cases} $$
Here $H: N_2 \times I \to \C P^2$ is a generic isotopy everywhere close to, and for $r=0,1$ or $x\in N_2\setminus N_1$ agreeing with, the isotopy $\widehat{H}: N_2\times I \to \C P^2$ defined by 
$$\widehat{H}(x,r)=\begin{cases} F(x,t_0-\varepsilon+2r\varepsilon) & \text{ for } r\in [0,\frac{1}{2}]\\
 G(x,t_0,2\varepsilon(r-\frac{1}{2})) & \text{ for } r\in [\frac{1}{2},1]\end{cases}$$

Note that a generic ($1$-parameter) isotopy mapping $N_2$ (which is $2$-dimensional) into $\C P^2$ (which is $4$-dimensional) will avoid the ($0$-dimensional) point $v_1\in \C P^2$ where $C_{i_1}^{(t,s)}$ intersects $C_{i_2}^{(t,s)}, \cdots , C_{i_{\ell}}^{(t,s)}$ for all $(t,s)$ (identifying $C_i^t$ with $C_i^{(t,0)}$ for $t<t_0$). By taking $H$ generic and close to $\widehat{H}$ we can ensure that $N_2$ maps into a small neighborhood of $v_1$ (thus avoiding introducing new intersections with other spheres) for each $r$ and $H(N_2,r)$ uniquely intersects each $H(S_{i_j}^2,r)$ in a distinct double point. 

Therefore $I$ provides an isotopy we can append between $F(\cdot , t)$ for $t\in [0,t_0-\varepsilon]$ and $G(\cdot, t_0,s)$ for $s\in [\varepsilon,1]$ followed by $G(\cdot, t, 1)$ for $t\in [t_0,1]$. Putting these all together we find an isotopy of the spheres $C_0,\cdots , C_d$ which preserves the intersection configuration.

We conclude our argument by dealing with the possibility that multiple degenerations occur as we go through the homotopy $\{J_t\}$. If the intersection configuration changes at times $t_0,t_1,\cdots, t_n$ during the isotopy defined by $\{J_t\}$ then repeat the process above starting at the most degenerate configuration ($t\in [t_n,1]$). Use the process above to create an isotopy for $t\in [t_{n-1}, 1]$ that preserves the intersection configuration of the spheres that occurs at time $t_{n-1}$. Then repeat the process to obtain an isotopy over successively larger time intervals $[t_{n-1},1]\supset [t_{n-1},1] \supset \cdots \supset [t_0,1] \supset [0,1]$, which fix the intersection configuration that occurs at the initial point of the interval.
\end{proof}

This tells us that, up to isotopy of a regular neighborhood, we may assume that the blow-down of our dual graph is a set of complex projective lines in $\C P^2$. We would like to know that the intersection information of the complex projective lines determines the isotopy type of a regular neighborhood of their union. It suffices to show that we can get from any collection of complex projective lines to any other of the same intersection configuration through such collections of lines.

\begin{lemma} The space of $d+1$ lines of intersection configuration $\mathcal{I}^{d}_1, \mathcal{I}^{d}_2, \mathcal{I}^4_4, \mathcal{I}^5_5$, $\mathcal{I}^5_6$, or $\mathcal{I}^d_7$ is connected and non-empty. \label{lemma:configconnected} \end{lemma}

\begin{proof}
Fix $(a,d)\in \{(1,d), (2,d), (4,4), (5,5), (5,6), (d,7)\}$. Order the lines in the configuration $C_0,C_1,\cdots, C_d$ where $C_0$ intersects all other lines in double points. For $1\leq n\leq d$, let $S_n(\mathcal{I}_a^d)$ denote the intersection configuration formed by the subset of lines $C_0 \cup \cdots \cup C_n$. Note that $S_d(\mathcal{I}_a^d)=\mathcal{I}_a^d$, and increasing n corresponds to placing additional lines in the configuration as a subset of $\mathcal{I}_a^d$. Let $M_n$ denote the moduli space of complex projective lines intersecting according to the configuration $S_n(\mathcal{I}_a^d)$.

Any complex projective line is determined by choosing two distinct points in $\C P^2$. Through a diffeomorphism given by an element in $PGL(2,\C)$ we may assume that two of the lines are standardly given as $C_0=\{[0:z_1:z_2]\}$ and $C_1=\{[z_0:0:z_2]\}$. Therefore, when $n=1$, $\mathcal{M}_n$ is a point after we mod out by the $PGL(2,\C)$ action on $\C P^2$. 

Any other line intersects each of $C_0$ and $C_1$ in a single point. Because $C_0$ intersects each line in a distinct point, there are no other $C_i$ passing through $C_0\cap C_1$. Therefore $C_2$ intersects $C_0$ and $C_1$ each away from $C_0\cap C_1$, therefore (regardless of our choice of $(a,d)$) $\mathcal{M}_2$ is parametrized by $[C_0\setminus (C_0\cap C_1)]\times [C_1\setminus (C_0\cap C_1)]$.

Now, we proceed by induction on $1\leq n\leq d$. Here it will be important that $(a,d)$ is one of the choices specified in the statement. The property that all of these intersection configurations share, is that no line contains more than two multi-points of intersection (where multi-point means a point where at least three lines pass through the same point).

Inductively assume that the moduli space $\mathcal{M}_{n-1}$ of configurations of projective lines $C_0,C_1,\cdots , C_{n-1}$ is a connected subset of $[C_0\setminus (C_0\cap C_1)]^{a_{n-1}}\times [C_1\setminus (C_0\cap C_1)]^{b_{n-1}}$. There are three cases for how $C_n$ is determined. The first is that $C_n$ meets $C_j$ in a double point for all $0\leq j \leq n-1$. The second is that $C_n$ passes through the common intersection of $C_{j_1}\cap \cdots \cap C_{j_{\ell}}$ for $\ell \geq 2$, $1\leq j_1<j_2<\cdots <j_{\ell}\leq n-1$, and $C_n$ meets $C_j$ at a double point for $j\neq j_1,\cdots, j_{\ell}$. The third is that $C_n$ passes through $C_{j_1}\cap\cdots \cap C_{j_{\ell}}$ and $C_{j_{\ell+1}}\cap\cdots \cap C_{j_{\ell+\ell'}}$ for $\ell,\ell'\geq 2$, $1\leq j_i\leq n-1$, and $j_i$ distinct, but $C_n$ intersects $C_j$ at a double point for $j\in \{1,\cdots , n-1\}\setminus \{j_1,\cdots,j_{\ell'}\}$. Note that these are all the possibilities we must consider, because in all of the configurations listed in the statement of the theorem, each line contains no more than two multi-intersection points.

In the first case, $C_n$ is determined by choosing its intersection points on $C_0\setminus (C_0\cap C_1)$ and $C_1\setminus (C_0\cap C_1)$, but some pairs $(v_1,v_2)\in [C_0\setminus (C_0\cap C_1)]\times [C_1\setminus (C_0\cap C_1)]$ will determine a line that passes through some intersection point $C_j\cap C_{j'}$ for $0\leq j,j'<n$. However, this is a real codimension $2$ phenomenon because there is a real 2-parameter family of projective lines that pass through points on $C_j$ in a neighborhood of $C_j\cap C_{j'}$. Therefore the subset of $\{(x,v_1,v_2)\in \mathcal{M}_{n-1}\times [C_0\setminus (C_0\cap C_1)]\times [C_1\setminus (C_0\cap C_1)]\}$ corresponding to configurations where $C_n$ does not intersect each $C_j$ for $j<n$ in a double point is a codimension $2$ subset, so its complement, $\mathcal{M}_n$, is connected and satisfies the inductive hypothesis.

In the second case, $C_n$ is required to pass through $C_{j_1}\cap \cdots \cap C_{j_{\ell}}$ and must avoid other intersections of the $C_j$, for $j<n$. The point $C_{j_1}\cap \cdots \cap C_{j_{\ell}}$ together with a point on $C_0\setminus (C_0\cap C_1)$ determines a unique line, and the configuration space of lines $C_0, \cdots , C_{n-1}$ together with this line is parameterized by $\mathcal{M}_{n-1}\times [C_0\setminus (C_0\cap C_1)]$, which is connected by the inductive assumption. Again, we must remove points from this space for which the last line passes through intersections of the $C_j$, $j<n$ other than $C_{j_1}\cap \cdots \cap C_{j_{\ell}}$, but this is a real codimension $2$ phenomenon so the complement, $\mathcal{M}_n$ is connected and satisfies the inductive hypothesis.

In the third case, $C_n$ is required to pass through two points, $C_{j_1}\cap\cdots \cap C_{j_{\ell}}$ and $C_{j_{\ell+1}}\cap \cdots \cap C_{j_{\ell+\ell'}}$, which are completely determined by the point in $\mathcal{M}_{n-1}$, so we gain no additional parameters. However, to get $\mathcal{M}_n$, we do need to cut out a subspace of $\mathcal{M}_{n-1}$ for which the line determined by the intersection points $C_{j_1}\cap\cdots \cap C_{j_{\ell}}$ and $C_{j_{\ell+1}}\cap \cdots \cap C_{j_{\ell+\ell'}}$ passes through an intersection of $C_j\cap C_{j'}$ for $(j,j')\notin \{j_1,\cdots , j_{\ell} \}^2\cup \{j_{\ell+1}, \cdots , j_{\ell+\ell'}\}^2$. This occurs when a complex equation of $x\in \mathcal{M}_{n-1}$ is satisfied, so it is again a codimension $2$ phenomenon so $\mathcal{M}_n$ is connected.

Inductively this proves that $\mathcal{M}_d$ is parameterized by a connected subspace of $[C_0\setminus (C_0\cap C_1)]^{a_d}\times [C_1\setminus (C_0\cap C_1)]^{b_d}$. 

It remains to show that each of these moduli spaces is non-empty. To do this we provide a model for each in terms of standard homogeneous coordinates $\{[z_0:z_1:z_2]\}$ on $\C P^2$.
$$\begin{array}{|c|l|}
\hline
\mathcal{I}_1^d&\{z_1=0\}, \{jz_0-z_2=0\}\text{ for } 1\leq j\leq d\\\hline
\mathcal{I}_2^d&\{z_1=0\}, \{z_0-z_1=0\}, \{jz_0-z_2=0\} \text{ for } 2\leq j \leq d\\\hline
\mathcal{I}_4^4&\{z_0=0\}, \{z_1=0\}, \{z_2=0\}, \{z_0+z_1+z_2=0\}, \{2z_0+z_1-z_2=0\}\\\hline
\mathcal{I}_5^5& \{z_0=0\}, \{z_1=0\}, \{z_2=0\}, \{z_0+z_1+z_2=0\}, \{2z_0+z_1-z_2=0\}, \{2z_1-z_2=0\}\\\hline
\mathcal{I}_6^5&\{z_0=0\}, \{z_1=0\}, \{z_2=0\}, \{z_0+z_1+z_2=0\}, \{2z_0+z_1-z_2=0\}, \{3z_0+2z-1-4z_2=0\}\\\hline
\mathcal{I}_7^d&\{z_0+z_1+z_2=0\}, \{z_0=0\}, \{z_1=0\}, \{z_2=0\}, \{2z_1-z_2=0\}, \{jz_2-z_0=0\} \text{ for } 5\leq j\leq d\\\hline
\end{array}$$
\end{proof}

\emph{Remark:} In contrast, if we considered a configuration where some of the lines passed through three or more multi-points, the fact that such a line exists puts a constraint on the variables determining the earlier lines. These yield polynomial relations on the homogeneous coordinates of the points determining the lines. With sufficiently many lines, one can construct intersection configurations for which the space of complex projective lines in that configuration is disconnected or empty. 

We conclude that for the listed configurations, the homology representation of the dual graph determines a unique embedding of a regular neighborhood of the dual graph. This is because the regular neighborhood of the dual graph can be obtained by blowing up one of the above discussed configurations at points specified by the homology representation.

Note that although this provides an embedding of this dual configuration of spheres into a symplectic manifold, it does not ensure that there is a concave neighborhood of this configuration (it is not known whether this always exists). Therefore we only obtain upper bounds on the number of symplectic fillings, and will find lower bounds through other means in our examples.
}
}

\subsection{Proof of theorem \ref{thm:finite}}
{
We can now obtain the finiteness results of theorem \ref{thm:finite} as a corollary.

\begin{proof}
Given any symplectic filling, it can be glued to the concave neighborhood of the dual configuration to give $(\C P^2\# M\overline{\C P^2},\omega_M)$ by sections \ref{dualgraph}, \ref{cutpaste}, and \ref{McDuff}. The homology classes the spheres in the dual configuration can represent in $H_2(\C P^2\# M\overline{\C P^2})$ are restricted as discussed in section \ref{sec:homology}. If the number of singular fibers is $k=3,4$ or $5$ and $e_0=-k-1$, the number of arms in the dual graph $d$ is sufficiently small so the only possible intersection configurations of the image of the spheres in the dual graph after blowing-down are $\mathcal{I}_1^3, \mathcal{I}_2^3$, $\mathcal{I}_1^4$, $\mathcal{I}_2^4$, $\mathcal{I}_4^4$, $\mathcal{I}_1^5$, $\mathcal{I}_2^5$, $\mathcal{I}_5^5$, $\mathcal{I}_6^5$, or $\mathcal{I}_7^5$. When $e_0=-k-2$ and $k=3,4,5$, the results of lemma \ref{lem:e0neg} imply that the only possible intersection configurations are $\mathcal{I}_1^4$, $\mathcal{I}_2^4$, $\mathcal{I}_1^5$,$\mathcal{I}_2^5$, $\mathcal{I}_7^5$, $\mathcal{I}_1^6$, $\mathcal{I}_2^6$, and $\mathcal{I}_7^6$. When $e_0\leq -k-3$, lemma \ref{lem:e0neg} implies that the only possible intersection configurations are $\mathcal{I}^d_1$ and $\mathcal{I}^d_2$ where $d=-e_0-1$. Therefore, lemmas \ref{lemma:4spherehtpy} and \ref{lemma:configconnected} imply that a regular neighborhood of the image of the spheres in the dual graph after blowing down embeds uniquely up to isotopy into $\C P^2$. Therefore the embedding of a regular neighborhood of the dual graph is uniquely determined up to isotopy by blowing up at the points dictated by the homology classes each sphere in the dual graph represents in $H_2(\C P^2 \# M\overline{\C P^2})$. The number of possible embeddings of a neighborhood of the dual graph into $\C P^2\# M\overline{\C P^2}$ provides an upper bound on the number of convex symplectic fillings of the corresponding Seifert fibered space with contact structure $\xi_{pl}$.

By section \ref{sec:homology}, in terms of a standard basis $(\ell, e_1,\cdots , e_M)$ for $H_2(\C P^2\# M\overline{\C P^2})$, the homology classes $[C_0],\cdots , [C_m]$ must have the form
\begin{eqnarray*}
\left[C_0\right]&=& \ell\\
\left[C_1\right]&=& \ell -e_{i^1_1}-\cdots -e_{i^1_{n_1+1}}\\
&\vdots &\\
\left[C_k\right]&=& \ell -e_{i^k_1}-\cdots -e_{i^k_{n_k+1}}\\
\left[C_{k+1}\right]&=& \ell-e_{i^{k+1}_1}-e_{i^{k+1}_2}\\
&\vdots & \\
\left[C_d\right]&=& \ell-e_{i^d_1}-e_{i^d_2}\\
\left[C_{d+1}\right]&=& e_{i^{k+1}_0} -e_{i^{k+1}_1}-\cdots -e_{i^{k+1}_{n_{k+1}-1}}\\
&\vdots &\\
\left[C_{m}\right]&=& e_{i^{m}_0} -e_{i^{m}_1}-\cdots -e_{i^{m}_{n_{m}-1}}\\
\end{eqnarray*}
so up to symmetry of relabelling the $e_i$, the only possible differences are determined by the pairs $(j,p),(j',p')$ for which $i_p^j=i_{p'}^{j'}$. There are significant restrictions determined by the intersection data of the spheres, but we can easily get a very loose upper bound on the number of possibilities. If $N=\sum_{j=1}^m (n_j+1)$, then there are less than $N$ different $e_i$ which have non-zero coefficient in any $[C_j]$. Thus there are at less than $N^N$ arrangements of these $e_i$ into the slots above. We also may assume that a minimal symplectic filling is the complement of one of these embeddings in $\C P^2\# M\overline{\C P^2}$ where $M\leq N$, because $N$ is the largest number of exceptional spheres which intersect the dual configuration of spheres. Therefore there are certainly no more than $N^{N+1}$ diffeomorphism types of minimal symplectic fillings of a given Seifert fibered space satisfying the hypotheses of the theorem with the contact structure $\xi_{pl}$.

For a general dually positive Seifert fibered space (no restrictions on $k$ or $e_0$ other than $e_0\leq -k-1$), the results discussed in sections \ref{dualgraph}, \ref{cutpaste}, \ref{McDuff}, and \ref{sec:homology} still hold, though we no longer have a proof that the dual graph has a unique embedding into $\C P^2\# M\overline{\C P^2}$ for each homology representation. However, we still have that any strong symplectic filling glues to the dual graph to give $\C P^2\#M\overline{\C P^2}$, so the homology of the filling fits into the following long exact sequence by the Mayer-Vietoris theorem.
$$0\to H_4(\C P^2\#\overline{\C P^2})\to H_3(Y)\to H_3(\text{dual graph})\oplus H_3(\text{filling})\to H_3(\C P^2\#\overline{\C P^2})$$
This immediately implies $H_3(\text{filling})=0$ after noting which terms are closed manifolds of the correct dimension. The rest of the Mayer-Vietoris sequence over rational coefficients is as follows.
$$H_3(\C P^2\# M\overline{\C P^2};\Q)\to H_2(Y;\Q)\to H_2(\text{dual graph};\Q)\oplus H_2(\text{filling};\Q)\to H_2(\C P^2\# M\overline{\C P^2};\Q)$$
$$\to H_1(Y;\Q) \to H_1(\text{dual graph};\Q)\oplus H_1(\text{filling};\Q)\to H_1(\C P^2\# M\overline{\C P^2})$$
Since $Y$ is a 3-manifold, Poincare duality implies $H_2(Y;\Q)\isom H_1(Y;\Q)\isom \Q^{b_1(Y)}$. Note that the homology of the dual graph is generated by the $m+1$ spheres that are plumbed together. Filling in known terms of the sequence gives
$$0\to \Q^{b_1(Y)}\to \Q^{m+1}\oplus H_2(\text{filling};\Q)\to \Q^{M+1}\to \Q^{b_1(Y)}\to H_1(\text{filling};\Q)\to 0.$$
Exactness at the end of the sequence implies $b_1(\text{filling})\leq b_1(Y)$. Chasing through the rest of the sequence gives the following equation.
$$b_1(\text{filling})+M+1=m+1+b_2(\text{filling}).$$
Therefore the Euler characteristic and signature satisfy $$\chi(\text{filling})=M-m+1$$
and
$$|\sigma(\text{filling})|\leq b_2(\text{filling})\leq M-m+b_1(Y).$$
If the symplectic filling is minimal, then every exceptional sphere in $\C P^2\#M\overline{\C P^2}$, must intersect the dual graph. Since there is some almost complex structure $J$ for which the spheres in the dual graph are $J$-holomorphic and we can find $J$-holomorphic representatives of the exceptional spheres (as argued in section \ref{sec:blowingdown}), the algebraic and geometric intersection numbers of these spheres coincide. In particular, every basis element $e_i\in H_2(\C P^2\#M\overline{\C P^2})$ of square $-1$ must appear with non-zero coefficient in the homology class $[C_j]$ of some sphere in the dual graph. As mentioned for the first half of the theorem, if $[C_j]^2=-n_j$ for $j=1,\cdots, m$, the number of $e_i$ showing up with non-zero coefficient in the homology of the dual graph is bounded above by $N=\sum_{j=1}^m(n_j+1)$. Therefore $$\chi(\text{filling})\leq M-m+1\leq N$$
and
$$|\sigma(\text{filling})|\leq M-m+b_1(Y)\leq N+b_1(Y).$$
Note that $N$ is determined by the dual graph, which can be generated given the Seifert invariants of $Y$. 
\end{proof}

Of course, the bound given here is far from sharp. Note that one can obtain tighter bounds in specific cases when the possible homology representations of the dual graph are known. Homology representations which are ``more efficient'' (namely they can be written using few distinct $e_i$), will yield fillings with stricter bounds on $\chi , \sigma$. We will see in examples that the homology representations yielding the filling given by the original plumbing are the least efficient with the $e_i$'s, so it is not surprising that our examples will give alternate fillings with smaller Euler characteristic than the original plumbing.
}
}

\section{Simplest Examples}
{
\label{simplestexamples}
In order to see how this argument produces a list of diffeomorphism types of possible symplectic fillings of a given dually positive $(Y,\xi_{pl})$, it is useful to look at concrete examples. Because the number of symplectic fillings is determined by possible ways of writing the homology classes of the spheres in the dual graph, one can obtain a simple family of examples by insisting that the dual graph is star-shaped with three arms, and each arm has length one as in figure \ref{fig:DualGraphFamilyA}. These graphs arise as the dual graphs of three armed graphs with central vertex decorated by $-4$ and all other vertices decorated by $-2$. These graphs depend on three positive integer parameters $n_1$, $n_2$, and $n_3$ which determine the lengths of the arms as in figure \ref{fig:GraphFamilyA}, and are the negations on the coefficients of the spheres in the arms of the dual graph. We work through the classification of convex fillings for this example in full detail to make it clear how to obtain the possible diffeomorphism types of symplectic fillings from the homological restrictions.

\begin{figure}
	\centering
	\subfloat[Graphs]{
		 \label{fig:GraphFamilyA}
		 \psfragscanon
		 \psfrag{1}{$-2$}
		 \psfrag{2}{$-2$}
		 \psfrag{3}{$-2$}
		 \psfrag{4}{$-2$}
		 \psfrag{5}{$-2$}
		 \psfrag{6}{$-2$}
		 \psfrag{7}{$-4$}
		 \psfrag{8}{$n_1-1$}
		 \psfrag{9}{$n_2-1$}
		 \psfrag{A}{$n_3-1$}
		 \includegraphics[width=5cm]{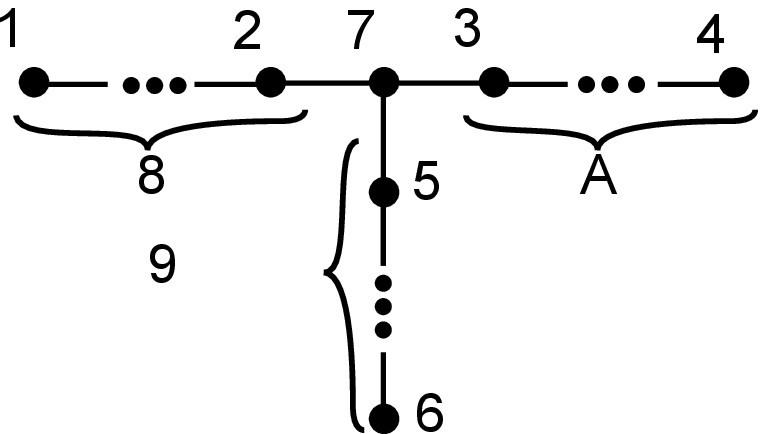}
		 \psfragscanoff
		} 
	\;\;\;\;\;\;\;\;\;	\subfloat[Dual Graphs]{
		\label{fig:DualGraphFamilyA}
		\psfragscanon
		\psfrag{0}{$+1$}
		\psfrag{A}{$-n_1$}
		\psfrag{B}{$-n_2$}
		\psfrag{C}{$-n_3$}
		\includegraphics[width=3.5cm]{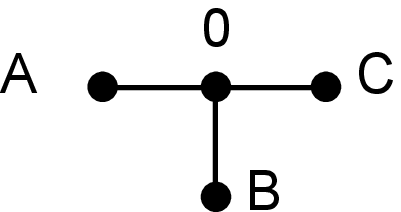}
		\psfragscanoff
		} 
 \caption{The graphs and dual graphs representing a simple family of plumbings.}
\label{fig:GraphDualGraphFamilyA}
 \end{figure}

Now we need to determine the possible ways to write $[C_0],[C_1],[C_2],[C_3]$ in terms of the standard basis $(\ell, e_1,\cdots , e_M)$ for $H_2(X_M;\Z)$. By section \ref{sec:homology}, we know that they must have the form
\begin{eqnarray*}
\left[C_0\right]&=& \ell\\
\left[C_1\right]&=& \ell -e_{i^1_1}-\cdots -e_{i^1_{n_1+1}}\\
\left[C_2\right]&=& \ell -e_{i^2_1}-\cdots -e_{i^1_{n_2+1}}\\
\left[C_3\right]&=& \ell -e_{i^3_1}-\cdots -e_{i^k_{n_3+1}}\\
\end{eqnarray*}
and by lemma \ref{lem:homologyint}, $|\{i_1^j, \cdots , i^j_{n_j+1}\}\cap \{i_1^{j'}, \cdots , i^{j'}_{n_{j'+1}}\}|=1$ for $j\neq j'\in \{1,2,3\}$. There are exactly two different ways three sets can have each pairwise intersection be a unique element. The first is that they all share a single element in common, and the second is that each of the three pairs has a different common element. Up to relabeling, the two possibilities are as follows.

$$\begin{array}{rcl|rcl} 
&&\text{Case A}& & &\text{Case B} \\\hline
\left[C_0\right]&=&\ell&\left[C_0\right]&=&\ell\\
\left[C_1\right] &=& \ell -e_1-e_1^1-\cdots - e_{n_1}^1 &\left[C_1\right] &=& \ell -e_1-e_2-e_1^1-\cdots - e_{n_1-1}^1 \\
\left[ C_2 \right] &=& \ell -e_1-e_1^2- \cdots - e_{n_2}^2&\left[ C_2 \right] &=& \ell -e_1-e_3-e_1^2- \cdots - e_{n_2-1}^2\\
\left[C_3\right] &=& \ell -e_1-e_1^3-\cdots - e_{n_3}^3&\left[C_3\right] &=& \ell -e_2-e_3-e_1^3-\cdots - e_{n_3-1}^3\\
\end{array}$$

Note here all of the $e_i^j$ are all distinct from each other and from the $e_{m}$'s.

Now we can see how each of these translates into an embedding of the dual configuration. After blowing down all exceptional spheres, the proper transform of the image of the dual graph spheres under the embedding will be four symplectic spheres homologous to $\C P^1$, and by section \ref{uniqueembedding} we may assume that they are four complex projective lines. In case A, $e_1$ appears in $[C_1],[C_2]$ and $[C_3]$ with non-zero coefficient, so when we blow up the sphere representing $e_1$, it will be at a point of intersection of three of the original $+1$ spheres (the complex projective lines after blowing down must therefore be in configuration $\mathcal{I}^3_1$). The remaining blow-ups are done on a point of a single one of these three spheres, so that the resulting proper transforms have sufficiently negative self-intersection numbers.

In a Kirby diagram, the original four complex projective lines are represented by four $+1$ framed unknots with a single positive twist as in figure \ref{fig:3AStep1}. In order to ensure this is a diagram for $\C P^2$, we cancel the extra three 2-handles with three 3-handles, and close off the manifold with a 4-handle. The first blow-up in case A introduces a new $-1$ framed unknot which links three of the four original link components, untwists these three components, and reduces the framing coefficients on each by $1$. The remaining blow-ups in case A introduce more $-1$ framed unknots which link once with one of the three untwisted original link components and reduce the corresponding framing coefficient by $1$. The resulting diagram is shown in figure \ref{fig:3AStep2}.

\begin{figure}
	\centering
	\subfloat[Diagram for $\C P^2$ with four $+1$ spheres.]{
	\label{fig:3AStep1}	
	\psfragscanon
	\psfrag{T}{3 3-handles, 1 4-handle}
	\psfrag{A}{$1$}
	\psfrag{B}{$1$}
	\psfrag{C}{$1$}
	\psfrag{D}{$1$}
	\includegraphics[scale=.5]{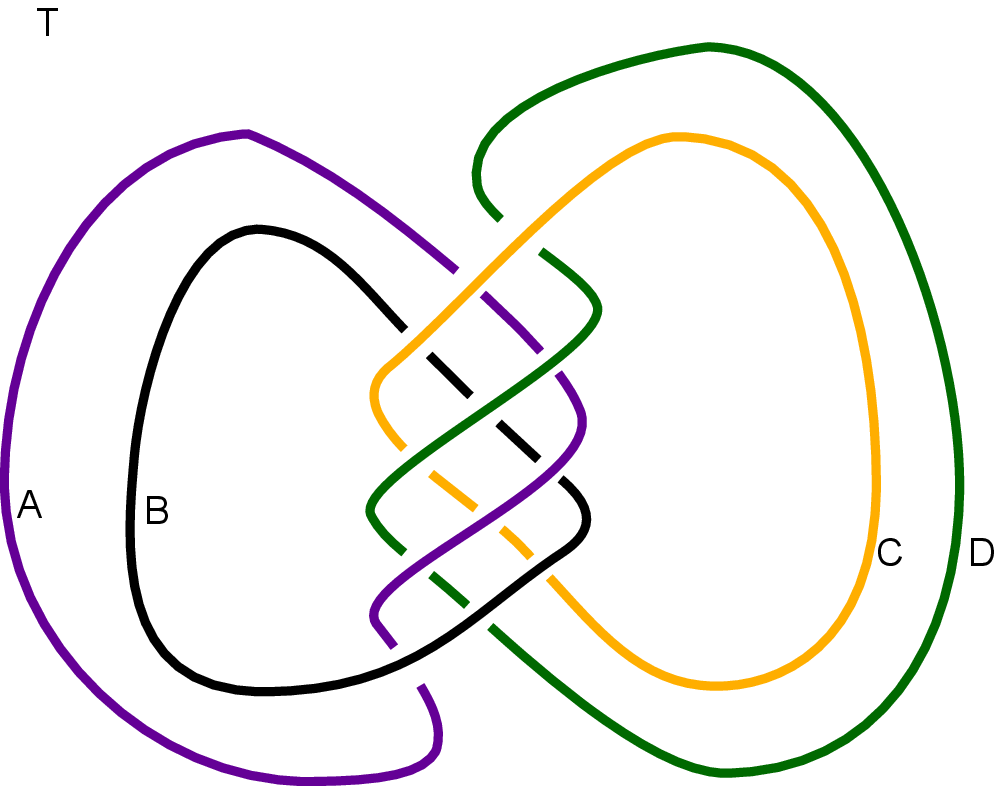}
	\psfragscanoff
	}
	\subfloat[Blow up in case A.]{
	\label{fig:3AStep2}
	\psfragscanon
	\psfrag{T}{3 3-h, 1 4-h}
	\psfrag{A}{$-n_1$}
	\psfrag{B}{$-n_2$}
	\psfrag{C}{$-n_3$}
	\psfrag{D}{$n_1$}
	\psfrag{E}{$n_2$}
	\psfrag{F}{$n_3$}
	\psfrag{G}{$1$}
	\psfrag{P}{$-1$}
	\includegraphics[scale=.55]{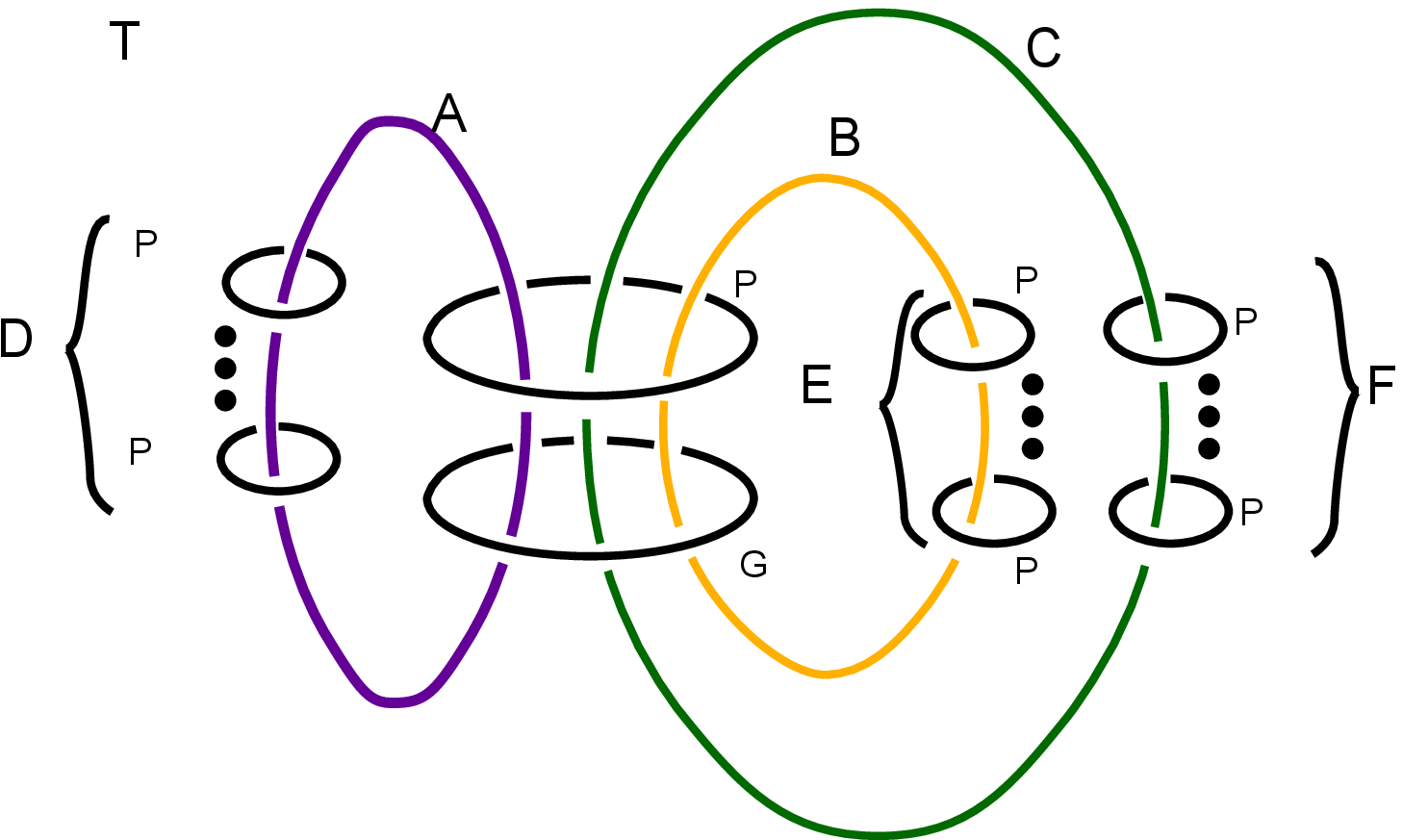}
	\psfragscanoff
	}\\
	\;\;\;\;\;\;\subfloat[Blow up in case B.]{
	\label{fig:3BStep2}
	\psfragscanon
	\psfrag{T}{3 3-h, 1 4-h}
	\psfrag{A}{$-n_1$}
	\psfrag{B}{$-n_2$}
	\psfrag{C}{$-n_3$}
	\psfrag{D}{$n_1-1$}
	\psfrag{E}{$n_2-1$}
	\psfrag{F}{$n_3-1$}
	\psfrag{G}{$1$}
	\psfrag{P}{$-1$}
	\includegraphics[scale=.6]{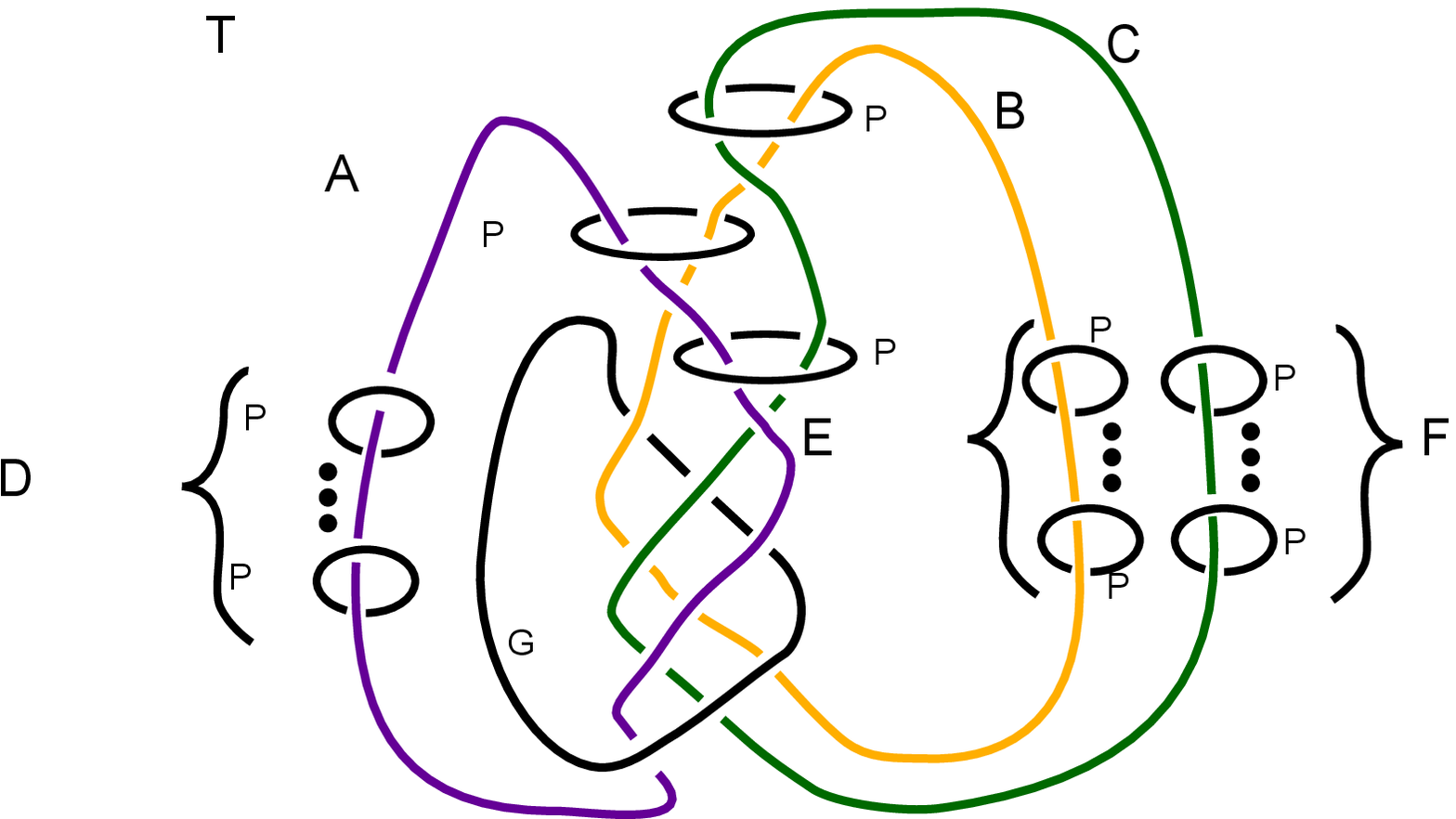}
	\psfragscanoff
	}
\caption{}
\label{fig:Embedding1}
\end{figure}

Similarly in case B, the blow-ups corresponding to $e_1,e_2,e_3$ are done at intersections of the three pairs of the three original spheres that are not $C_0$, and the other blow-ups are done at points on only one of their proper transforms. In the Kirby diagram, these are visible in figure \ref{fig:3BStep2}.

The symplectic dual configuration is visible in each of the previous diagrams for $X_M$ as the union of the cores of four 2-handles together with their Seifert surfaces pushed into the 0-handle. We wish to find the diffeomorphism types of the complements of these embeddings, since these are the potential symplectic fillings.

Because the union of the Seifert surfaces of the four attaching circles for $C_0,C_1,C_2,C_3$ is connected, their complement in the 0-handle retracts to a subset of the boundary of the 0-handle. Therefore the complement of the dual configuration is given by deleting the 0-handle and the four 2-handles corresponding to $C_0,C_1,C_2,C_3$. It is easier to understand the diffeomorphism type of the resulting manifold with boundary by turning the manifold upside down so the boundary appears on the top instead of on the bottom. Since both possible diagrams (figures \ref{fig:3AStep2} and \ref{fig:3BStep2}) have three 3-handles, the resulting upside-down handlebody in the complement of the dual configuration will have three 1-handles, together with 2-handles corresponding to all the extra 2-handles in the diagram which are not part of the dual configuration. The boundary of the 0-handle and 1-handles is $\#_3 S^2\times S^1$. This appears as a surgery diagram given by the mirror image of the original diagram with surgery coefficients the negations of the framings. Surgery coefficients are put in brackets, $\langle \cdot \rangle$. An attaching circle of an upside down 2-handle will be a 0-framed meridian of the surgery circle corresponding to the attaching circle of the original 2-handle (see \cite{GompfStipsicz} for more details on turning handlebodies upsidedown). In order to get the diagram into a more standard form, we perform diffeomorphisms on the boundary between the 1-handlebody and the 2-handles until the boundary looks standard (like 0-surgery on the three component unlink). Once the boundary of the 1-handlebody looks standard, we can replace each zero surgered unlink component with a dotted circle representing a 1-handle. The corresponding diagrams are in figures \ref{fig:3AStep3} and \ref{fig:3AStep4}. Further handleslides and 1-2 handle cancellations yield figure \ref{fig:3AStep5}.

\begin{figure}
	\centering
	\subfloat[Upsidedown complement of dual configuration in case A]{
	\label{fig:3AStep3}
	\psfragscanon
	\psfrag{A}{$\langle n_1 \rangle$}
	\psfrag{B}{$\langle n_2 \rangle$}
	\psfrag{C}{$\langle n_3 \rangle$}
	\psfrag{D}{$n_1$}
	\psfrag{E}{$n_2$}
	\psfrag{F}{$n_3$}
	\psfrag{G}{$\langle 1 \rangle$}
	\psfrag{H}{$0$}
	\psfrag{K}{$\langle -1 \rangle$}
	\psfrag{T}{3 1-h, 1 0-h}
	\includegraphics[scale=.55]{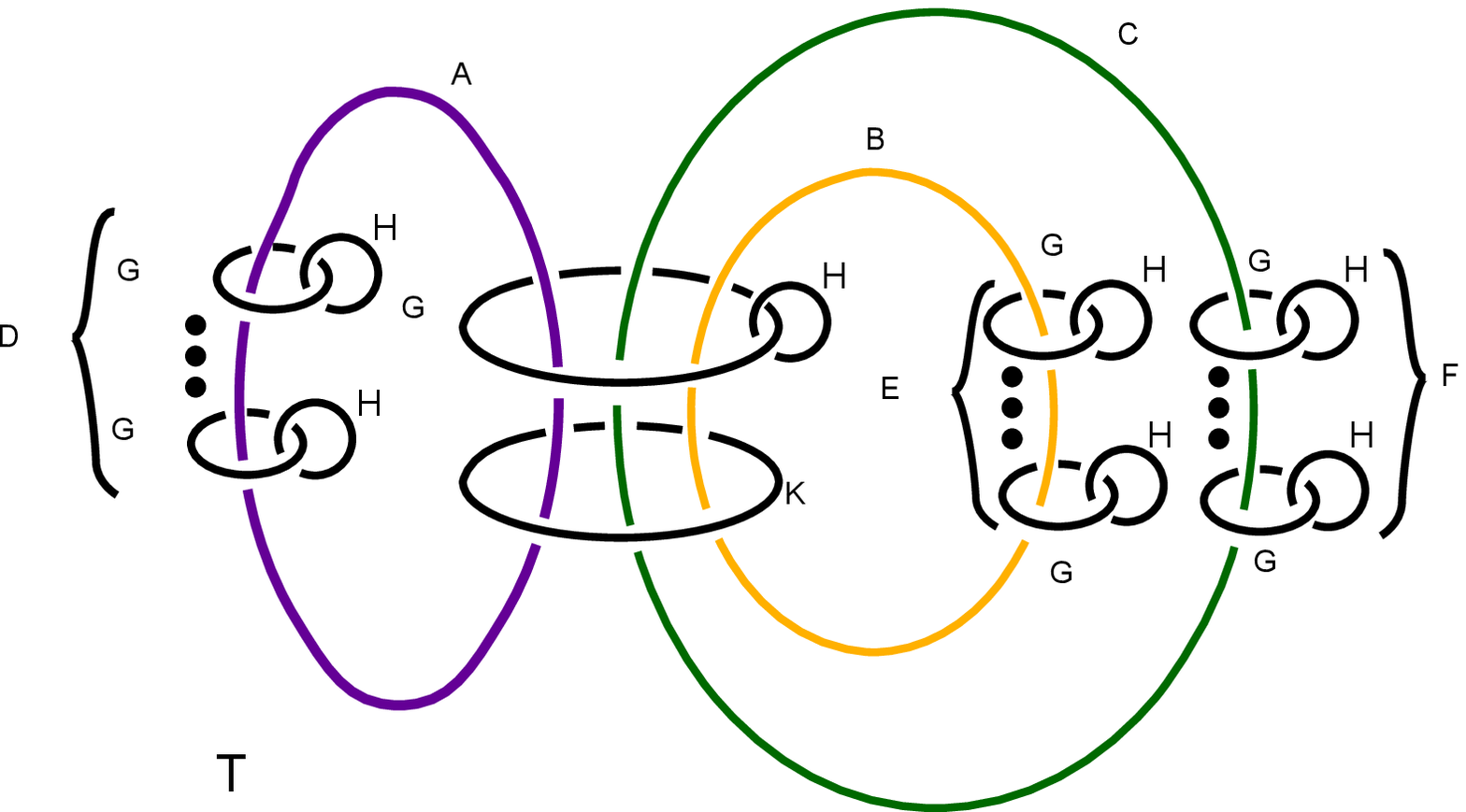}
	\psfragscanoff
	}
	\subfloat[Dotted circle notation]{
	\label{fig:3AStep4}
	\psfragscanon
	\psfrag{A}{$n_1$}
	\psfrag{B}{$n_2$}
	\psfrag{C}{$n_3$}\
	\psfrag{P}{$-1$}
	\includegraphics[scale=.55]{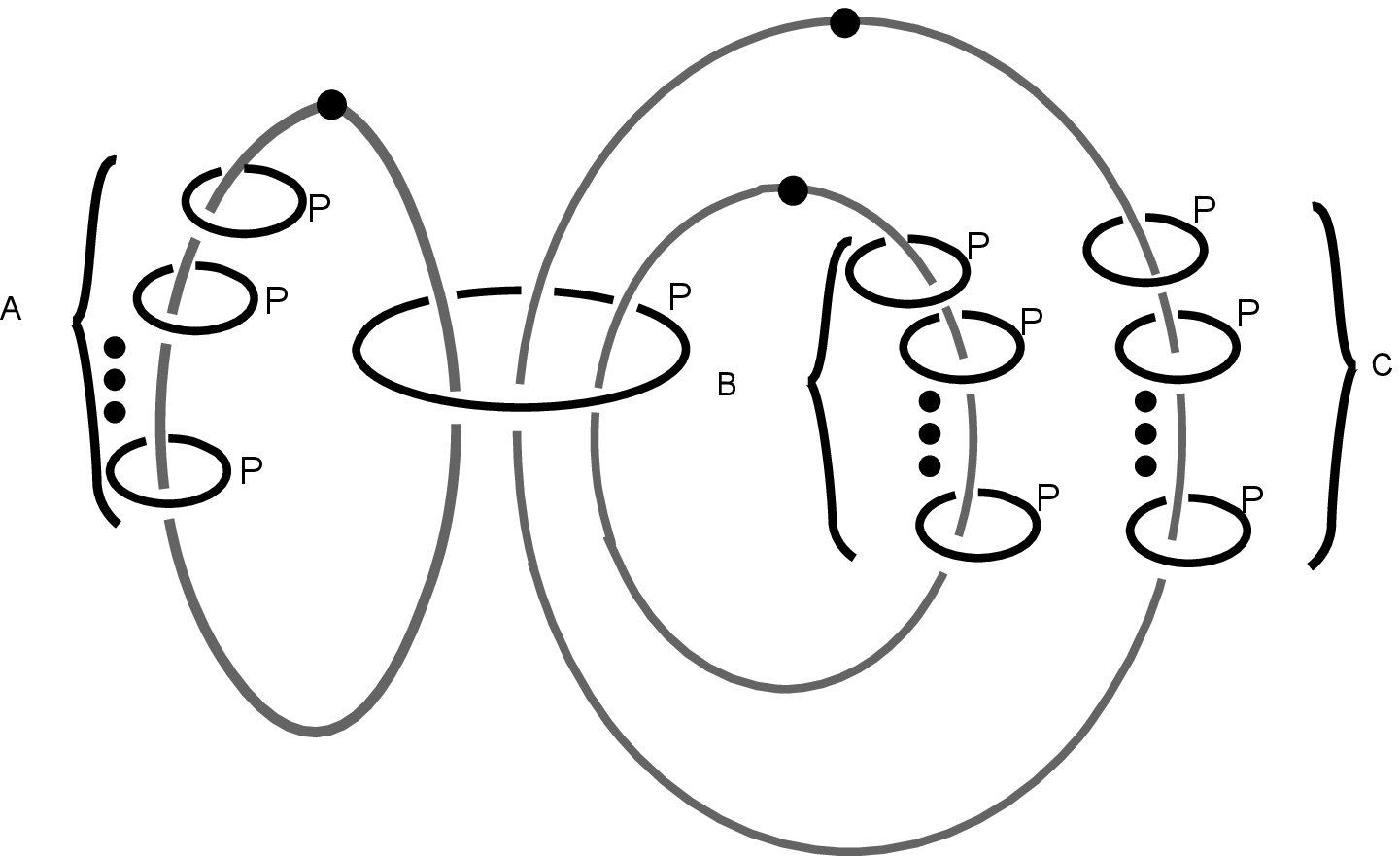}
	\psfragscanoff
	}\\
	\subfloat[Handleslides and cancellations]{
	\label{fig:3AStep5}\
	\psfragscanon
	\psfrag{A}{$n_1-1$}
	\psfrag{B}{$n_2-1$}
	\psfrag{C}{$n_3-1$}
	\psfrag{D}{$-4$}
	\psfrag{E}{$-2$}
	\includegraphics[width=6cm]{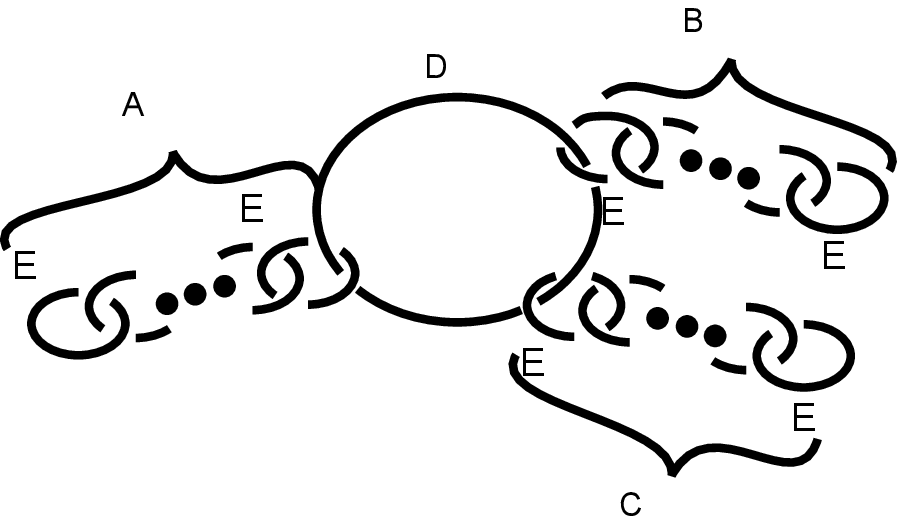}
	\psfragscanoff
	}
\caption{The diffeomorphism type of the possible symplectic filling from case A.}
 \label{fig:3AStep345}
 \end{figure}
 
Notice that the diagram in figure \ref{fig:3AStep5} is the original plumbing of spheres, which we know has a standard symplectic structure with convex boundary inducing $\xi_{pl}$ by theorem \ref{thmplumbingconvex}.
  
 In the embedding in figure \ref{fig:3BStep2} determined by case B, we find a different possible diffeomorphism type for a symplectic filling. The complement of the dual configuration turned upsidedown is given by figure \ref{fig:3BStep3}. Figures \ref{fig:3BStep4}, \ref{fig:3BStep5}, and \ref{fig:3BStep6} are obtained by surgery and handle moves.

\begin{figure}
	\begin{tabular}{cc}
	\subfloat[Upsidedown complement of dual configuration in Case B]{
	\label{fig:3BStep3}
	\psfragscanon
	\psfrag{A}{$\langle n_1 \rangle$}
	\psfrag{B}{$\langle n_2 \rangle $}
	\psfrag{C}{$\langle n_3 \rangle $}
	\psfrag{D}{$n_1-1$}
	\psfrag{E}{$n_2-1$}
	\psfrag{F}{$n_3-1$}
	\psfrag{G}{$\langle 1 \rangle $}
	\psfrag{H}{$\langle -1 \rangle $}
	\psfrag{K}{$0$}
	\psfrag{T}{1 $0$-h, 3 $1$-h}
	\includegraphics[scale=.55]{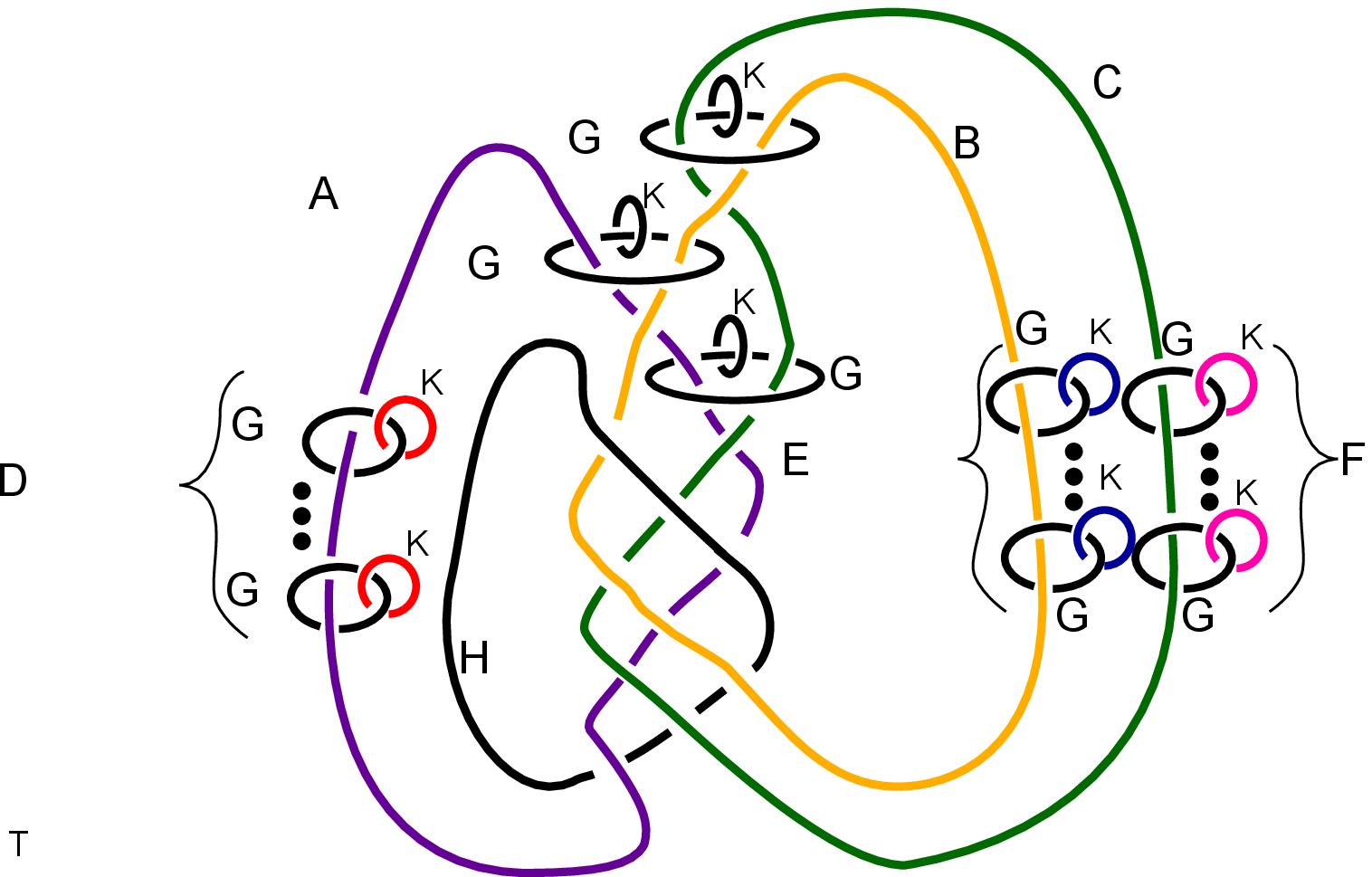}
	\psfragscanoff
	}&
	\subfloat[Moves on surgery diagram]{
	\label{fig:3BStep4}
	\psfragscanon
	\psfrag{A}{$n_1-1$}
	\psfrag{B}{$n_2-1$}
	\psfrag{C}{$n_3-1$}
	\psfrag{P}{$-1$}
	\psfrag{D}{$\langle 0 \rangle$}
	\psfrag{T}{1 $0$-h, 3 $1$-h}
	\includegraphics[scale=.55]{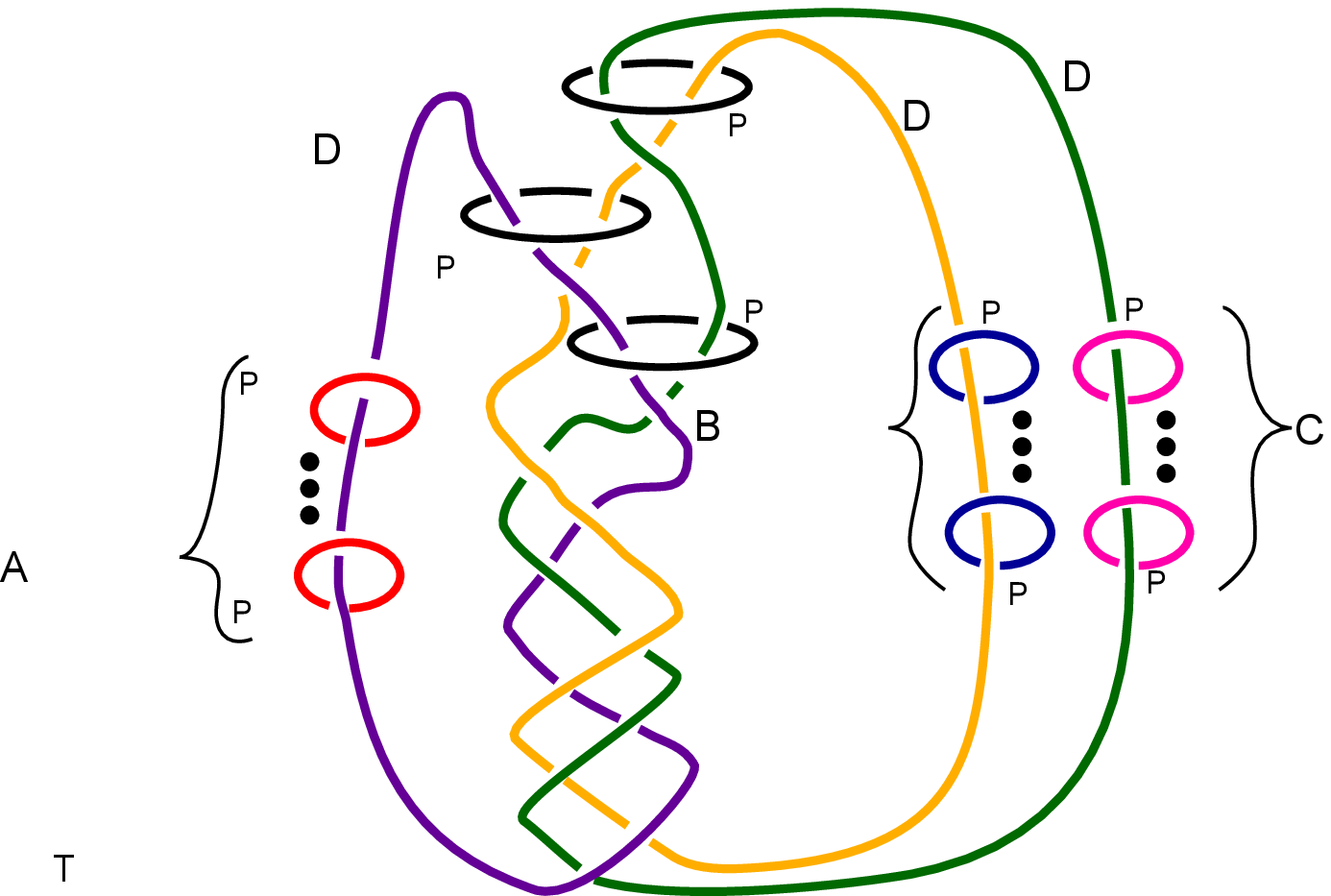}
	\psfragscanoff
	}\\
	\subfloat[Dotted circle notation]{
	\label{fig:3BStep5}
	\psfragscanon
	\psfrag{A}{$n_1-1$}
	\psfrag{B}{$n_2-1$}
	\psfrag{C}{$n_3-1$}
	\psfrag{P}{$-1$}
	\includegraphics[scale=.8]{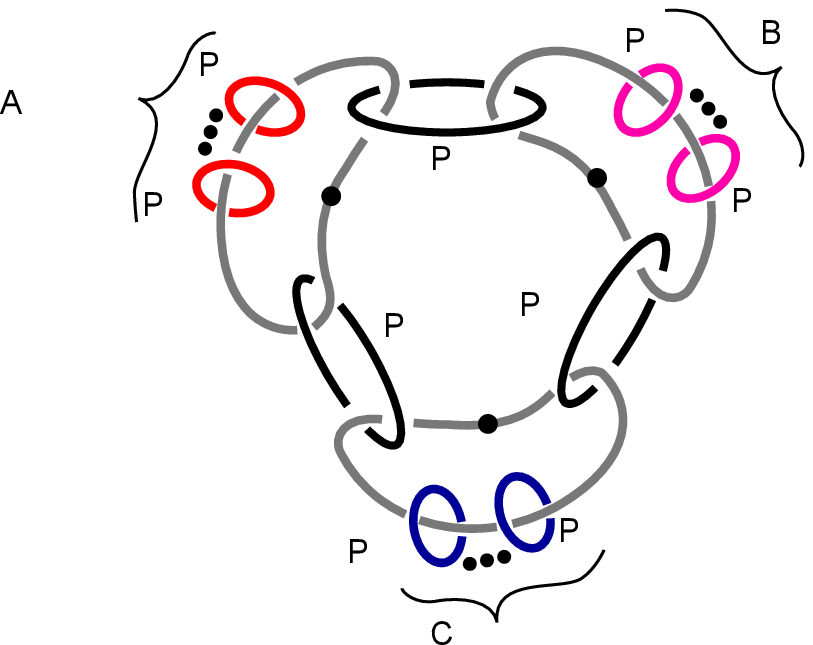}
	\psfragscanoff
	}&
	\subfloat[Handle cancellation]{
	\label{fig:3BStep6}
	\psfragscanon
	\psfrag{A}{$n_1-1$}
	\psfrag{B}{$n_2-1$}
	\psfrag{C}{$n_3-1$}
	\psfrag{D}{$-3$}
	\psfrag{E}{$-2$}
	\psfrag{F}{$-1$}
	\includegraphics[scale=.8]{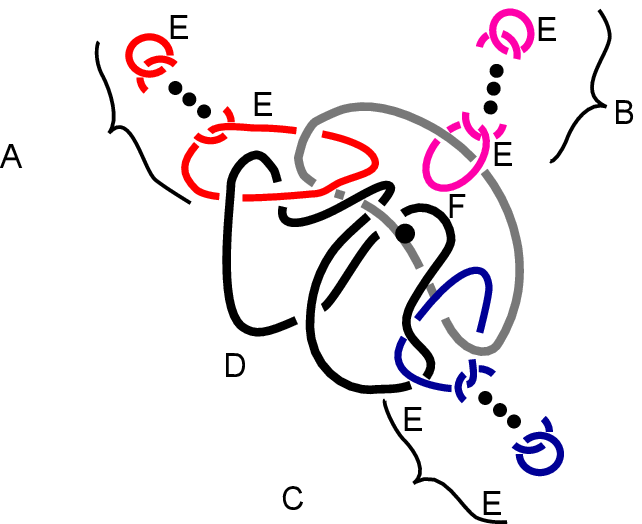}
	\psfragscanoff
	}
	\end{tabular}
\caption{The diffeomorphism type of the possible symplectic filling from case B.}
\label{fig:3BStep3456}
\end{figure}

Note that the resulting manifold in figure \ref{fig:3BStep6} for case B can be obtained from the resulting manifold in figure \ref{fig:3AStep5} by a rational blow-down of the $-4$ framed sphere (the 1-handle and the black $2$-handle form a rational homology ball). Because the symplectic structure on the manifold in figure \ref{fig:3AStep5}, has the $-4$-framed sphere as a symplectic submanifold, by an observation of Gompf \cite{GompfSymplGluing} we can cut out this $-4$-framed sphere and replace it with the rational homology ball as in figure \ref{fig:3BStep6} such that the symplectic structure extends over the rational homology ball. Note that this does not change the symplectic structure in a neighborhood of the boundary of the plumbing of spheres in figure \ref{fig:3AStep5}, so the boundary remains convex and induces the same contact structure $\xi_{pl}$.

 Alternatively we can arrange this diagram to be a Stein handlebody. If we arrange that the Kirby diagram is in a standard form such that all the 2-handles are contained in a rectangular box, the two attaching balls for each of the 1-handles are aligned on opposite sides of the box, and the 2-handles are attached along Legendrian tangles inside the box such that the framing coefficient is given by 1 less than the Thurston-Bennequin number of the attaching circle then there exists a Stein structure on this 4-manifold (by \cite{EliashbergStein}, \cite{GompfStein}). Replacing the dotted circle in figure \ref{fig:3BStep6} with two attaching balls, and keeping track of framings carefully, after isotopies we can achieve a diagram where the 2-handles are attached along Legendrian knots with framing coefficients given by $tb-1$ in figure \ref{fig:3BStep7}.
 \begin{figure}
 \begin{center}
 \includegraphics[scale=1.4]{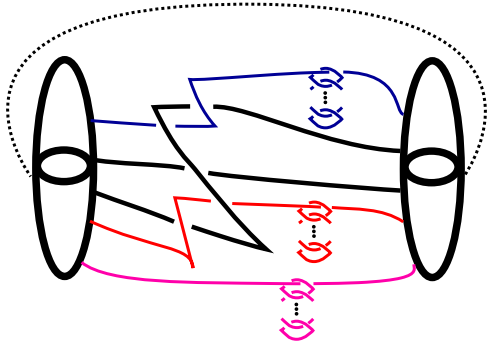}
 \end{center}
 \caption{Stein handlebody for figure \ref{fig:3BStep6}. Framings on 2-handles are given by 1 less than the Thurston-Bennequin number of the attaching circle.}
 \label{fig:3BStep7}
 \end{figure}
To verify that the contact structure is correct, we look at the classification of contact structures on these Seifert fibered spaces. In these cases, there are three distinct contact structures which are distinguished from each other by their Euler class. We can compute the Euler class of the induced contact structure on the boundary of a Stein manifold in a standard way involving rotation numbers of the attaching circles (see \cite{GompfStein}). To check this matches the Euler class of the contact manifold we started with, track $PD(e(\xi))$ through a diffeomorphism taking this diagram representing the 3-manifold to the standard one as the boundary of a star-shaped plumbing of spheres.

A third way to see the convex symplectic structure on this manifold is to view it as a Lefschetz fibration over a disk. An isotopy of the diagram in figure \ref{fig:3BStep4}, gives figure \ref{fig:3BLefschetz}. We claim this diagram represents Lefschetz fibrations which induces an open book decomposition on its boundary that support the contact structure $\xi_{pl}$. The Lefschetz fibration will have base $D^2$ and regular fibers $3$-hold disks, $D_3=D^2\setminus \{N(p_1),N(p_2) , N(p_3)\}$ where $\{N(p_i)\}$ are disjoint neighborhoods of points contained in the interior of $D^2$. We obtain a handlebody diagram for $D^2\times D_3$ from the $3$ disjoint parallel dotted circles, because one can think of a dotted circle as the removal of a 2-handle from the interior of the 0-handle. If we view the 0-handle as $D^2\times D^2$, we can view the dotted circles as removing a neighborhood of $D^2\times \{p_i\}$. Then we attach the $-1$ framed 2-handles along the boundary of $D^2\times D_3$. We can see a trivial open book decomposition with pages $D_3$ on $\bdry(D^2\times D_3)=S^1\times D_3\cup D^2\times[\bdry D^2\sqcup \bdry N(p_1) \sqcup N(p_2) \sqcup \bdry N(p_3)]$. Note that each attaching circles of a 2-handles lies in a page of this trivial open book and the Seifert framing in the handlebody diagram agrees with the page framing coming from this open book decomposition. Therefore the framing on the 2-handles is $-1$ relative to the page framing, so the attaching circles are vanishing cycles in a Lefschetz fibration.

It is useful to understand the open book decomposition supporting $\xi_{pl}$ given to us by theorem \ref{thmplumbingconvex}. The construction of Gay and Mark tells us that the pages of the open book are given by the surface obtained by connect summing $|s_j|$ copies of $D^2$ to each sphere $C_j$ in our graph, and then connect summing these surfaces according to the plumbing graph. In our case, the central sphere $C_0$ has $s_0=-4+3=-1$,  the spheres in the arms but not on the ends have $s_j=-2+2=0$, and the spheres on the ends have $s_j=-2+1=-1$. Therefore the pages are surfaces as in figure \ref{fig:page4}. The monodromy is given by a product of positive Dehn twists about the simple closed curves around each connect sum neck, shown as the blue curves in figure \ref{fig:page4}. Note that the order of these Dehn twists does not matter, because the curves are all disjoint from each other so the corresponding Dehn twists commute in the mapping class group.

\begin{figure}
\includegraphics[scale=.4]{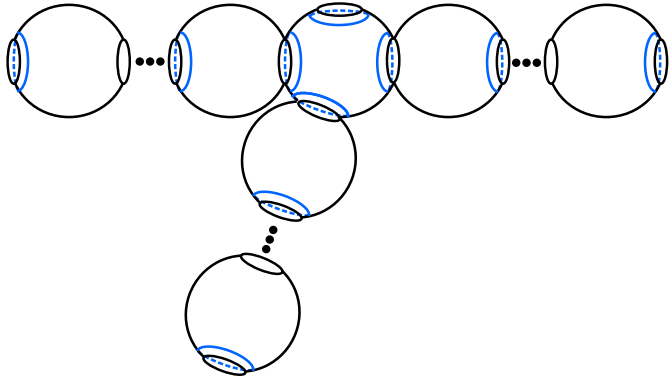}
\caption{A page of the open book decomposition supporting $\xi_{pl}$ given by \cite{GayMark}. A product of Dehn twists about the blue curves gives the monodromy.}
\label{fig:page4}
\end{figure}

\begin{figure}
\centering
\subfloat[Lefschetz fibration]{
\psfragscanon
\psfrag{A}{$n_1-1$}
\psfrag{B}{$n_2-1$}
\psfrag{C}{$n_3-1$}
\psfrag{P}{$-1$}
\includegraphics[scale=.4]{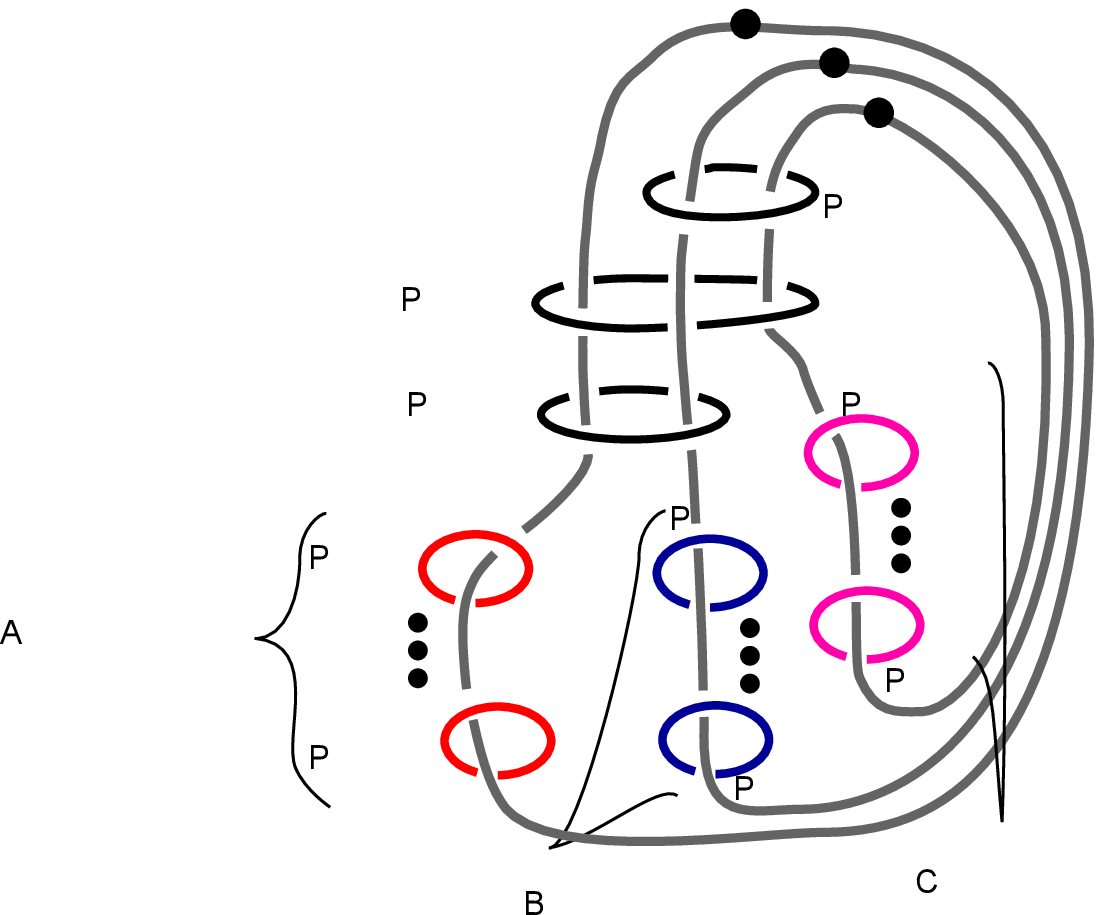}
\psfragscanoff
\label{fig:3BLefschetz}
}
\;\;\;\;\;\subfloat[Open book decomposition]{
\includegraphics[scale=.3]{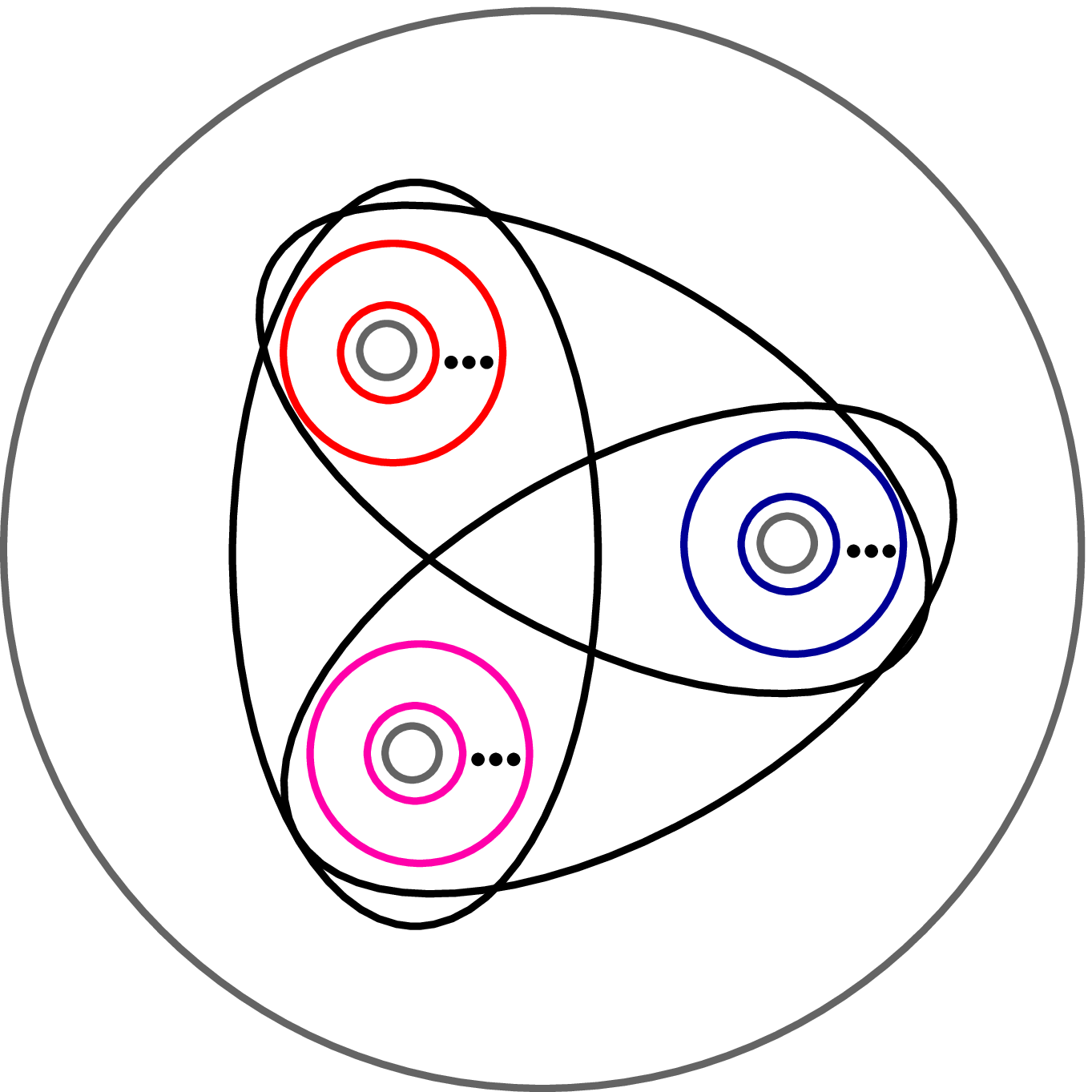}
\label{fig:3BOpenBook}
}
\;\;\;\;\;\subfloat[Equivalent open book]{
\includegraphics[scale=.3]{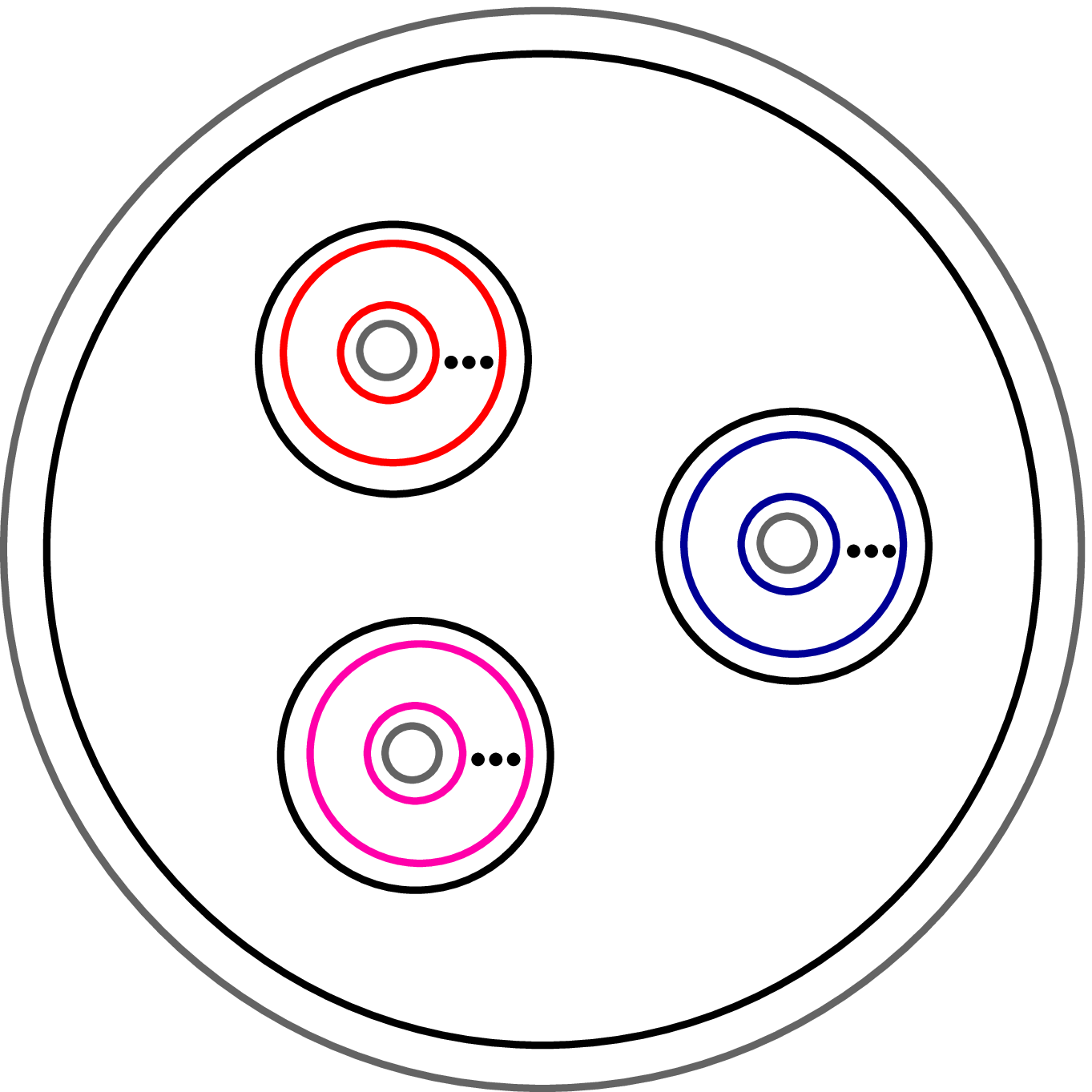}
\label{fig:3BOpenBook2}
}
\caption{The Lefschetz fibration is visible in this handlebody description of the second symplectic filling, and the open book decomposition induced on the boundary has monodromy given by a product of positive Dehn twists about the curves pictured here.}
\label{fig:3BLefschetzOBD}
\end{figure}

On the other hand, our Lefschetz fibration induces an open book decomposition on its boundary whose pages are disks with three holes and whose monodromy is a product of positive Dehn twists about the vanishing cycles ordered starting at the bottom of the picture to the top, as in figure \ref{fig:3BOpenBook}. This monodromy is equivalent by a lantern relation to positive Dehn twists about the curves in figure \ref{fig:3BOpenBook2}, which is equivalent to the open book decomposition determined by figure \ref{fig:page4} which we know supports $\xi_{pl}$. Therefore our filling has the structure of a Lefschetz fibration which induces an open book decomposition on its boundary that supports the contact structure $\xi_{pl}$ we are interested in.

A Lefschetz fibration is \emph{allowable} if the vanishing cycles are non-zero in the homology of the page. It was shown in \cite{AkbulutOzbagci1} and \cite{LoiPiergallini} that an allowable Lefschetz fibration admits a Stein structure, and verified in \cite{Plamenevskaya} that the contact structure induced on the boundary is supported by an open book decomposition whose pages are diffeomorphic to the regular fibers of the Lefschetz fibration and whose monodromy is a product of Dehn twists about the vanishing cycles.
%
Since the vanishing cycles in this Lefschetz fibration are homologically essential, this 4-manifold supports a Stein structure, inducing $\xi_{pl}$ on the boundary. The Stein structure induces a convex symplectic structure on this filling.

This proves (in multiple ways) part (1) of theorem \ref{thm:easyexamples}.
}

\section{Classifications given by rational blow-downs}
{\label{rationalclassifications}
\subsection{A simple family with $e_0\leq -k-3$}
{
\label{sec:e0neg}
We can provide a complete classification for a similar family to our simplest family in the case where $e_0$ is sufficiently negative.
\begin{proof}[Proof of theorem \ref{thm:e0neg}]
When the central vertex on our graph is labeled by $e_0\leq -k-3$ where $k$ is the number of arms in our graph, the dual graph will have $d=-e_0-1$ arms, where $d-k$ of these arms are made up of a single symplectic sphere of square $-1$. If we assume all the arms of the dual graph have length one, the possible homology classes each dual graph sphere can represent is completely determined by lemma \ref{lem:e0neg}. The graphs that these correspond to have central vertex with coefficient $e_0$ and $k$ arms, each made up of some number of spheres of square $-2$ as in figure \ref{fig:GraphDualGraphFamilye0neg}.
\begin{figure}
	\centering
	\subfloat[Graphs]{
	\label{fig:GraphFamilye0neg}
	\psfragscanon
	\psfrag{A}{$n_1-1$}
	\psfrag{B}{$n_2-1$}
	\psfrag{C}{$n_k-1$}
	\psfrag{D}{$-2$}
	\psfrag{E}{$e_0$}
	 \includegraphics[scale=.6]{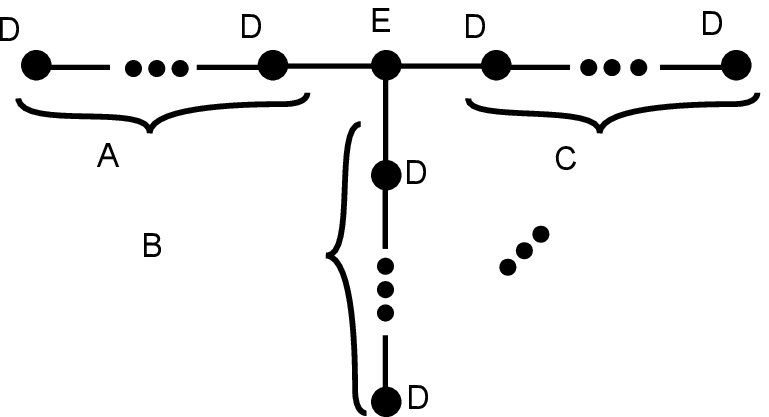}
	 \psfragscanoff
	 }
	 \;\;\;\;\;\;\;\;\;\;\;\;\subfloat[Dual graphs]{
	 \label{fig:DualGraphFamilye0neg}
	 \psfragscanon
	 \psfrag{A}{$-n_1$}
	 \psfrag{B}{$-n_2$}
	 \psfrag{C}{$-n_k$}
	 \psfrag{D}{$-1$}
	 \psfrag{E}{$+1$}
	\includegraphics[scale=.6]{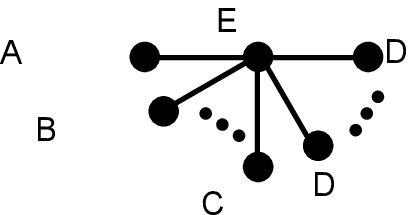}
	\psfragscanoff
	}
\caption{A simple family of graphs and their dual graphs ($e_0\leq -k-3$ and $n_1,n_2,\cdots, n_k \geq 2)$.}
\label{fig:GraphDualGraphFamilye0neg}
\end{figure}

When $\max\{n_1+1,\cdots , n_k+1\} < d-1=-e_0-2$, only one of the homology representations in lemma \ref{lem:e0neg} is possible and this corresponds to the diffeomorphism type of the original plumbing of disk bundles over spheres. When $\max\{n_1+1,\cdots , n_k+1\}\geq d-1=-e_0-2$, we have other diffeomorphism types obtained from the original plumbing of disk bundles over spheres by a rational blow-down of a linear subgraph consisting of the central vertex and the next $-e_0-4$ spheres of square $-2$ in one of the arms. Such rational blow-downs were shown to be symplectic operations by Symington in \cite{Symington1}. Since we can perform this operation on the interior of the filling, this will not change the convex symplectic structure near the boundary. Therefore we can realize all these diffeomorphism types as convex symplectic fillings.
\end{proof}
}

\subsection{A simple family whose dual graphs have long arms}
{
In the first example, there were only four spheres in the dual graph to keep track of, which restricted the number of different ways of writing their homology classes. Another restrictive condition on the possibilities for homology of the dual graph is to require that the spheres have small self-intersection numbers. If we look at three-armed dual graphs where every sphere except the central vertex has self-intersection number $-2$ as in figure \ref{fig:DualGraphFamilyB}, we can understand fillings of Seifert fibered spaces $Y(-4;-n_1,-n_2,-n_3)$ that bound plumbings according to the graphs in figure \ref{fig:GraphFamilyB}. Note that when $n_1=n_2=n_3=2$, we are back in example 1. Therefore we can assume at least one of the arms in the dual graph has length at least 2.

\begin{figure}
	\centering
	\subfloat[Graphs]{
	\label{fig:GraphFamilyB}
	\psfragscanon
	\psfrag{A}{$-n_1$}
	\psfrag{B}{$-n_2$}
	\psfrag{C}{$-n_3$}
	\psfrag{D}{$-4$}
	 \includegraphics[scale=.6]{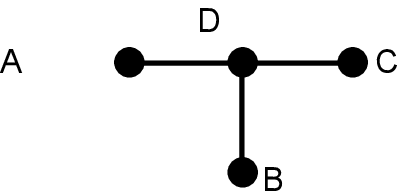}
	 \psfragscanoff
	 }
	 \;\;\;\;\;\;\;\;\;\;\;\;\subfloat[Dual graphs]{
	 \label{fig:DualGraphFamilyB}
	 \psfragscanon
	 \psfrag{A}{$n_1-1$}
	 \psfrag{B}{$n_2-1$}
	 \psfrag{C}{$n_3-1$}
	 \psfrag{D}{$-2$}
	 \psfrag{E}{$+1$}
	\includegraphics[scale=.6]{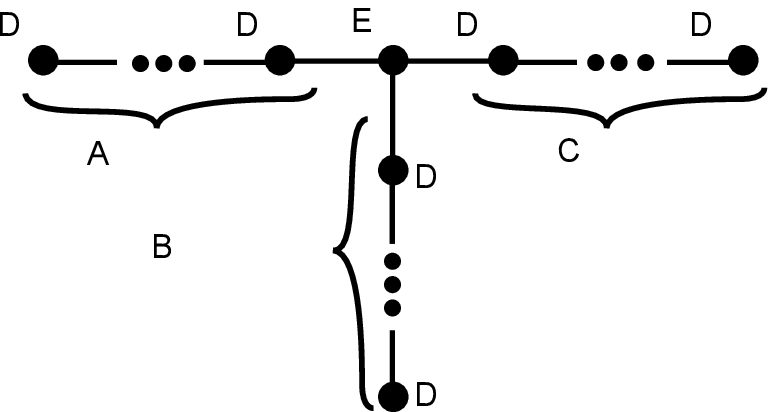}
	\psfragscanoff
	}
\caption{Another simple family of graphs and their dual graphs $(n_1,n_2,n_3\geq 1)$.}
\label{fig:GraphDualGraphFamilyB}
\end{figure}
 
 \begin{proof}[Proof of theorem \ref{thm:easyexamples} part (2)]
 
We consider what each sphere in the dual graph can represent in $H_2(X_M;\Z)$. The central vertex, and its adjacent vertices must represent one of two possible homology choices, as in the previous example. Let $C_0$ represent the central vertex, and $C_1$, $C_2$, and $C_3$ represent the adjacent vertices.
$$\begin{array}{rcl|rcl}
 &&\text{Case A} & &&\text{Case B}\\\hline
\left[C_0\right]&=&\ell&[C_0]&=&\ell\\
\left[C_1\right] &=& \ell -e_1-e_2-e_3&\left[C_1\right] &=& \ell -e_1-e_2-e_4\\
\left[C_2\right] &=& \ell - e_1-e_4-e_5&\left[C_2\right] &=& \ell - e_1-e_3-e_5\\
\left[C_3\right] &=& \ell -e_1-e_6-e_7&\left[C_3\right] &=& \ell -e_2-e_3-e_6\\
\end{array}$$

The remaining spheres in the concave cap have the form $e_{i_1}-e_{i_2}$. Suppose $C_4$ is adjacent to $C_1$. Then since $[C_1]\cdot[C_4]=1$, and $[C_j]\cdot[C_4]=0$ for $j=0,2,3$, one of the $e_i$ which has $-1$ coefficient in $[C_1]$ must have $+1$ coefficient in $[C_4]$. In case A, up to relabeling, this determines $[C_4]=e_2-e_8$ (if $e_1$ appeared in $[C_4]$ it would be impossible to cancel the algebraic intersection of $C_4$ with both $C_2$ and $C_3$ since only two $e_i$'s can appear with nonzero coefficient in $[C_4]$). Continuing with the assumption that we are in case A for now, if $C_5$ is adjacent to $C_4$ then the intersection relations imply $[C_5]=e_8-e_9$ or $[C_5]=e_3-e_2$. However if there is another sphere $C_6$ adjacent to $C_5$ then we cannot have $[C_5]=e_3-e_2$, because there is no way to write $[C_6]=e_{i_1}-e_{i_2}$ such that $[C_6]\cdot (e_3-e_2)=1$ and $[C_6]\cdot [C_1]=0$. The homology of the spheres in the other arms is determined independently in the same way. Therefore if we are in case A, and each arm in the dual graph has length at least four (i.e. $n_1,n_2,n_3>4$), the homology of the spheres is unique up to relabeling the $e_i$.

If the first four spheres are configured as in case B, and $C_4$ is adjacent to $C_1$, then the intersection relations imply $[C_4]=e_4-e_7$ or $[C_4]=e_1-e_5$ (up to symmetric relabeling). If $C_5$ is adjacent to $C_4$, then we cannot have $[C_4]= e_1-e_5$. This is because it is not possible to find $[C_5]=e_{i_1}-e_{i_2}$ such that $[C_5]\cdot (e_1-e_5)=1$, $[C_5]\cdot [C_2]=0$, and $[C_5]\cdot [C_3]=0$. Therefore if $C_5$ is adjacent to $C_4$, $[C_4]=e_4-e_7$. Furthermore $[C_5]=e_7-e_8$, since if $[C_5]=e_i-e_4$ it is not possible to ensure $[C_5]\cdot[C_1]=0$ and $[C_5]\cdot [C_2]=[C_5]\cdot[C_3]=0$ simultaneously. In conclusion, if then lengths of all of the arms in the dual graph is at least three, (i.e. $n_1,n_2,n_3>3$), the homology of all of the spheres are determined, up to obvious symmetries, by the choice that the first four spheres are as in case B.

This implies that when $n_1,n_2,n_3>4$, there are at most two diffeomorphism types of strong symplectic fillings of the Seifert fibered space arising as the boundary of the plumbing in figure \ref{fig:GraphFamilyB}. It is clear from the previous analysis in section \ref{simplestexamples} that the diffeomorphism types obtained from case A and case B differ by a rational blow-down of the central $-4$ sphere. Furthermore, we know that the original symplectic plumbing and the rational blow-down of the $-4$ sphere provide two non-diffeomorphic symplectic fillings.
\end{proof}
}

\subsection{A special case}
{
When some of $n_1,n_2,$ and $n_3$ take values $3$ or $4$, some more interesting fillings can appear. The most interesting case is when $n_1=n_2=n_3=3$, which we discuss here. (When some of the $n_j=4$, one can perform rational blow-downs of any disjoint collection of $-4$ spheres. If only one or two of the $n_j=3$, nothing new appears.)

 \begin{proof}[Proof of theorem \ref{thm:4333}]
When $n_1=n_2=n_3=3$, each arm in the dual graph has length $2$. Following the analysis of the previous proof, the four spheres $C_0,C_1,C_2,C_3$ represent in homology one of two possibilities refered to as case A and case B. There are three more spheres in the dual graph, $C_4,C_5,C_6$ adjacent to $C_1,C_2,C_3$ respectively. As described above, if the first four spheres represent the homology in case A, the homology classes of $C_4,C_5,C_6$ are determined up to symmetries. The resulting embedding determined by this homology will have complement diffeomorphic to the original plumbing of spheres. 

If the first four spheres are configured as in case B, there are two possibilities for each adjacent sphere. However, because $[C_4]\cdot [C_5]=[C_5]\cdot[C_6]=[C_4]\cdot[C_6]=0$, these choices cannot be made independently of each other. Up to permutation of the indices there are only two possibilities if the first four spheres represent the homology as in case B, given as follows.
$$\begin{array}{rcl|rcl}
\left[C_4\right] &=& e_4-e_7&\left[C_4\right] &=& e_1-e_5\\
\left[C_5\right] &=& e_5-e_8&\left[C_5\right] &=& e_3-e_6\\
\left[C_6\right] &=& e_6-e_9&\left[C_6\right] &=& e_2-e_4\\
\end{array}$$
The first of these gives rise to the rational blow-down of the $-4$ sphere in the original plumbing. The second gives rise to a rational blow-down of the entire original plumbing. The embedding of the concave cap into $\C P^2\#6\overline{\C P^2}$ is visible in figure \ref{fig:3334Step1}. By turning the complement upside down, one obtains a manifold built from a 0-handle, three 1-handles, and three 2-handles, as in figure \ref{fig:3334Step2}. One can cancel two of the 1-handles with two of the 2-handles, and the remaining 2-handle links nine times with the 1-handle, and is thus a diagram for a rational homology ball.

\begin{figure}
	\centering
	\subfloat[Embedding into $\C P^2\#6\overline{\C P^2}$]{
	\label{fig:3334Step1}
	\psfragscanon
	\psfrag{T}{3 $3$-h, 1 $4$-h}
	\psfrag{X}{$-2$}
	\psfrag{Y}{$-1$}
	\psfrag{Z}{$+1$}
	\includegraphics[scale=.5]{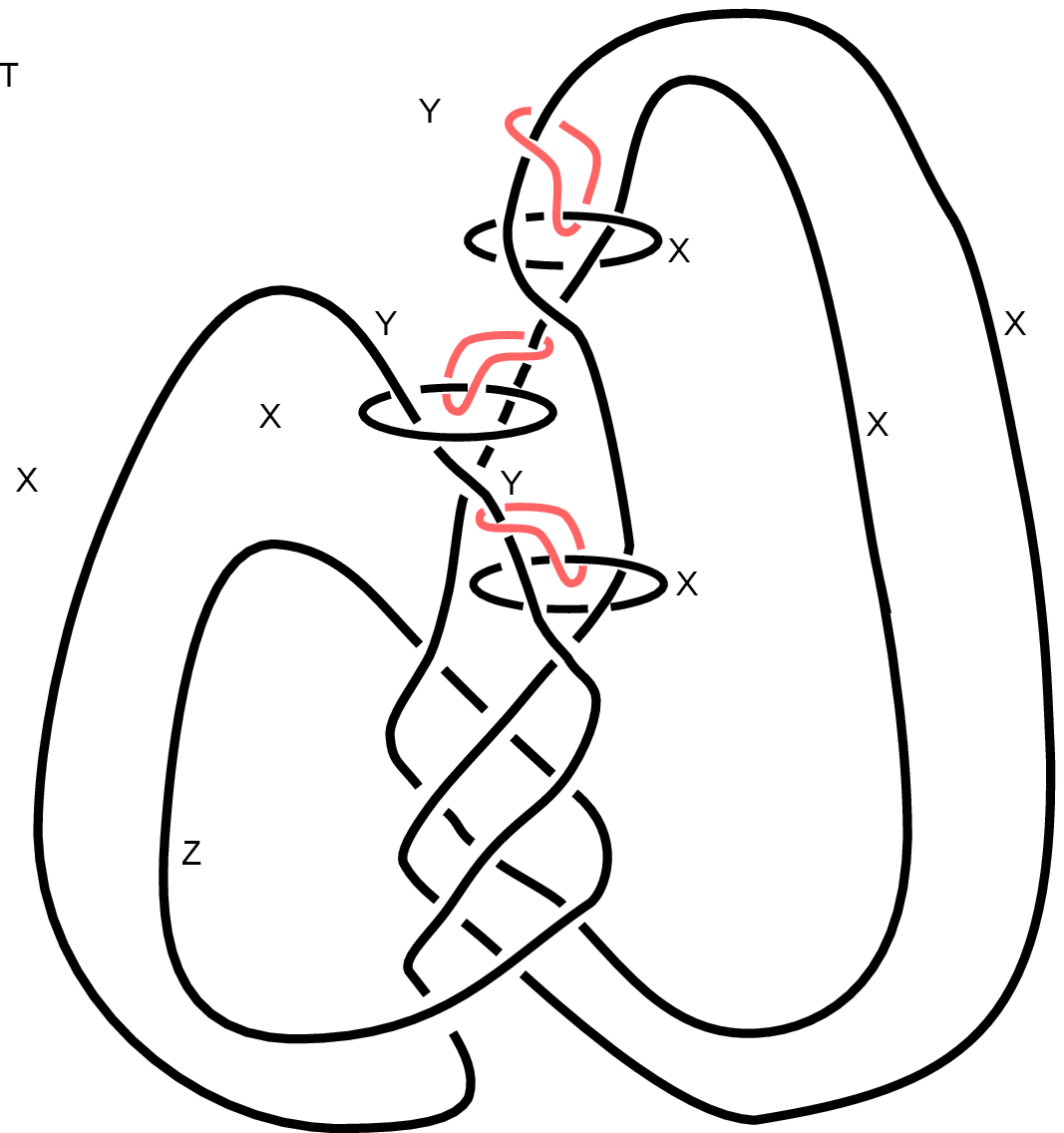}
	\psfragscanoff
	}
	\;\;\;\;\;\;\subfloat[Upside down complement of dual configuration]{
	\label{fig:3334Step2}
	\psfragscanon
	\psfrag{A}{$-1$}
	\includegraphics[scale=.5]{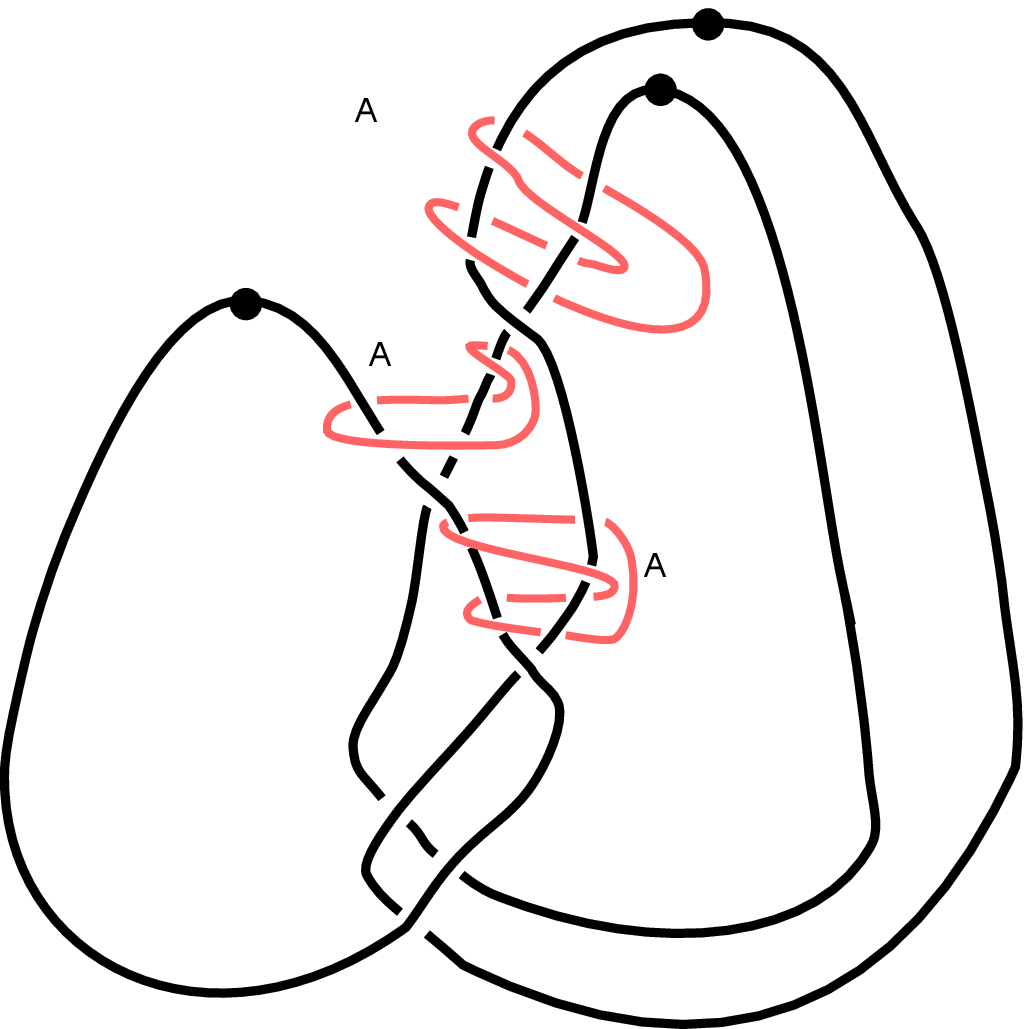}
	\psfragscanoff
	}
\caption{The embedding of the concave cap and its complement which is a rational homology ball.}
\label{fig:3334}
\end{figure}

It is known that a smoothing of a complex singularity gives this rational homology ball whose boundary is $Y(-4;-3,-3,-3)$. The complex structure gives a convex symplectic structure to this smooth manifold, filling the same contact structure $\xi_{pl}$ (by \cite{StipsiczSzaboWahl} and \cite{ParkStipsicz}). Therefore all three of these diffeomorphism types must be strong symplectic fillings of $(Y(-4;-3,-3,-3),\xi_{pl})$.
\end{proof}
}

\subsection{The family $\mathcal{W}_{p,q,r}$}
{
The special case in the previous example generalizes to a family of dually positive symplectic plumbings of spheres that can be completely rationally blown down, given by the graphs in figure \ref{fig:GraphExamplesWpqr}. This is the only family of dually positive graphs which has a symplectic rational blowdown (of the entire configuration) due to the classifications in \cite{BhupalStipsicz} and \cite{StipsiczSzaboWahl}.

\begin{figure}
	\centering
	\subfloat[Graphs $\mathcal{W}_{p,q,r}$]{
	\label{fig:GraphExamplesWpqr}
	\psfragscanon
	\psfrag{A}{$p$}
	\psfrag{B}{$q$}
	\psfrag{C}{$r$}
	\psfrag{D}{$-p-3$}
	\psfrag{E}{$-q-3$}
	\psfrag{F}{$-r-3$}
	\psfrag{G}{$-2$}
	\psfrag{H}{$-4$}
	\includegraphics[scale=.5]{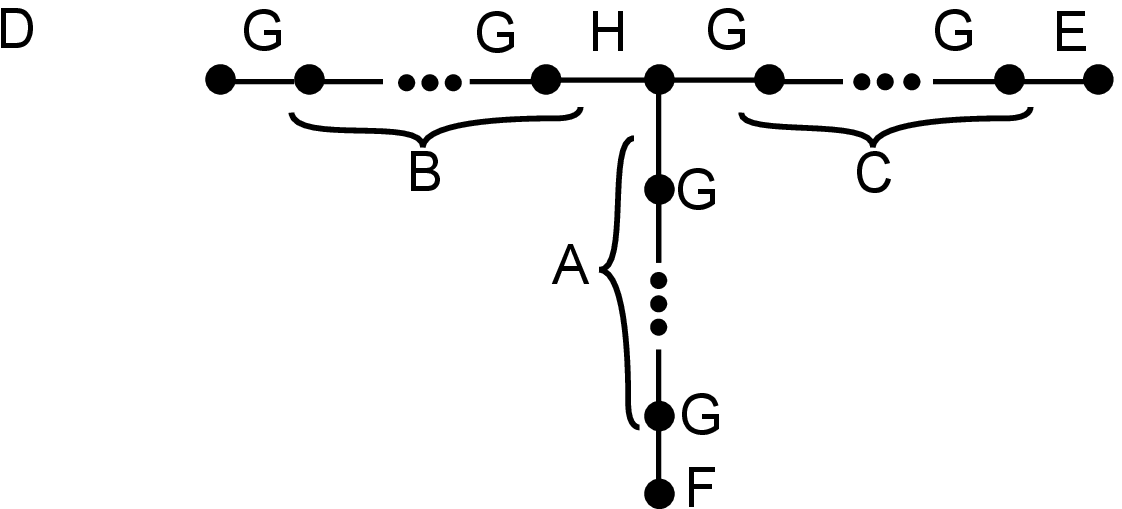}
	\psfragscanoff
	}
	\;\;\;\;\;\;\;\;\subfloat[Dual graphs]{
	\label{fig:DualGraphExamplesWpqr}
	\psfragscanon
	\psfrag{A}{$p+1$}
	\psfrag{B}{$q+1$}
	\psfrag{C}{$r+1$}
	\psfrag{D}{$-p-2$}
	\psfrag{E}{$-q-2$}
	\psfrag{F}{$-r-2$}
	\psfrag{G}{$-2$}
	\psfrag{H}{$1$}
	\includegraphics[scale=.5]{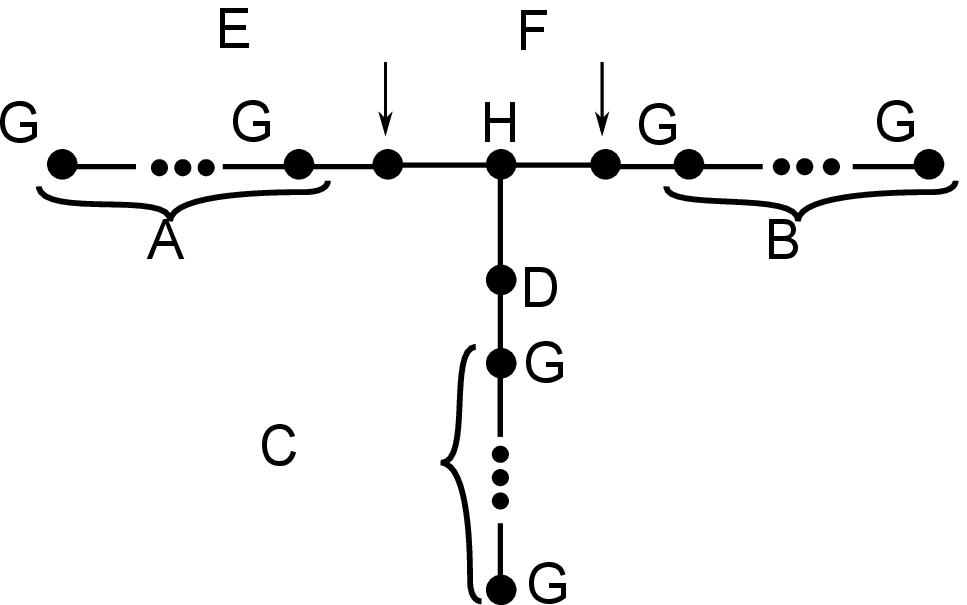}
	\psfragscanoff
	}
\caption{$\mathcal{W}_{p,q,r}$ and dual graphs.}
\label{fig:GraphDualGraphWpqr}
\end{figure}

We can classify the convex symplectic fillings completely for these graphs. Let $Y(\mathcal{W}_{p,q,r})$  denote the boundary of the plumbing of spheres according to the graph $\mathcal{W}_{p,q,r}$.
\begin{theorem}
The convex symplectic fillings of $(Y(\mathcal{W}_{p,q,r}),\xi_{pl})$ are of the following diffeomorphism types:
\begin{enumerate}
\item The original symplectic plumbing of spheres according to the graph $\mathcal{W}_{p,q,r}$.
\item A rational blow-down of the central $-4$ sphere in the original plumbing of spheres.
\item A rational blow-down of a subset of the spheres in the first arm, the first with square $-p-3$ and the next $(p-1)$ spheres with square $-2$, (\emph{assuming $p-1\leq q$}).
\item A rational blow-down of a subset of the spheres in the second arm, the first with square $-r-3$ and the next $(r-1)$ spheres with square $-2$, (\emph{assuming $r-1\leq p$}).
\item A rational blow-down of a subset of the spheres in the third arm, the first with square $-q-3$ and the next $(q-1)$ spheres with square $-2$, (\emph{assuming $q-1\leq r$}).
\item Any combination of (3),(4), and/or (5) assuming all the necessary hypotheses given above on $p,q,r$ are met. Also, any combination of (3),(4), and (5) with (2), but in that case we require the stronger conditions on (3),(4), and (5) that $p\leq q$, $r\leq p$ and $q\leq r$ respectively (this ensures the rational blow-downs can all be done disjointly).
\item A rational blow-down of the entire graph.
\end{enumerate}
\label{thm:Wpqr}
\end{theorem}

\begin{proof}
First note that all the above diffeomorphism types are realized as convex symplectic fillings of $\xi_{pl}$ because these rational blow-downs are known to be symplectic operations by \cite{Symington1}, \cite{StipsiczSzaboWahl}. Furthermore, all of these rational blow-downs produce non-diffeomorphic manifolds  which can be distinguished by their intersection forms (except when there are obvious symmetries of the three arms). Therefore it suffices to provide an upper bound on the number of convex fillings which matches the number of diffeomorphism types provided in the statement of the theorem.

By section \ref{mainargument}, an upper bound is given by the number of ways to represent the homology classes of the spheres in the dual graph in terms of a standard basis for $H_2(\C P^2\# M\overline{\C P^2})$ of the form given in section \ref{sec:homology}. The dual graph for $\mathcal{W}_{p,q,r}$ is given in figure \ref{fig:DualGraphExamplesWpqr}.

Denote the sphere representing the central vertex by $C_0$, the three spheres intersecting $C_0$ by $C_1,C_2,C_3$. Denote the string of spheres of square $-2$ whose first sphere is adjacent to $C_j$ by $C_1^j,\cdots , C_{p+1}^j$ for $j=1,2,3$. There are two possibilities for the homology classes of the first four spheres, as in previous computations. Throughout this computation, all $e_i^n$ will be distinct basis elements for $H_2(\C P^2\#M\overline{\C P^2};\Z)$ of square $-1$.
$$\begin{array}{rcl|rcl}
&&\text{Case A} & &&\text{Case B}\\\hline
\left[C_0\right]&=&\ell& \left[C_0\right]&=&\ell\\
\left[C_1\right] &=& \ell -e_1^0-e_1^1-\cdots -e_{q+2}^1&\left[C_1\right] &=& \ell -e_1^0-e_2^0-e_1^1-\cdots -e_{q+1}^1\\
\left[C_2\right] &=& \ell -e_1^0-e_1^2-\cdots -e_{p+2}^2&\left[C_2\right] &=& \ell -e_1^0-e_3^0-e_1^2-\cdots -e_{p+1}^2\\
\left[C_3\right] &=& \ell -e_1^0-e_1^3-\cdots -e_{r+2}^3&\left[C_3\right] &=& \ell -e_2^0-e_3^0-e_1^3-\cdots -e_{r+1}^3\\
\end{array}$$
The remaining strings of spheres each have two possible configurations that can occur in either case A or B. They can each occur independently of each other, but one of each of these two choices requires some inequality between $p,q,r$ to be true. The configurations are:
$$\begin{array} {rcl|rcl|rcl}
\left[C_1^1\right] &=& e_1^1-e_1^4 & \left[C_1^2\right] &=& e_1^2-e_1^5 & \left[C_1^3\right] &=& e_1^3-e_1^6\\
\left[C_2^1\right] &=& e_1^4-e_2^4 & \left[C_2^2\right] &=& e_1^5-e_2^5 & \left[C_2^3\right] &=& e_1^6-e_2^6\\
&\vdots& &&\vdots & &\vdots &\\
\left[C_{p+1}^1\right] &=& e_p^4-e_{p+1}^4 &\left[C_{r+1}^2\right] &=& e_r^5-e_{r+1}^5&\left[C_{q+1}^3\right] &=& e_q^6-e_{q+1}^6\\\hline
&\text{or}& & &\text{or}& & &\text{or}& \\\hline
\left[C_1^1\right] &=& e_1^1-e_1^4 & \left[C_1^2\right] &=& e_1^2-e_1^5 & \left[C_1^3\right] &=& e_1^3-e_1^6\\
\left[C_2^1\right] &=& e_2^1-e_1^1 & \left[C_2^2\right] &=& e_2^2-e_1^2 & \left[C_2^3\right] &=& e_2^3-e_1^3\\
&\vdots& &&\vdots & &\vdots &\\
\left[C_{p+1}^1\right] &=& e_{p+1}^1-e_{p}^1 &\left[C_{r+1}^2\right] &=& e_{r+1}^2-e_{r}^2&\left[C_{q+1}^3\right] &=& e_{q+1}^3-e_{q}^3.\\
\end{array}$$
Note that the bottom choices are only possible if there are enough $e_i^n$ for $n=1,2,3$, namely we need $p+1\leq q+2$, $r+1\leq p+2$, or $q+1\leq r+2$ if we want the bottom choice for the homology classes of $\{C_j^1\}$, $\{C_j^2\}$ or $\{C_j^3\}$ respectively, and the first four spheres represent the homology given by case A. We need $p+1\leq q+1$, $r+1\leq p+1$, or $q+1\leq r+1$ if we want the bottom choice for the homology classes of $\{C_j^1\}$, $\{C_j^2\}$ or $\{C_j^3\}$ respectively, and the first four spheres represent the homology given by case B.

As in all previous examples, the effect of choosing the first four spheres to represent the homology in case B versus case A is to rationally blow-down the central $-4$ sphere. The conditions for when one can choose the lower representation of the homology of the $j^{th}$ arm of the dual graph match up precisely with the conditions for when one can rationally blow down a linear subgraph of the $j^{th}$ arm in the original graph in case A. In case B, these conditions on when we have a second choice for the homology of the $j^{th}$ arm of the dual graph match the conditions needed to rationally blow down a linear subgraph of the $j^{th}$ arm disjointly from the central $-4$ sphere.

Additionally, the symmetries which make the rational blow-down of one arm diffeomorphic to the rational blow-down of another arm, correspond to symmetries in the $e_i^j$ which permute the values of $j$.

There is one additional way to represent the homology of the arms in the dual graphs, when the first four spheres represent homology given by case B. We know that $[C_1^j]=e_{i_1}^{m_1}-e_{i_2}^{m_2}$ where $e_{i_1}^{m_1}$ must show up with coefficient $-1$ in $[C_j]$. In case A, we cannot have $e_{i_1}^{m_1}=e_1^0$ and still have $[C_1^j]\cdot [C_{j'}]=0$ for $j\neq j'\in \{1,2,3\}$. However, in case B, this can occur but it uniquely determines the remaining homology classes as follows.
$$\begin{array} {rcl|rcl|rcl}
\left[C_1^1\right] &=& e_1^0-e_1^2 & \left[C_1^2\right] &=& e_3^0-e_1^3 & \left[C_1^3\right] &=& e_2^0-e_1^1\\
\left[C_2^1\right] &=& e_1^2-e_2^2 & \left[C_2^2\right] &=& e_1^3-e_2^3 & \left[C_2^3\right] &=& e_1^1-e_2^1\\
&\vdots& &&\vdots & &\vdots &\\
\left[C_{p+1}^1\right] &=& e_{p}^2-e_{p+1}^2 &\left[C_{r+1}^2\right] &=& e_r^3-e_{r+1}^3&\left[C_{q+1}^3\right] &=& e_q^1-e_{q+1}^1\\
\end{array}$$
This gives an embedding of the dual configuration of symplectic spheres into $\C P^2\# (p+q+r+6)\overline{\C P^2}$. Note that the dual graph has $p+q+r+7$ vertices, so a regular neighborhood $N$ of the corresponding configuration of spheres has $b_2(N)=p+q+r+7=b_2(\C P^2\# (p+q+r+6)\overline{\C P^2}$. The first and second homology of $\bdry N=Y(\mathcal{W}_{p,q,r})$ are both torsion, and $N$ and $\C P^2\# (p+q+r+6)\overline{\C P^2}$ are simply connected, so their first homologies are zero. Therefore, the Mayer-Vietoris theorem implies that the first and second homologies of the complement of $N$ in this embedding are both torsion. Therefore this is the diffeomorphism type of a rational homology ball. Since this is the unique possible rational homology ball which can strongly symplectically fill its contact boundary, it must be diffeomorphic to the smoothing of the normal surface singularity studied in \cite{StipsiczSzaboWahl}.

Therefore the number of ways to represent the homology of the spheres in the dual graph is in direct correspondence with the diffeomorphism types we can realize by starting with the original symplectic plumbing, and rationally blowing down a subgraph or the entire graph, so these are all possible convex symplectic fillings of $(Y,\xi_{pl})$.
\end{proof}

}
}

\section{Fillings of Seifert fibered spaces with more than three singular fibers}
{
\label{newexamples}
We have completed the classification of the simplest cases of symplectic fillings of Seifert fibered spaces with $k>3$ singular fibers when $e_0\leq -k-3$ in section \ref{sec:e0neg}. Here we consider the simplest cases when $k>3$, but $e_0=-k-1$ and $e_0=-k-2$. We are prevented from giving full classifications for all such examples because we need lemmas \ref{lemma:4spherehtpy} and \ref{lemma:configconnected} to say that each homology representation of the spheres in the dual graph uniquely determines a symplectic filling. These lemmas do not apply when $e_0=-k-1$ or $-k-2$ except when $k\in\{3,4,5\}$. However, we can still get a hold on what happens when $k=4,5$, and the resulting fillings provide new symplectic operations as opposed to the previous examples.

We look at the case where $e_0=-k-1$ with $k=4$ or $5$ arms, with coefficients as in figure \ref{fig:GraphExampleskarms}. The boundary of the corresponding plumbing is $Y(-k-1;\frac{-n_1}{n_1-1},\cdots, \frac{-n_k}{n_k-1})$. The dual graph is in figure \ref{fig:DualGraphExampleskarms}.

\begin{figure}[H]
	\centering
	\subfloat[Graphs]{
	\label{fig:GraphExampleskarms}
	\psfragscanon
	\psfrag{A}{$n_1-1$}
	\psfrag{B}{$n_2-1$}
	\psfrag{C}{$n_3-1$}
	\psfrag{D}{$n_k-1$}
	\psfrag{E}{$-k-1$}
	\psfrag{F}{$-2$}
	\includegraphics[width=4.5cm]{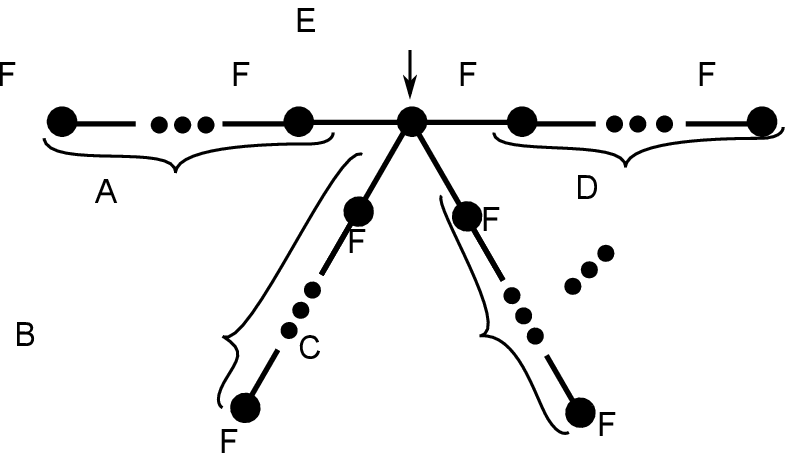}
	\psfragscanoff
	}
	\;\;\;\;\;\;\;\;\;\subfloat[Dual graphs]{
	\label{fig:DualGraphExampleskarms}
	\psfragscanon
	\psfrag{A}{$-n_1$}
	\psfrag{B}{$-n_2$}
	\psfrag{C}{$-n_3$}
	\psfrag{D}{$-n_k$}
	\psfrag{E}{$1$}
	\includegraphics[width=3.5cm]{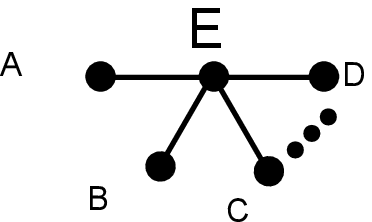}
	\psfragscanoff
	}
\caption{The dually positive graphs for which the boundary of the corresponding plumbing is $Y(-k-1;\frac{-n_1}{n_1-1},\cdots, \frac{-n_k}{n_k-1})$ and their dual graphs.}
\label{fig:karms}
\end{figure}

By gluing an arbitrary convex filling to the concave neighborhood of the dual graph, we obtain an embedding of $C_0\cup\cdots\cup C_k$ into $\C P^2\#M\overline{\C P^2}$. By lemma \ref{lem:homologyint}, the homology of these spheres in terms of the standard basis for $H_2(\C P^2\#M\overline{\C P^2};\Z)$ has the following form:
\begin{eqnarray*}
[C_0]&=& \ell\\
\left[C_1\right] &=& \ell-e_{i_1^1}-\cdots -e_{i_{n_1+1}^1}\\
&\vdots &\\
\left[C_k\right] & =& \ell -e_{i_1^k}-\cdots -e_{i_{n_k+1}^k}\\
\end{eqnarray*}
and $|\{i_1^j,\cdots , i_{n_j+1}^j\}\cap \{i_1^{j'},\cdots , i_{n_{j'}+1}^{j'}\}|=1$ for all $j\neq j' \in \{1,\cdots , k\}$. When $k=4$ this can occur in $3$ different ways, and when $k=5$ this can occur in $5$ different ways (up to symmetries). After blowing down the corresponding embeddings into $\C P^2\# M\overline{\C P^2}$, these dual graphs descend to the configurations $\mathcal{I}_1^4$, $\mathcal{I}_2^4$, and $\mathcal{I}_4^4$ when $k=4$, and the configurations $\mathcal{I}_1^5$, $\mathcal{I}_2^5$, $\mathcal{I}_5^5$, $\mathcal{I}_6^5$, and $\mathcal{I}_7^5$ when $k=5$ (see figure \ref{fig:configurations}). The spheres which share a common intersection point after blowing down, will share a common $e_i$ with coefficient $-1$ in their homology representations in the blow-up.

Note that if the $n_j$ are not sufficiently large, it may not be possible to realize all possible homology configurations. It suffices to assume $n_j\geq k-2$ for all $j\in \{1,\cdots , k\}$ to obtain all possibilities. It will be clear which homology representations and thus diffeomorphism types of fillings cannot be realized when the $n_j$ are smaller.

We first produce diagrams which depict the embeddings of the dual graph into $\C P^2 \# M \overline{\C P^2}$ that correspond to these different representations of the homology, and then use them to generate diagrams of the complement of these embeddings. Then we show that these diffeomorphism types have the structure of an appropriate Lefschetz fibration to prove that these diffeomorphism types are actually realized as convex symplectic fillings.

The result of blowing up according to one of the possible homology representations gives an embedding represented by one of the choices in figures \ref{fig:embedding4arms} and \ref{fig:embedding5arms}.  The attaching circles of the $k+1$ spheres in the dual configuration which are proper transforms of the original $k+1$ complex projective lines are shown in red, and the attaching circles for the 2-handles corresponding to the exceptional spheres are shown in black. 

\begin{figure}
\centering
\subfloat[Embeddings of dual graph]{
\psfragscanon
\psfrag{T}{4 $3$-h, 1 $4$-h}
\psfrag{A}{$-n_1$}
\psfrag{B}{$-n_2$}
\psfrag{C}{$-n_3$}
\psfrag{D}{$-n_4$}
\psfrag{E}{$+1$}
\includegraphics[scale=.4]{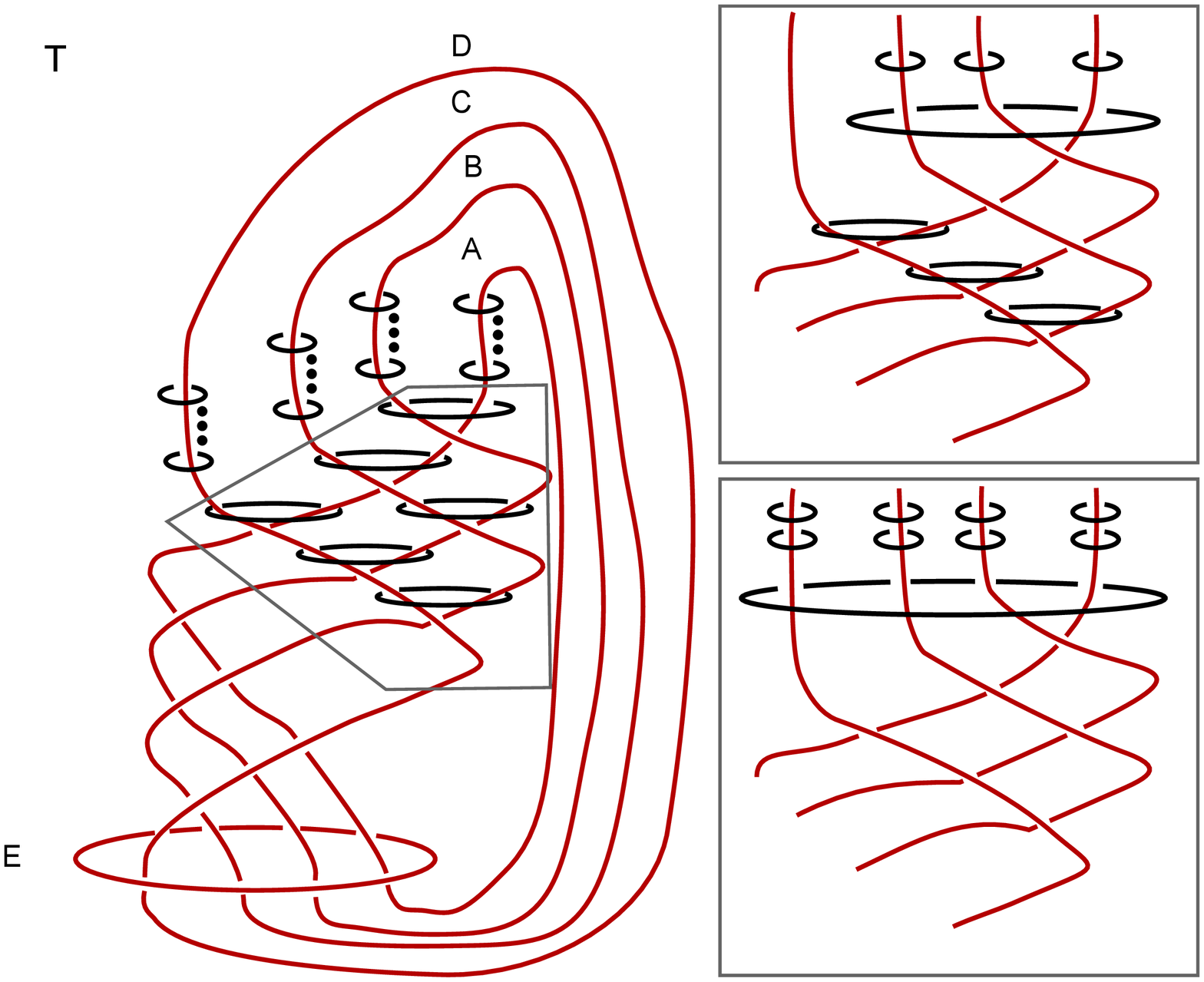}
\psfragscanoff
\label{fig:embedding4arms}
}
\;\;\;\subfloat[Complement of dual graph]{
\includegraphics[scale=.4]{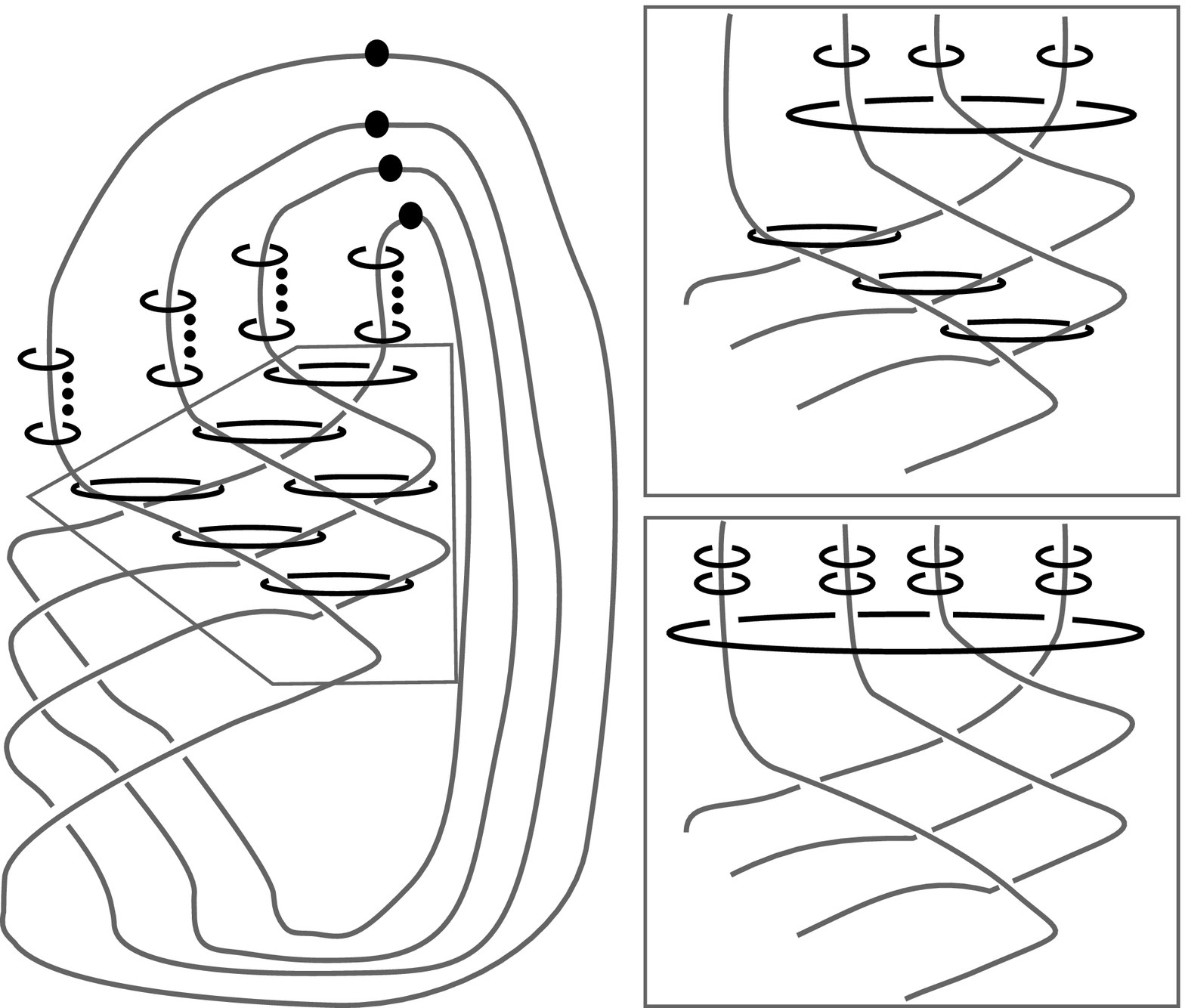}
\label{fig:Filling4arms}
}
\caption{Embeddings of the dual graph into $\C P^2\# M\overline{\C P^2}$ when $k=4$, and the complement of the embedding. Alternate embeddings (and their complements) can be obtained by replacing the boxed portion with one of the adjacent boxes. All framings on the black attaching circles are $-1$.}
\label{fig:EmbeddingFilling4arms}
\end{figure}

\begin{figure}
\psfragscanon
\psfrag{A}{$-n_1$}
\psfrag{B}{$-n_2$}
\psfrag{C}{$-n_3$}
\psfrag{D}{$-n_4$}
\psfrag{E}{$-n_5$}
\psfrag{F}{$+1$}
\psfrag{T}{5 $3$-h, 1 $4$-h}
\includegraphics[scale=.4]{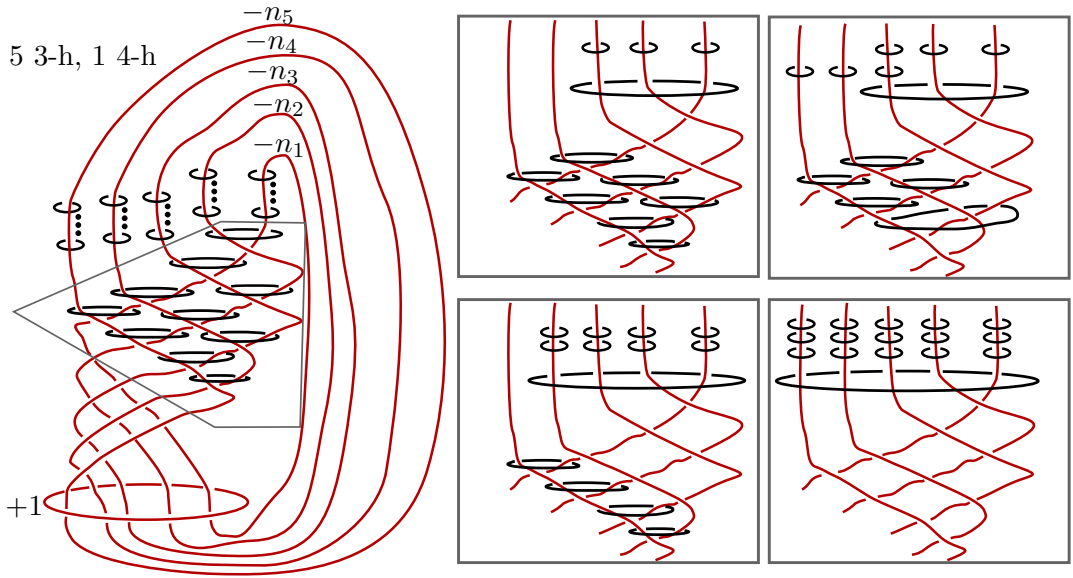}
\psfragscanoff
\caption{The embeddings of the dual graph when $k=5$. Alternate embeddings can be obtained by replacing the boxed portion with one of the adjacent boxes. All framings of the black attaching circles are $-1$.}
\label{fig:embedding5arms}
\end{figure}

We obtain the diffeomorphism type of the potentially convex symplectic filling by cutting out the dual configuration, and turning the manifold upside-down. Cutting out the dual configuration corresponds to deleting the $0$-handle, and the $2$-handles corresponding to the spheres in the dual graph (those whose attaching circles are shown in red). To turn the resulting manifold upside-down, we take the mirror image of the original link formed by all attaching circles (which we can do by reversing all the crossings), change all the framing coefficients to surgery coefficients with the opposite sign, attach 2-handles along $0$-framed meridians of the black surgery circles (since we cut out the other 2-handles with the dual graph), and replace the $k$ 3-handles by $k$ 1-handles. After a diffeomorphism of the boundary (performing all the blow-downs along the surgery curves with coefficients $\pm 1$), all of the crossings are reversed back to their original orientation (the lower half of the crossings are reversed when the surgery curve coming from the red $+1$ framed curve is blown down, and the upper half of the crossings are reversed when performing the blow-downs along the surgery curves corresponding to the exceptional spheres that intersect more than one sphere in the dual configuration), and the only 2-handles are $-1$ framed in the same position as in the embedding. Representing the $k$ $1$-handles by dotted circles we obtain diagrams like that in figure \ref{fig:Filling4arms}. Rotating the plane we are projecting the link onto by $90^{\circ}$ about the vertical axis, we obtain diagrams for the possible fillings given by figures \ref{fig:RotatedFilling4arms} and \ref{fig:RotatedFilling5arms}. 

\begin{figure}
\centering
\subfloat[$k=4$]{
\includegraphics[scale=.4]{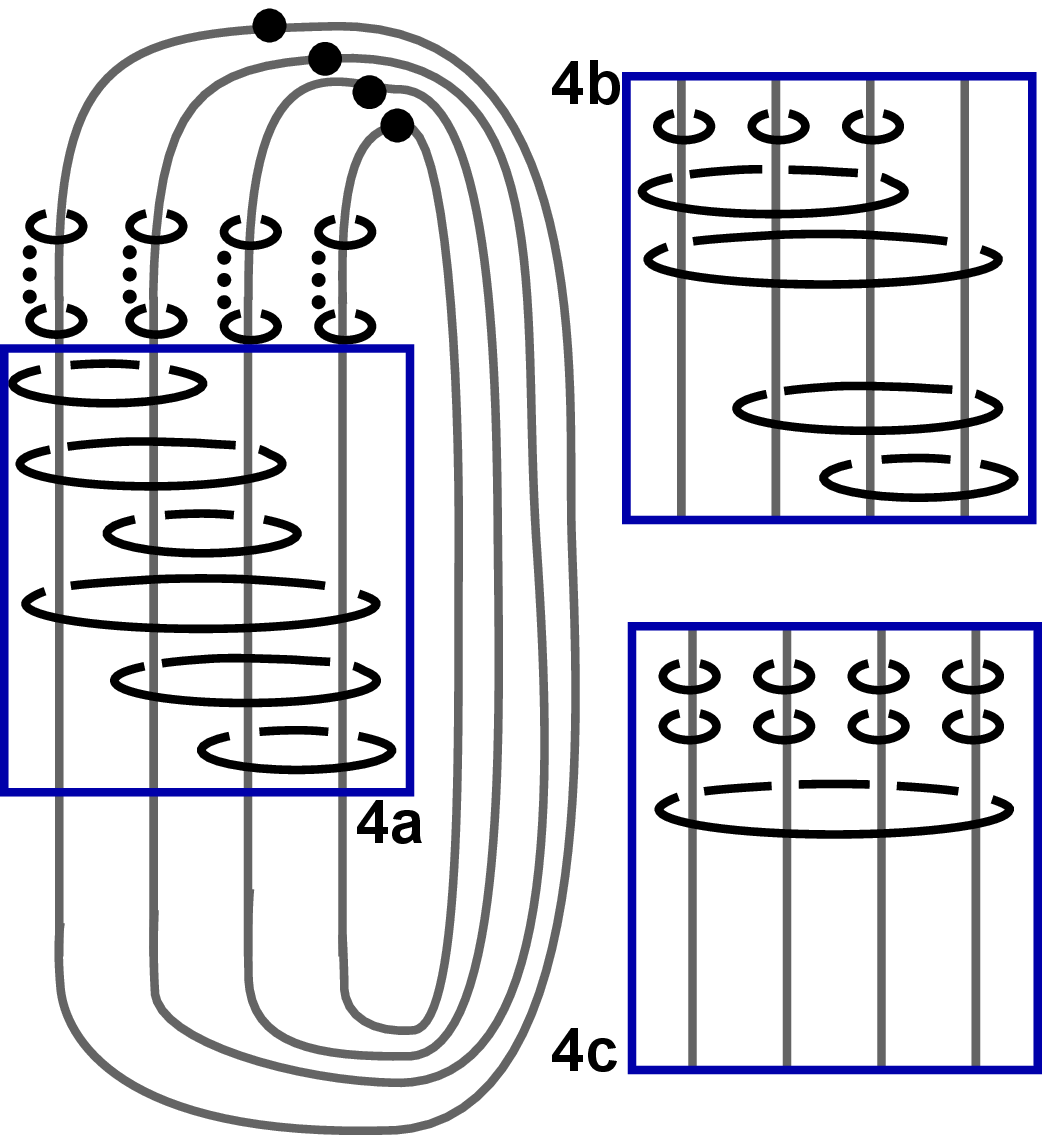}
\label{fig:RotatedFilling4arms}
}
\;\;\;\;\;\;\;\subfloat[$k=5$]{
\includegraphics[scale=.4]{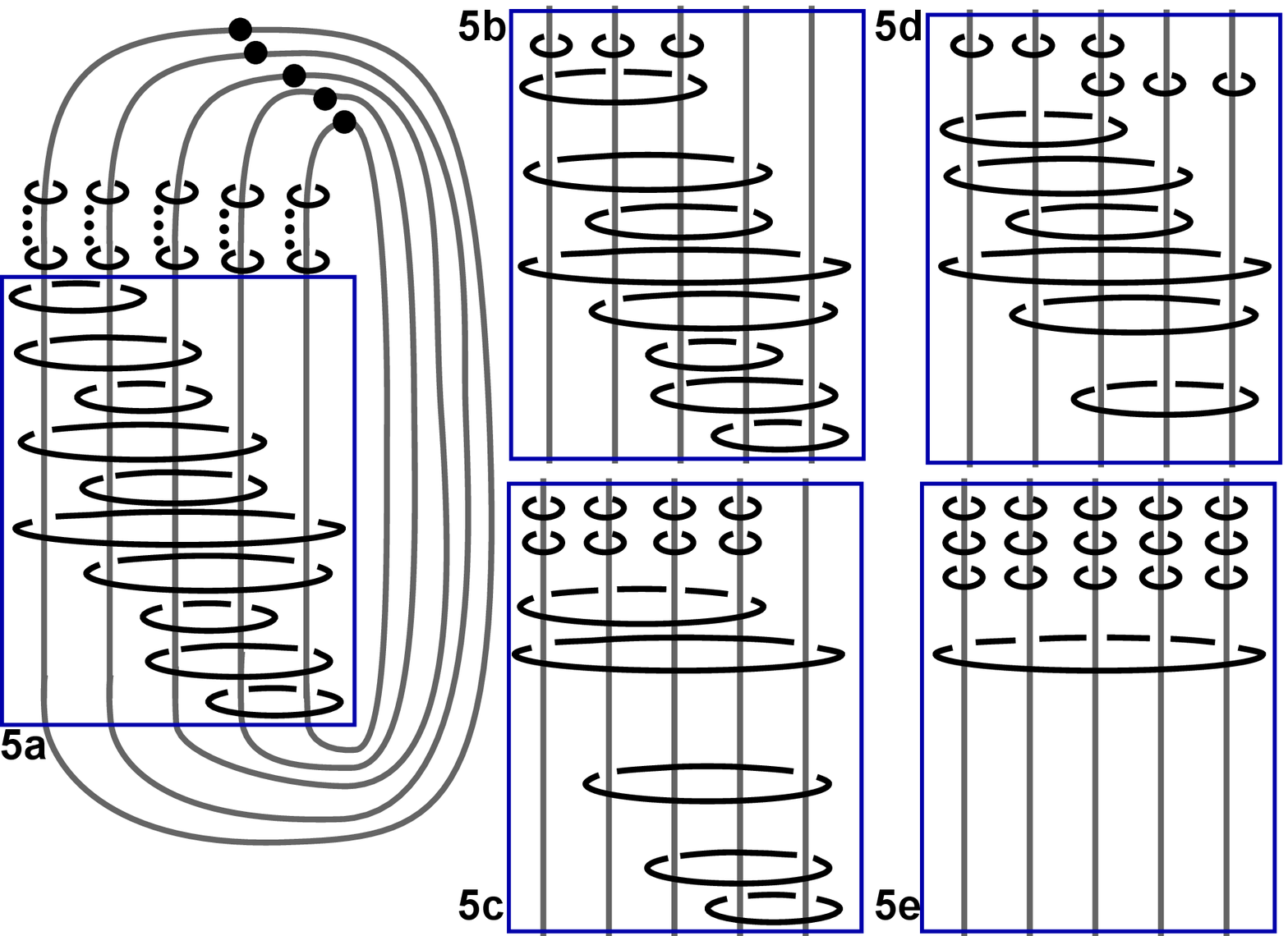}
\label{fig:RotatedFilling5arms}
}
\caption{The rotated diagrams giving all possible diffeomorphism types of these symplectic fillings. All 2-handles (shown in black) have framing coefficient $-1$.}
\label{fig:RotatedFillingkarms}
\end{figure}

As in our simplest examples in section \ref{simplestexamples}, these diagrams represent Lefschetz fibrations, now over $D^2$ with fibers $k$-holed disks. The vanishing cycles are given by the $-1$-framed $2$-handles (note that the Seifert framing agrees with the page framing in this diagram so this is indeed a positive Lefschetz fibration). Now we need to check that the open book decomposition induced on the boundary of the Lefschetz fibration supports the contact structure $\xi_{pl}$. 

\begin{prop} The open book decompositions arising as the boundary of these Lefschetz fibrations supports $\xi_{pl}$.\end{prop}

\begin{proof}
The open book decompositions on the boundaries of these Lefschetz fibrations have pages which are disks with $k$ holes, and the monodromy is given by positive Dehn twists around the vanishing cycles, ordered from the lowermost vanishing cycle in the figure to the uppermost. If we place the holes on the page on the vertices of a regular $k$-gon contained in the disk, we can isotope the attaching circles so that its projection onto a page has convex interior and contains some non-empty subset of the $k$-holes. We will denote a Dehn twist about a curve on this standard $k$-holed disk that convexly contains holes $\{h_1,\cdots , h_j\}$ by $D_{h_1,\cdots, h_j}$. We will label the holes corresponding to the dotted circles $1,\cdots, k$ from the left-most dotted circle to the right-most. Then the monodromy $M_{k\cdot }$ of the open books on the boundary of the above given Lefschetz fibrations in figure \ref{fig:RotatedFillingkarms} is given as follows.
\begin{eqnarray*}
M_{4a}&=& D_{3,4}D_{2,4}D_{1,4}D_{2,3}D_{1,3}D_{1,2}D_1^{n_1-2}D_2^{n_2-2}D_3^{n_3-2}D_4^{n_4-2}\\
M_{4b}&=& D_{3,4}D_{2,4}D_{1,4}D_{1,2,3}D_1^{n_1-1}D_2^{n_2-1}D_3^{n_3-1}D_4^{n_4-2}\\
M_{4c}&=& D_{1,2,3,4}D_1^{n_1}D_2^{n_2}D_3^{n_3}D_4^{n_4}\\
M_{5a}&=& D_{4,5}D_{3,5}D_{3,4}D_{2,5}D_{1,5}D_{2,4}D_{1,4}D_{2,3}D_{1,3}D_{1,2}D_1^{n_1-3}D_2^{n_2-3}D_3^{n_3-3}D_4^{n_4-3}D_5^{n_5-3}\\
M_{5b}&=&D_{4,5}D_{3,5}D_{3,4}D_{2,5}D_{1,5}D_{2,4}D_{1,4}D_{1,2,3}D_1^{n_1-2}D_2^{n_2-2}D_3^{n_3-2}D_4^{n_4-3}D_5^{n_5-3}\\
M_{5c}&=& D_{4,5}D_{3,5}D_{2,5}D_{1,5}D_{1,2,3,4}D_1^{n_1-1}D_2^{n_2-1}D_3^{n_3-1}D_4^{n_4-1}D_5^{n_5-3}\\
M_{5d}&=& D_{3,4,5}D_{2,5}D_{1,5}D_{2,4}D_{1,4}D_{1,2,3}D_1^{n_1-2}D_2^{n_2-2}D_3^{n_3-1}D_4^{n_4-2}D_5^{n_5-2}\\
M_{5e}&=& D_{1,2,3,4,5}D_1^{n_1}D_2^{n_2}D_3^{n_3}D_4^{n_4}D_5^{n_5}
\end{eqnarray*}

Theorem \ref{thmplumbingconvex} tells us that $\xi_{pl}$ is supported by an open book decomposition whose pages are $k$-holed disks, and whose monodromy consists of boundary parallel curves, $n_j$ of them about each hole, and one parallel to the outer boundary. Note that this is identical to the monodromy given by $M_{4c}$ and $M_{5e}$. We will show that the other elements are just other factorizations of the same monodromy.

Elements in the mapping class group of the disk with $k$ holes placed on the vertices of a regular $k$-gon inside the disk satisfy the lantern relations. These relations say that if $A,B,$ and $C$ are disjoint collections of holes then $D_AD_BD_CD_{A\cup B\cup C}=D_{A\cup B}D_{B\cup C}D_{A\cup C}$ where $A,B,C$ are subsets of the holes $\{1,\cdots, k\}$ (see \cite{MargalitMcCammond}). We first prove a lemma which gives us more general moves obtained from a sequence lantern relations. 

\begin{lemma}Suppose $B_0,B_1,\cdots , B_m$  are disjoint subsets of the $k$ holes on the disk ($m\geq 2$. Then 
$$D_{B_0}^{m-1}D_{B_1}\cdots D_{B_m} = D_{B_0\cup B_1}D_{B_0\cup B_2}\cdots D_{B_0\cup B_m}D_{B_1\cup \cdots \cup B_m}D_{B_0\cup B_1\cup \cdots \cup B_m}^{-1}.$$  \label{lanternmove}\end{lemma}
\begin{proof}
This relation can be obtained by performing $m-1$ lantern relations, and using the fact that Dehn twists around disjoint curves commute. In the notation above, for the $j^{th}$ lantern relation in this sequence, let $A=\{B_1,\cdots , B_j\}$, $B=B_0$, and $C=B_{j+1}$.
\end{proof}

Now we can use these moves to show that $M_{4a}=M_{4b}=M_{4c}$ and $M_{5a}=M_{5b}=M_{5c}=M_{5d}=M_{5e}$. A simple lantern relation shows that $M_{4a}=M_{4b}$. An application of the lemma with $B_0=4$, $B_1=3$, $B_2=2$, $B_3=1$ shows that $M_{4b}=M_{4c}$.

One lantern relation gives that $M_{5a}=M_{5b}$, and two lantern relations gives that $M_{5a}=M_{5d}$. We can commute Dehn twists about disjoint curves so we can commute $D_{3,4}$ past $D_{2,5}$ and $D_{1,5}$. Therefore,
$$M_{5b}=D_{4,5}D_{3,5}D_{2,5}D_{1,5}D_{3,4}D_{2,4}D_{1,4}D_{1,2,3}D_1^{n_1-2}D_2^{n_2-2}D_3^{n_3-2}D_4^{n_4-3}D_5^{n_5-3}$$
Then an application of the lemma with $B_0=4$, $B_1=3$, $B_2=2$, $B_3=1$ gives $M_{5b}=M_{5c}$. A final application of the lemma with $B_0=5$, $B_1=4$, $B_2=3$, $B_3=2$, $B_4=1$ implies $M_{5c}=M_{5d}$
\end{proof}

As in the examples in section \ref{simplestexamples}, we can stabilize this Lefschetz fibration to an allowable Lefschetz fibration, so by \cite{AkbulutOzbagci1},\cite{LoiPiergallini}, and \cite{Plamenevskaya} this manifold supports a Stein structure inducing a contact structure on its boundary that is supported by the boundary open book decomposition, which is $\xi_{pl}$ by the previous proposition.

Note that the diffeomorphism types of the fillings $4a$, $4b$, and $4c$ can be distinguished by Euler characteristic:
\begin{eqnarray*} \chi(4a)&=&(n_1-2+n_2-2+n_3-2+n_4-2+6)-(4)+(1)=n_1+n_2+n_3+n_4-5\\
\chi(4b)&=&n_1+n_2+n_3+n_4-4\\
\chi(4c)&=&n_1+n_2+n_3+n_4-2\\
\end{eqnarray*}
Similarly, we can distinguish $5a$, $5b$, $5c$, $5d$, and $5e$:
\begin{eqnarray*} \chi(5a)&=&n_1+n_2+n_3+n_4+n_5-9\\
\chi(5b)&=&n_1+n_2+n_3+n_4+n_5-8\\
\chi(5c)&=&n_1+n_2+n_3+n_4+n_5-6\\
\chi(5d)&=&n_1+n_2+n_3+n_4+n_5-7\\
\chi(5e)&=&n_1+n_2+n_3+n_4+n_5-3\\
\end{eqnarray*}

Note that the homology representation yielding the possible diffeomorphism type given by $4a$ can only occur when $n_1,n_2,n_3,n_4\geq 2$, $4b$ requires $n_4\geq 2$. In general since we are classifying fillings whose arms have length $n_j-1$, we have implicitly assumed $n_j\geq 1$. The restrictions for $5a$ are that $n_j\geq 3$ $\forall j$; for $5b$ we need $n_1,n_2,n_3\geq 2$ and $n_4,n_5\geq 3$; for we need $5c$ $n_5\geq 3$; for $5d$ we need $n_1,n_2,n_4,n_5\geq 2$.

This completes the classification of symplectic fillings of $(Y(-k-1; \frac{n_1}{n_1-1}, \cdots , \frac{n_k}{n_k-1}), \xi_{pl})$ when $k=4,5$, thus proving theorem \ref{thm:newfillings}. Note that $4c$ and $5e$ are diagrams for the original plumbing after handle-slides and cancellations. $4b$ and $5c$ are obtained from these plumbings by performing a rational blow-down of the central vertex (a sphere of square $-k-1$) together with $k-4$ spheres of square $-2$ in one of the arms. However, $4a$, $5a$, $5b$, and $5d$ are new diffeomorphism types. Note that the restrictions on the $n_j$ for these diffeomorphism types imply that their Euler characteristics are always strictly greater than $1$ so none of these are rational homology balls. Furthermore, these cannot be obtained by a rational blow-down of any subgraph containing the central vertex because by varying the $n_j$, we have classified all fillings of subgraphs containing the central vertex, so any such rational blow-down would give back one of our original diffeomorphism types ($4a,4b,4c,5a,5b,5c,5d,5e$). Each of these diffeomorphism types is distinguished from the others by Euler characteristic so there cannot be any non-trivial rational blow-downs of a subgraph containing the central vertex. Any subgraph not containing the central vertex is a union of linear subgraphs of spheres of square $-2$, which are known to have no rational blow-downs. Therefore these diffeomorphism types cannot be obtained by any rational blow-down of a subgraph of the original plumbing of spheres. Note, however that their Euler characteristics are still strictly small than that of the original plumbing.

}

\section{Further Observations and Questions}
{
Note that in many of the examples we have considered, each possible diffeomorphism type of a strong symplectic filling was actually realized by a Stein filling. This is not surprising, because each of the contact manifolds we consider are supported by planar open book decompositions (theorem \ref{thmplumbingconvex}), and a result of Wendl in \cite{Wendl} says that any minimal strong symplectic filling of a contact structure supported by a planar open book can be deformed through a homotopy of strong symplectic fillings to a Stein filling. Thus any diffeomorphism type supporting a convex symplectic structure also supports a Stein structure.

In all the examples we have considered, we have been able to prove that each possible diffeomorphism type allowed by the homological restrictions of section \ref{mainargument}, has a symplectic structure with convex boundary by direct construction. While it is valuable to understand explicit symplectic manifolds of each diffeomorphism type which can convexly fill a given contact manifold, we can only provide these constructions on a case by case basis. It would be useful to know abstractly that every diffeomorphism type produced by the arguments in section \ref{mainargument} are actually realized as strong symplectic fillings. One way to show this, would be to prove that any symplectic embedding of the dual configurations of spheres has a concave neighborhood. Since the blow-up construction gives a symplectic embedding of the dual configuration into a blow-up of $\C P^2$, if we could ensure this dual configuration has a concave neighborhood, the complement would necessarily have a symplectic structure with convex boundary. 

Additional complications in the opposite direction arise in the cases when $k$ is large and $e_0=-k-2,-k-1$, since we do not understand the possible embeddings of the dual graph into a blow up of $\C P^2$. While the space of complex projective lines in a complicated intersection configuration may be disconnected or empty, it is not clear what the space of smoothly embedded spheres in this intersection configuration looks like. It would be interesting to understand what information is generally implied about the classification of minimal strong symplectic fillings by the representations of the homology classes of the spheres in the dual graph into the second homology of a blow-up of $\C P^2$.

\begin{question}
Are there dually positive Seifert fibered spaces for which the number of diffeomorphism types of minimal strong symplectic fillings is strictly less than or strictly greater than the number of homology representations of the dual graph into the homology of a blow-up of $\C P^2$?
\end{question}

One phenomenon that is apparent from all the examples we have computed is that the original plumbing of spheres has the highest Euler characteristic of all the symplectic fillings of its boundary. Thus replacing these plumbings with alternate fillings could be useful in searching for symplectic manifolds with small Euler characteristic. One may ask if this phenomenon is true in greater generality.

\begin{question} Does a (dually positive or more general) symplectic plumbing of spheres have the highest Euler characteristic of any symplectic filling of its contact boundary? \end{question}

Finally, we have shown in theorem \ref{thm:finite} that dually positive Seifert fibered spaces with the contact structure $\xi_{pl}$ have finitely many strong symplectic fillings. Similar boundaries of other symplectic plumbings have been shown to admit infinitely many symplectic fillings (see in particular the work of Akhmedov and Ozbagci \cite{AkhmedovOzbagci}). However, these plumbings involve curves of higher genus at the central vertex. While those examples and the ones considered here provide some information about when symplectic plumbings can be replaced by only finitely many versus infinitely many symplectic fillings, there is still plenty of unknown ground. Since every negative definite graph (with each vertex decorated by the genus and self-intersection of the corresponding curve) corresponds to a symplectic plumbing of spheres with convex boundary inducing a contact structure $\xi_{pl}$ (by \cite{GayStipsicz2}), these contact boundaries are natural manifolds to start with.

\begin{question} Which manifolds, $(Y,\xi_{pl})$, arising as the boundary of a negative definite symplectic plumbing have only finitely many strong symplectic fillings?\end{question}

}

\bibliographystyle{plain}
\bibliography{bibliography}

\end{document}